\documentclass[11]{amsart}
\usepackage{amsfonts,tikz,amsmath,mathtools,amsthm,enumitem,mathrsfs,latexsym,amssymb,blkarray,afterpage,mathrsfs}
\usepackage[margin=1in, paperwidth=8.5in,paperheight=11in]{geometry}
\usepackage{mleftright}
\usepackage{mathtools}
\usepackage{xr-hyper}
\usepackage{hyperref}

\makeatletter
\def\@seccntformat#1{\@ifundefined{#1@cntformat}%
    {\csname the#1\endcsname\quad}
    {\csname #1@cntformat\endcsname}}
\newcommand{\section@cntformat}{\S\thesection\quad}
\newcommand{\subsection@cntformat}{\S\thesubsection\quad}
\makeatletter

\usepackage{graphicx}

\allowdisplaybreaks

\newlist{thmenum}{enumerate}{1}
\setlist[thmenum, 1]{label=(\roman*), ref=\thethm (\roman*)}

\begin{document}



\theoremstyle{plain}
\newtheorem{thm}{Theorem}[subsection]
\newtheorem{lem}[thm]{Lemma}
\newtheorem{cor}[thm]{Corollary}
\newtheorem{prop}[thm]{Proposition}
\newtheorem{conj}[thm]{Conjecture}
\newtheorem{exmp}[thm]{Example}

\theoremstyle{definition}
\newtheorem{defn}[thm]{Definition}
\newtheorem{exer}[thm]{Exercise}
\newtheorem{rmk}[thm]{Remark}
\newtheorem*{notation}{Notation}

\makeatletter
\@addtoreset{thm}{section}
\makeatother

\newcommand{\cc}{\mathbb{C}}
\newcommand{\rr}{\mathbb{R}}
\newcommand{\dd}{\mathbb{D}}
\newcommand{\cldd}{\overline{\mathbb{D}}}
\newcommand{\epsi}{\varepsilon}
\newcommand{\nn}{\mathbb{N}}
\newcommand{\zz}{\mathbb{Z}}
\newcommand{\fy}{\varphi}
\newcommand{\sign}{\text{sign}}
\newcommand{\bfs}{\textbf{S}}
\newcommand{\triv}{\textbf{1}}
\newcommand{\bb}{\textbf{B}}
\newcommand{\alga}{\mathcal{A}}
\newcommand{\hilb}{\mathcal{H}}
\newcommand{\inv}{\mathrm{GL}}
\newcommand{\nil}{\mathrm{Nil}}
\newcommand{\qnil}{\mathrm{QNil}}
\newcommand{\bh}{\mathcal{B(H)}}
\newcommand{\qh}{\mathcal{Q(H)}}
\newcommand{\ol}{\overline}
\newcommand{\mc}{\mathcal}
\newcommand{\dist}{\mathrm{dist}}
\newcommand{\nor}{\mathrm{Nor}}
\newcommand{\mm}{\mathbb{M}}
\newcommand{\au}{\sim_{au}}
\newcommand{\sorb}{\mathcal{S}}
\newcommand{\alg}{Alg}
\newcommand{\bqt}{\mathrm{BQT}}
\newcommand{\anti}{\mathrm{Anti}}
\newcommand{\rad}{Rad(\mc{A})}
\newcommand{\rank}{\mathrm{rank}}
\newcommand{\ran}{\mathrm{ran}}
\newcommand\scalemath[2]{\scalebox{#1}} 
\renewcommand{\qedsymbol}{$\blacksquare$}

\begin{abstract}
A subalgebra $\mc{A}$ of $\mm_n(\cc)$ is said to be \textit{idempotent compressible} if $E\mc{A}E$ is an algebra for all idempotents $E\in\mm_n(\cc)$. Likewise, $\mc{A}$ is said to be \textit{projection compressible} if $P\mc{A}P$ is an algebra for all orthogonal projections $P\in\mm_n(\cc)$. In this paper, a case-by-case analysis is used to classify the unital projection compressible subalgebras of $\mm_n(\cc)$, $n\geq 4$, up to transposition and unitary equivalence. It is observed that every algebra shown to admit the projection compression property is, in fact, idempotent compressible. We therefore extend the findings of \cite{CMR1} in the setting of $\mm_3(\cc)$, proving that the two notions of compressibility agree for all unital matrix algebras. 
\end{abstract}

\title{Matrix Algebras with a Certain Compression Property II}

\thanks{Research supported in part by NSERC (Canada)}
\author[Z. Cramer]{{Zachary Cramer}}

\newcommand{\Addresses}{{
  \bigskip
  \footnotesize

  Zachary~Cramer, \textsc{Faculty of Mathematics, University of Waterloo,
    Waterloo, Ontario, N2L 3G1}\par\nopagebreak
  \textit{E-mail address}: \texttt{zcramer@uwaterloo.ca}
}}

\date{\today}
\maketitle

\section{Introduction}\label{intro}

	Let $\mm_n=\mm_n(\cc)$ denote the algebra of $n\times n$ matrices with complex entries. In \cite{CMR1}, the notions of idempotent compressibility and projection compressibility were defined for subalgebras of $\mm_n$. In particular, a subalgebra $\mc{A}$ of $\mm_n$ was said to be \textit{idempotent compressible} if the corner $E\mc{A}E$ is an algebra for all idempotents $E\in\mm_n$. Analogously, $\mc{A}$ was said to be \textit{projection compressible} if the corner $P\mc{A}P$ is an algebra for all orthogonal projections $P\in\mm_n$. 
	
	It is immediate from the definitions that every idempotent compressible subalgebra of $\mm_n$ is also projection compressible, though the converse is much less clear. When $n\leq 2$, dimension considerations that every algebra is idempotent compressible---hence projection compressible---though this fact does not hold for $n\geq 3$. In \cite{CMR1}, however, it was shown that every unital subalgebra of $\mm_3$ with the projection compression property is in fact, idempotent compressible. Furthermore, a complete description of unital subalgebras of $\mm_3$ that admit these properties was obtained up to transposition and similarity \cite[Theorem~6.0.1]{CMR1}.
	
	The goal of this paper is to extend the results of \cite{CMR1} to higher dimensional settings. Specifically, we wish to obtain a classification of the unital subalgebras of $\mm_n$, $n\geq 4$, that admit the projection compression property, and investigate whether or not this notion agrees with that of idempotent compressibility.  
%

	Several subalgebras of $\mm_n$, $n\geq 4$, are known to exhibit the idempotent compression property. For example, if $\mc{A}$ is the intersection of a left ideal and a right ideal, then $\mc{A}$ is idempotent compressible \cite[Corollary~2.0.11]{CMR1}. Algebras of this form are known as $\mc{LR}$\textit{-algebras}, and are exactly the algebras of the form $\mc{A}=P\mm_nQ$ for some projections $P$ and $Q$ in $\mm_n$ \cite[Corollary 2.0.10]{CMR1}. 
	
	The following example showcases three additional collections of algebras that exhibit the idempotent compression property. 
	\begin{exmp}\cite[Examples 3.1.1, 3.1.3, 3.1.6]{CMR1}\label{exmp:families of compressible algebras} Let $n\geq 4$ be an integer, and let $Q_1$, $Q_2$, and $Q_3$ be projections in $\mm_n$ that sum to $I$. In what follows, all matrices are expressed with respect to the decomposition $\cc^n=\ran(Q_1)\oplus\ran(Q_2)\oplus\ran(Q_3)$.
		\begin{thmenum}
			\item The algebra $$\begin{array}{l}\mathcal{A}=\cc Q_1+\cc Q_3+(Q_1+Q_2)\mm_n(Q_2+Q_3)\vspace{0.2cm}\\
			\left.\right.\hspace{0.27cm}=\left\{\begin{bmatrix}
			\alpha I & M_{12} & M_{13}\\
			0 & M_{22} & M_{23}\\
			0 & 0 & \beta I
			\end{bmatrix}:\alpha,\beta\in\cc,M_{ij}\in Q_i\mm_nQ_j\right\}\end{array}$$
	is idempotent compressible. \label{exmp:families of compressible algebras:1}\bigskip
			
			\item If $\rank (Q_1)=\rank (Q_2)=1$, then the algebra $$\begin{array}{l}\mathcal{A}=\cc Q_1+\cc Q_2+\cc Q_3+(Q_1+Q_2)\mm_nQ_3\vspace{0.2cm}\\
			\left.\right.\hspace{0.27cm}=\left\{\begin{bmatrix}
			\alpha & 0 & M_{13}\\
			0 & \beta & M_{23}\\
			0 & 0 & \gamma I 
			\end{bmatrix}:\alpha,\beta,\gamma\in\cc,M_{ij}\in Q_i\mm_nQ_j\right\}\end{array}$$ is idempotent compressible.\label{exmp:families of compressible algebras:2}\smallskip
			
			\item If $\rank (Q_1)=\rank (Q_2)=1$, then the algebra $$\begin{array}{l}\mathcal{A}=\cc(Q_1+Q_2)+Q_1\mm_nQ_2+(Q_1+Q_2)\mm_{n}Q_3+\cc Q_3\vspace{0.2cm}\\
			\left.\right.\hspace{0.27cm}=\left\{\begin{bmatrix}
			\alpha & x & M_{13}\\
			0 & \alpha & M_{23}\\
			0 & 0 & \beta I
			\end{bmatrix}:\alpha,\beta,x\in\cc,M_{ij}\in Q_i\mm_nQ_j\right\}\end{array}$$
		is idempotent compressible.\smallskip\label{exmp:families of compressible algebras:3}
		\end{thmenum}
		
%
%
%
	\end{exmp}	
	Our main result, Theorem~\ref{main result}, states that for every integer $n\geq 4$, the algebras from Example~\ref{exmp:families of compressible algebras}, together with the unitization of $\mc{LR}$-algebras described above, form an exhaustive list of unital projection compressible subalgebras of $\mm_n$ up to transposition and similarity. Since each algebra in this collection is known to be idempotent compressible, it will follow that a unital matrix algebra is projection compressible if and only if it is idempotent compressible.
	
	As in \cite{CMR1}, a case-by-case analysis will be used to obtain the classification of unital projection compressible algebras described above. The requisite results from \cite{LMMRWedderburn} concerning the structure theory for matrix algebras will be reintroduced in \S2. In \S3, we present a necessary condition for projection compressibility (Theorem~\ref{At most one non-scalar corner theorem}) that imposes significant restrictions on the structure of a projection compressible algebra. As we shall see, the algebras that satisfy this condition can be grouped into three distinct types determined by their block upper triangular forms. The unital projection compressible algebras of each type will be classified up to transposition and similarity in sections \S4-6, and ultimately up to transposition and unitary equivalence in \S7.\\

\section[2]{Preliminaries}

We will begin by reintroducing the notation, definitions, and preliminary results from \cite{CMR1} surrounding idempotent and projection compressibility. Additionally, we will present some of the key results from \cite{LMMRWedderburn} concerning the structure theory for matrix algebras.

\begin{notation}
		Given vectors $x,y\in\cc^n$, define $x\otimes y^*:\cc^n\rightarrow \cc^n$ to be the rank-one operator $z\mapsto\langle z,y\rangle x.$\smallskip
	\end{notation}

Observe that if $\mc{A}$ is a subalgebra of $\mm_n$ and $E\in\mm_n$ is an idempotent, then $E\mc{A}E$ is always a linear space. Thus, $E\mc{A}E$ is an algebra if and only if it is multiplicatively closed. By dimension considerations, $E\mc{A}E$ be an algebra for all idempotents $E$ of rank $1$. 

\begin{defn}\cite[Definition 2.0.2]{CMR1}
Given a subset $\mc{A}$ of $\mm_n$, we define the \textit{transpose} and \textit{anti-transpose} of $\mc{A}$ to be $$\begin{array}{ccc}\mc{A}^T\coloneqq \left\{A^T:A\in\mc{A}\right\} & \text{and} & \mc{A}^{aT}\coloneqq \left\{JA^TJ:A\in\mc{A}\right\},\end{array}$$ respectively, where $J$ denotes the anti-diagonal unitary matrix whose $(i,j)$-entry is $\delta_{j,n-i+1}$. We say that two subalgebras $\mc{A}$ and $\mc{B}$ of $\mm_n$ are \textit{transpose similar} (resp. \textit{transpose equivalent)} if $\mc{B}$ is similar (resp. unitarily equivalent) to $\mc{A}$ or $\mc{A}^T$. 
\end{defn}

Since the set of idempotents in $\mm_n$ is closed under transpose similarity, so too is the set of all idempotent compressible subalgebras of $\mm_n$. Likewise, the set of projection compressible subalgebras of $\mm_n$ is closed under transpose equivalence. From this it follows that for a given algebra $\mc{A}$, either $\mc{A}$, $\mc{A}^T$, and $\mc{A}^{aT}$ are all idempotent (resp. projection) compressible, or none of them are. 

Finally, if an algebra $\mc{A}$ is idempotent (resp. projection) compressible, then so too is its unitization $\widetilde{A}=\mc{A}+\cc I$ \cite[Proposition~2.0.6]{CMR1}. The converse, however, is false.

The classification of unital projection compressible subalgebras of $\mm_n$, $n\geq 4$, will require much of structure theory for matrix algebras applied in the analysis from \cite{CMR1}. Thus, it will be important to recall the following.
	\begin{defn}\cite[Definition 9]{LMMRWedderburn}\label{definition of reduced block upper triangular}
A subalgebra $\mathcal{A}$ of $\mm_n$ is said to have a \textit{reduced block upper triangular form} with respect to a decomposition $\cc^n=\mathcal{V}_1\dotplus \mathcal{V}_2\dotplus\cdots\dotplus\mathcal{V}_m$ if 

\begin{itemize}

	\item[(i)]when expressed as a matrix, every $A$ in $\mathcal{A}$ has the form  
$$A=\begin{bmatrix}
A_{11} & A_{12} & A_{13} & \cdots & A_{1m}\\
0 & A_{22} & A_{23} & \cdots & A_{2m}\\
0 & 0 & A_{33} & \cdots & A_{3m}\\
\vdots & \vdots & \vdots & \ddots & \vdots\\
0 & 0 & 0 & \cdots & A_{mm}
\end{bmatrix} $$
with respect to this decomposition, and  \\

	\item[(ii)]for each $i$, the subalgebra $\mathcal{A}_{ii}\coloneqq \{A_{ii}:A\in\mathcal{A}\}$ is irreducible. That is, either $\mathcal{A}_{ii}=\{0\}$ and $\dim\mathcal{V}_i=1$, or $\mathcal{A}_{ii}=\mm_{\dim\mathcal{V}_i}$.

\end{itemize}

\end{defn}

An application of Burnside's Theorem \cite{Burnside} shows that every algebra $\mc{A}$ admits a reduced block upper triangular form with respect to some orthogonal decomposition of $\cc^n$. Moreover, one may verify that  if $S$ is an invertible matrix that is block upper triangular with respect to the same decomposition as that of $\mc{A}$, then $S^{-1}\mc{A}S$ also has a reduced block upper triangular form with respect to this decomposition. 

\begin{thm}\cite[Corollary 14]{LMMRWedderburn}\label{Cor 14 from LMMR}
If a subalgebra $\mathcal{A}$ of $\mm_n$ has a reduced block upper triangular form with respect to a decomposition $\cc^n=\mathcal{V}_1\dotplus\mathcal{V}_2\dotplus\cdots\dotplus\mathcal{V}_m$, then the set $\{1,2,\ldots, m\}$ can be partitioned into disjoint subsets $\Gamma_1,\Gamma_2,\ldots, \Gamma_k$ such that   

\begin{itemize}

	\item[(i)] If $i\in\Gamma_s$ and $\mc{A}_{ii}\neq\{0\}$, then there exists $G^{<i>}$ in $\mathcal{A}$ such that $G_{jj}^{<i>}=I_{\mathcal{V}_j}$ for all $j\in\Gamma_s$, and $G_{jj}^{<i>}=0$ for all $j\notin\Gamma_s$.  \\
	
	\item[(ii)]If $i$ and $j$ belong to the same $\Gamma_s$, then $\dim\mathcal{V}_i=\dim\mathcal{V}_j$, and there is an invertible linear map 
	$S_{ij}:\mathcal{V}_i\rightarrow\mathcal{V}_j$ such that  
	$$A_{ii}=S_{ij}^{-1}A_{jj}S_{ij} $$
	for all $A\in\mathcal{A}$.\\

	\item[(iii)]If $i$ and $j$ do not belong to the same $\Gamma_s$, then  
	$$\left\{(A_{ii},A_{jj}):A\in\mathcal{A}\right\}=\left\{A_{ii}:A\in\mathcal{A}\right\}\times\left\{A_{jj}:A\in\mathcal{A}\right\}.\smallskip$$

\end{itemize}
\end{thm}

Given an algebra $\mathcal{A}$ of the form described in Theorem~\ref{Cor 14 from LMMR}, we say that indices $i$ and $j$ in $\{1,2,\ldots, m\}$ are \textit{linked} if they belong to the same $\Gamma_s$, and \textit{unlinked} otherwise. It is easy to see that if $S$ is an invertible matrix that is block upper triangular with respect to the same decomposition as that of $\mc{A}$, then two indices are linked in $\mc{A}$ if and only if they are linked in $S^{-1}\mc{A}S$.

By \cite[Corollary 28]{LMMRWedderburn}, every subalgebra $\mc{A}$ of $\mm_n$ decomposes as an algebraic direct sum $$\mc{A}=\mc{S}\dotplus\rad,$$ where $\mc{S}$ is a semi-simple subalgebra of $\mc{A}$ and $\rad$ denotes the nil radical of $\mc{A}$. When $\mc{A}$ is in reduced block upper triangular form with respect to some decomposition of $\cc^n$, $\rad$ consists of all elements of $\mc{A}$ that are strictly block upper triangular \cite[Proposition~19]{LMMRWedderburn}.

Theorem~\ref{every algebra is similar to an unhinged algebra} provides a description of the structure of $\mc{S}$ that will be used frequently throughout the classification to come. First, we will require the following definition.

\begin{defn}
Let $\mc{A}$ be a subalgebra of $\mm_n$ in reduced block upper triangular form with respect to some decomposition of $\cc^n$. For $A\in\mc{A}$, define the \textit{block-diagonal} of $A$ to be the matrix $BD(A)$ obtained by replacing the block `off-diagonal' entries of $A$ with zeros. Furthermore, define the \textit{block-diagonal} of $\mc{A}$ to be the algebra 
$$BD(\mc{A})=\left\{BD(A):A\in\mc{A}\right\}.$$
\end{defn}

\begin{thm}\cite[Corollary 30]{LMMRWedderburn}\label{every algebra is similar to an unhinged algebra} If a subalgebra $\mathcal{A}$ of $\mm_n$ has a reduced block upper triangular form with respect to a decomposition of $\cc^n$, then there exists an invertible linear operator $S$ that is block upper triangular with respect to the same decomposition as that of $\mc{A}$, and $S^{-1}\mathcal{A}S$ has an unhinged reduced block upper triangular form with respect to this decomposition.
\end{thm}

If an algebra $\mc{A}$ in reduced block upper triangular form with respect to some decomposition of $\cc^n$ is such that $\mc{A}=BD(\mc{A})\dotplus\rad$, we say that $\mc{A}$ is \textit{unhinged} with respect to this decomposition.

We emphasize that the transformation of an algebra $\mc{A}$ into an unhinged reduced block upper triangular form described in Theorem~\ref{every algebra is similar to an unhinged algebra} can be achieved via application of a block upper triangular similarity, but not, in general, via unitary equivalence. Additionally, we note that if $\mc{A}$ is in reduced block upper triangular form and $BD(\mc{A})=\cc I$, then Theorem~\ref{every algebra is similar to an unhinged algebra} implies that $\mc{A}=\cc I\dotplus\rad$. Thus, $\mc{A}$ is unhinged with respect to any decomposition in which it admits a reduced block upper triangular form.

\section[3]{A Strategy for Classification} In this section we will develop a strategy for characterizing the unital subalgebra of $\mm_n$ that admit the projection compression property. By the comments preceding Theorem~\ref{Cor 14 from LMMR}, we may assume that all algebras under consideration are expressed in reduced block upper triangular form with respect to some orthogonal direct sum decomposition of $\cc^n$.

We begin by presenting a simple structural requirement for a unital subalgebra of $\mm_n$, $n\geq 4$, to admit the projection compression property. This result and its corollaries impose substantial restrictions on the reduced block upper triangular form of a projection compressible algebra.

\begin{thm}\label{At most one non-scalar corner theorem}

Let $n\geq 4$ be an integer, and let $\mc{A}$ be a projection compressible subalgebra of $\mm_n$. Suppose there exist mutually orthogonal projections $P_1$ and $P_2$ in $\mm_n$ such that $\min(\rank(P_1),\rank(P_2))\geq 2$ and $P_2\mc{A}P_1=\{0\}$. Then $P_1\mc{A}P_1=\cc P_1$ or $P_2\mc{A}P_2=\cc P_2$.

\end{thm}

\begin{proof}

First assume that $\rank(P_1)=\rank(P_2)=2$. By replacing $\mc{A}$ with the compression $(P_1+P_2)\mc{A}(P_1+P_2)$ if necessary, we may also assume that $P_1+P_2=I$.

Arguing by contradiction, suppose that $P_1\mc{A}P_1\neq\cc P_1$ and $P_2\mc{A}P_2\neq\cc P_2$. There then exists an operator $A\in\mc{A}$ such that $P_iAP_i\notin\cc P_i$ for each $i\in\{1,2\}$. Indeed, choose operators $A_1,A_2\in\mc{A}$ such that $P_1A_1P_1\notin\cc P_1$ and $P_2A_2P_2\notin\cc P_2$. If $P_2A_1P_2\notin \cc P_2$ or $P_1A_2P_1\notin \cc P_1$, then $A_1$ or $A_2$ will satisfy the above requirements. Otherwise, $A\coloneqq A_1+A_2$ will suffice.

Thus, assume that $A\in\mc{A}$ has been chosen such that $P_1AP_1\notin\cc P_1$ and $P_2AP_2\notin\cc P_2$. For each $i\in\{1,2\}$, choose an orthonormal basis $\left\{e_1^{(i)},e_2^{(i)}\right\}$ for $\ran(P_i)$ such that $P_iAP_i$ is not diagonal with respect to $\mc{B}=\left\{e_1^{(1)},e_2^{(1)},e_1^{(2)},e_2^{(2)}\right\}$. By permuting the basis vectors if necessary, we may assume that $\langle Ae_2^{(i)},e_1^{(i)}\rangle\neq 0$ for each $i\in\{1,2\}$.

Consider the matrix
$$Q\coloneqq \begin{bmatrix}
1 & 0 & 0 & 1\\
0 & 2 & 0 & 0\\
0 & 0 & 2 & 0\\
1 & 0 & 0 & 1
\end{bmatrix}\smallskip$$
written with respect to $\mc{B}$. It is straightforward to check that $\frac{1}{2}Q$ is a projection in $\mm_4$ and every $B\in Q\mc{A}Q$ satisfies $\langle Be_2^{(1)},e_1^{(2)}\rangle=0.$
With $A$ as above, however,
$$\langle (QAQ)^2e_2^{(1)},e_1^{(2)}\rangle=8\langle Ae_2^{(1)},e_1^{(1)}\rangle\langle Ae_2^{(2)},e_1^{(2)}\rangle\neq 0.\smallskip$$
Thus, $(QAQ)^2$ does not belong to $Q\mc{A}Q$, so $Q\mc{A}Q$ is not an algebra. This contradicts the assumption that $\mc{A}$ is projection compressible.

Now consider the general case in which each $P_i$ has rank at least $2$. One may deduce from the above analysis that for some $i\in\{1,2\}$, every rank-two subprojection $P\leq P_i$ is such that $P\mc{A}P=\cc P$. It then follows that $P_i\mc{A}P_i=\cc P_i$, as required.
\end{proof}

As we shall see in the coming analysis, Theorem~\ref{At most one non-scalar corner theorem} has significant implications for the classification of projection compressible algebras. Additionally, it highlights a major difference between the classification in this setting and that of $\mm_3$. Since $\mm_3$ cannot contain projections $P_1$ and $P_2$ as described in Theorem~\ref{At most one non-scalar corner theorem}, this result may help to explain why there exist certain projection compressible subalgebras of $\mm_3$ that do not admit analogues in higher dimensions (see \cite[Examples 3.2.1, 3.2.4, 3.2.7]{CMR1}).

The following corollaries to Theorem~\ref{At most one non-scalar corner theorem} provide a more explicit description of the reduced block upper triangular forms that can exist for a unital projection compressible algebra.

\begin{cor}\label{unique integer k corollary}

Let $n\geq 4$ be an integer, and let $\mc{A}$ be a unital subalgebra of $\mm_n$. Suppose that there is an orthogonal decomposition $\bigoplus_{i=1}^m\mc{V}_i$ of $\cc^n$ with respect to which 
\begin{itemize}
	\item[(i)]$\mc{A}$ is reduced block upper triangular, and\smallskip
	
	\item[(ii)]there is an index $k\in\{1,2,\ldots,m\}$ such that if $Q_1$, $Q_2$, and $Q_3$ denote the orthogonal projections onto $\bigoplus_{i<k}\mc{V}_i$, $\mc{V}_k$, and $\bigoplus_{i> k}\mc{V}_i$, respectively, then $$\begin{array}{ccc}(Q_1+Q_2)\mc{A}(Q_1+Q_2)\neq\cc (Q_1+Q_2) & \text{and} & (Q_2+Q_3)\mc{A}(Q_2+Q_3)\neq \cc (Q_2+Q_3). \end{array}\smallskip$$

\end{itemize}
If $\mc{A}$ is projection compressible, then $k$ is unique. When this is the case, $Q_1\mc{A}Q_1=\cc Q_1$ and $Q_3\mc{A}Q_3=\cc Q_3$.

\end{cor}

\begin{proof}

Assume that $\mc{A}$ is projection compressible. Suppose to the contrary that there were a second index $k^\prime$ together with corresponding projections $Q_1^\prime$, $Q_2^\prime$, and $Q_3^\prime$ such that 
$$\begin{array}{ccc}(Q_1^\prime+Q_2^\prime)\mc{A}(Q_1^\prime+Q_2^\prime)\neq\cc (Q_1^\prime+Q_2^\prime) & \text{and}
&(Q_2^\prime+Q_3^\prime)\mc{A}(Q_2^\prime+Q_3^\prime)\neq\cc (Q_2^\prime+Q_3^\prime).\end{array}\smallskip$$
Assume without loss of generality that $k<k^\prime$. The projections $P_1\coloneqq Q_1+Q_2$ and $P_2\coloneqq Q_2^\prime+Q_3^\prime$ then satisfy the hypotheses of Theorem~\ref{At most one non-scalar corner theorem}, so $P_1\mc{A}P_1=\cc P_1$ or $P_2\mc{A}P_2=\cc P_2$. This is a contradiction.

The final claim follows immediately from the uniqueness of $k$. Indeed, if $Q_1\mc{A}Q_1\neq \cc Q_1$, then $k-1$ would be another such index. If instead $Q_3\mc{A}Q_3\neq \cc Q_3$, then one could derive a similar contradiction by considering the index $k+1$. 
\end{proof}
\smallskip

The following special case of Corollary~\ref{unique integer k corollary} describes the situation for algebras whose block-diagonal contains a block of size at least $2$.

\begin{cor}\label{special case of uniqueness corollary}

Let $n\geq 4$ be an integer, and let $\mc{A}$ be a unital subalgebra of $\mm_n$. Suppose that there is a decomposition $\bigoplus_{i=1}^m\mc{V}_i$ of $\cc^n$ with respect to which

\begin{itemize}

	\item[(i)]$\mc{A}$ is reduced block upper triangular, and\smallskip
	
	\item[(ii)]there is an index $k\in\{1,2,\ldots,m\}$ such that $\dim\mc{V}_k\geq 2$.
	
	\end{itemize}
If $\mc{A}$ is projection compressible, then $k$ is unique. When this is the case, if $Q_1$, $Q_2$, and $Q_3$ denote the orthogonal projections onto $\bigoplus_{i<k}\mc{V}_i$, $\mc{V}_k$, and $\bigoplus_{i>k}\mc{V}_i$, respectively, then $$\begin{array}{cccc}Q_1\mc{A}Q_1=\cc Q_1, & Q_2\mc{A}Q_2=Q_2\mm_nQ_2, & \text{and} & Q_3\mc{A}Q_3=\cc Q_3.\end{array}\smallskip$$

\end{cor}

%

The results presented above provide a strategy for classifying the unital subalgebras of $\mm_n$ that exhibit the projection compression property. Indeed, we may use Corollaries \ref{unique integer k corollary} and \ref{special case of uniqueness corollary} to partition the unital subalgebras of $\mm_n$ into the following three distinct types determined by their reduced block upper triangular forms:\smallskip

\noindent \textit{Type I:} $\mc{A}$ has a reduced block upper triangular form with respect to an orthogonal decomposition of $\cc^n$ such that there does \textit{not} exist an index $k$ as in Corollary~\ref{unique integer k corollary};\\

\noindent \textit{Type II:} $\mc{A}$ has a reduced block upper triangular form with respect to an orthogonal decomposition of $\cc^n$ such that $BD(\mc{A})$ contains a block of size at least $2$ (i.e., there is an integer $k$ as in Corollary~\ref{special case of uniqueness corollary}).\\

\noindent \textit{Type III:} For each orthogonal decomposition of $\cc^n$ with respect to which $\mc{A}$ is reduced block upper triangular, every block in $BD(\mc{A})$ is $1\times 1$, and there is an integer $k$ as in Corollary~\ref{unique integer k corollary}.\\

The unital projection compressible algebras of type I, II, and III will be analysed in \S4, \S5, and \S6, respectively. In each case, a classification of these algebras will be obtained up to transpose similarity by examining the structure their semi-simple and radical parts.\smallskip

	
	\section[4]{Algebras of Type I}

In what follows, the term \textit{type I} will be used to describe a unital subalgebra $\mc{A}$ of $\mm_n$, $n\geq 4$, that has a reduced block upper triangular form with respect to an orthogonal decomposition $\bigoplus_{i=1}^m\mc{V}_i$ of $\cc^n$, such that there does not exist an integer $k$ as in Corollary~\ref{unique integer k corollary}. If $\mc{A}$ is such an algebra, then it must be the case that $\dim\mc{V}_i=1$ for all $i$ (i.e., $m=n$). For instance, the algebra from Example~\ref{exmp:families of compressible algebras:1} is of type I if and only if $Q_2=0$; or $\rank(Q_2)=1$ and $Q_i=0$ for some $i\in\{1,3\}$.

The goal of this section is to determine which type I algebras possess the projection compression property. As we shall see, the type I algebras satisfying this condition are either unitizations of $\mc{LR}$-algebras, or unitarily equivalent to the type I algebra from Example~\ref{exmp:families of compressible algebras:1}. In order to demonstrate this systematically, it will be useful to keep a record of the orthogonal decompositions of $\cc^n$ with respect to which $\mc{A}$ satisfies the definition of type I.

\begin{defn}\label{definition of F(A)}
If $\mc{A}$ is an algebra of type I, let $\mc{F}_I=\mc{F}_I\mc{(A)}$ denote the set of pairs $\Omega=(d,\bigoplus_{i=1}^n\mc{V}_i)$, where 

\begin{itemize}
		\item[(i)] $\bigoplus_{i=1}^n\mc{V}_i$ is an orthogonal decomposition of $\cc^n$ with respect to which $\mc{A}$ is reduced block upper triangular, and 
		
		\item[(ii)]$d$ is an integer in $\{1,2,\ldots,n\}$ such that if $Q_{1\Omega}$ denotes the orthogonal projection onto $\bigoplus_{i=1}^{d}\mc{V}_i$, and $Q_{2\Omega}$ denotes its complement $I-Q_{1\Omega}$, then $$\begin{array}{ccc}Q_{1\Omega}\mc{A}Q_{1\Omega}=\cc Q_{1\Omega} & \text{and} & Q_{2\Omega}\mc{A}Q_{2\Omega}=\cc Q_{2\Omega}.\end{array}\smallskip$$
		
	\end{itemize}

	\end{defn}


\begin{notation}
\upshape{
If $\mc{A}$ is a type I algebra and $\Omega=(d,\bigoplus_{i=1}^n\mc{V}_i)$ is a pair in $\mc{F}_I\mc{(A)}$, the notation $n_{1\Omega}=d$ and $n_{2\Omega}=n-d$ will be used to refer to the ranks of $Q_{1\Omega}$ and $Q_{2\Omega}$, respectively.
}
\end{notation}

Suppose that $\mc{A}$ is a projection compressible algebra of type I and $\Omega$ is a pair in $\mc{F}_I\mc{(A)}$. In the language of \S2, each corner $Q_{i\Omega}\mc{A}Q_{i\Omega}=\cc Q_{i\Omega}$ is a diagonal algebra comprised of mutually linked $1\times 1$ blocks. Note that the blocks in $Q_{1\Omega}\mc{A}Q_{1\Omega}$ may or may not be linked to those in $Q_{2\Omega}\mc{A}Q_{2\Omega}$. If these blocks are linked, we will say that the projections $Q_{1\Omega}$ and $Q_{2\Omega}$ are \textit{linked}. Otherwise, we will say that $Q_{1\Omega}$ and $Q_{2\Omega}$ are \textit{unlinked}. Note that the projections $Q_{1\Omega}$ and $Q_{2\Omega}$ are linked for some pair in $\Omega\in\mc{F}_I(\mc{A})$ if and only if they are linked for every pair in $\mc{F}_I(\mc{A})$. 

It will be important to distinguish between the type I algebras whose projections are linked and those whose projections are unlinked. The projection compressible type I algebras with unlinked projections will be classified in $\S\ref{Subsection: Type I algebras unlinked}$, while those with linked projections will be classified in $\S\ref{Subsection: Type I algebras linked}$.  Before our analysis splits, however, let us examine one extreme case that will be relevant to the classification in either setting. 

Observe that if $\mc{A}$ is an algebra of type I and $\mc{F}_I(\mc{A})$ contains a pair $\Omega=(d,\bigoplus_{i=1}^n\mc{V}_i)$ with $d=n$, then $\mc{A}=\cc I$, and hence $\mc{A}$ is idempotent compressible. If instead $d=1$ or $d=n-1$, then Proposition~\ref{prop k0=2 or k0=n} indicates that $\mc{A}$ is the unitization of an $\mc{LR}$-algebra. 
The proof of this result relies on the following structure theorem for $\mm_n$-modules, which will be applied frequently throughout our analysis. For reference, see \linebreak \cite[Theorem 3.3]{Lam}.

\begin{thm}\label{structure of modules over Mn}
Let $n$ and $p$ be positive integers. 
	\begin{itemize}
		\item[(i)]If $\mc{S}\subseteq\mm_{n\times p}$ is a left $\mm_n$-module, then there is a projection $Q\in\mm_p$ such that $\mc{S}=\mm_{n\times p}Q$.
		\item[(ii)]If $\mc{S}\subseteq\mm_{p\times n}$ is a right $\mm_n$-module, then there is a projection $Q\in\mm_p$ such that $\mc{S}=Q\mm_{p\times n}$.\smallskip
	\end{itemize}	
\end{thm}

\begin{prop}\label{prop k0=2 or k0=n}
Let $\mc{A}$ be a type I subalgebra of $\mm_n$. If there is a pair $\Omega=(d,\bigoplus_{i=1}^n\mc{V}_i)$ in $\mc{F}_I\mc{(A)}$ with $d=1$ or $d=n-1$, then $\mc{A}$ is the unitization of an $\mc{LR}$-algebra, and hence $\mc{A}$ is idempotent compressible.

\end{prop} 

\begin{proof}
Assume that $\mc{F}_I(\mc{A})$ contains a pair $\Omega=(n-1,\bigoplus_{i=1}^n\mc{V}_i)$. 
By Theorem~\ref{every algebra is similar to an unhinged algebra}, there exists an invertible upper triangular matrix $S$ such that ${\mc{A}_0\coloneqq S^{-1}\mc{A}S}$ is unhinged with respect to $\bigoplus_{i=1}^n\mc{V}_i$. Thus, since the class of $\mc{LR}$-algebras is invariant under similarity, it suffices to prove that $\mc{A}_0$ is the unitization of an $\mc{LR}$-algebra. 

Note that by Theorem~\ref{structure of modules over Mn}, there is a subprojection $Q_1^\prime\leq Q_{1\Omega}$ such that $$Q_{1\Omega}\mc{A}_0Q_{2\Omega}=Q_{1\Omega}Rad(\mc{A}_0)Q_{2\Omega}=Q_1^\prime\mm_nQ_{2\Omega}.\smallskip$$ Thus, either $\mc{V}_n$ is linked to the other $\mc{V}_i$'s, in which case $\mc{A}_0=Q_1^\prime\mm_nQ_{2\Omega}+\cc I;$ or $\mc{V}_n$ is not linked to the other $\mc{V}_i$'s, in which case $\mc{A}_0=(Q_1^\prime+Q_{2\Omega})\mm_nQ_{2\Omega} +\cc I.$ In either scenario, $\mc{A}_0$ is the unitization of an $\mc{LR}$-algebra.

Suppose instead that $\mc{F}_I\mc{(A)}$ contains a pair whose first entry is $1$. It follows that $\mc{F}_I(\mc{A}^{aT})$ contains a pair whose first entry is $n-1$. The above analysis then shows that $\mc{A}^{aT}$ is the unitization of an $\mc{LR}$-algebra, and thus so too is $\mc{A}$.
\end{proof}
\smallskip

\subsection{Type I Algebras with Unlinked Projections}\label{Subsection: Type I algebras unlinked}
In this section we consider the type I algebras $\mc{A}$ for which the pairs $\Omega=(d,\bigoplus_{i=1}^n\mc{V}_i)$ in $\mc{F}_I(\mc{A})$ are such that $Q_{1\Omega}$ and $Q_{2\Omega}$ are unlinked. In light of Proposition~\ref{prop k0=2 or k0=n} and its preceding remarks, we may assume that $1<d<n-1$ for all pairs $\Omega$. Thus, if $\Omega$ is any such pair, then $\min(d,n-d)\geq 2$. That is, the corresponding projections $Q_{1\Omega}$ and $Q_{2\Omega}$ have ranks $n_{1\Omega}\geq 2$ and $n_{2\Omega}\geq 2$, respectively. 

It will be shown in Theorem~\ref{unlinked type I general case theorem} that every projection compressible type I algebra satisfying the above assumptions is unitarily equivalent to the type~I algebra from Example~\ref{exmp:families of compressible algebras:1}. The majority of the work leading to this classification, however, occurs in Lemma~\ref{Case I - unlinked, same size lemma}. The proof of Lemma~\ref{Case I - unlinked, same size lemma} itself relies on several intermediate results concerning the structure of the radical of a projection compressible type I algebra.

It should be noted that while Lemmas~\ref{Case I - no 0 entry in radical lemma}, \ref{sum of blocks is an algebra lemma}, and \ref{particular algebra with 3 dimensional radical is not compressible lemma} are presented here in the context of type I algebras with unlinked projections, these results are also applicable to type I algebras whose projections are linked.

\begin{lem}\label{Case I - no 0 entry in radical lemma}

	Let $\mc{A}$ be a projection compressible type I subalgebra of $\mm_n$, and suppose that $\Omega=(d,\bigoplus_{i=1}^n\mc{V}_i)$ is a pair in $\mc{F}_I(\mc{A})$ with $1<d<n-1$. Suppose further that there are orthonormal bases $\left\{e_i^{(1)}\right\}_{i=1}^{n_{1\Omega}}$ for $\ran(Q_{1\Omega})$ and $\left\{e_i^{(2)}\right\}_{i=1}^{n_{2\Omega}}$ for $\ran(Q_{2\Omega})$, as well as indices $i_0$ and $j_0$ such that $$\langle Re_{j_0}^{(2)},e_{i_0}^{(1)}\rangle=0 \,\,\,\text{for all}\,\,\,R\in\rad.\smallskip$$ Then $Q_{1\Omega}$ and $Q_{2\Omega}$ are linked, and either $\langle R e_{j_0}^{(2)}, e_k^{(1)}\rangle=0$ for all $k\in\{1,2,\ldots, n_{1\Omega}\}$, or $\langle Re_k^{(2)},e_{i_0}^{(1)}\rangle=0$ for all $k\in\left\{1,2,\ldots, n_{2\Omega}\right\}.$
		
\end{lem}

\begin{proof}

	Suppose to the contrary that $Q_{1\Omega}$ and $Q_{2\Omega}$ are unlinked. By considering a suitable principal compression of $\mc{A}$ to a subalgebra of $\mm_4$, we may assume without loss of generality that $d=n_{1\Omega}=n_{2\Omega}=2$. Furthermore, we may reorder the bases if necessary to assume that $\langle R e_1^{(2)},e_2^{(1)}\rangle=0$ for all $R\in\rad$. 
	
	Since $\mc{A}$ is similar to $BD(\mc{A})\dotplus \rad$ via an upper triangular similarity, there is a fixed matrix $M$ in $Q_{1\Omega}\mc{A}Q_{2\Omega}$ such that with respect to the basis $\left\{e_1^{(1)},e_2^{(1)},e_1^{(2)},e_2^{(2)}\right\}$, every $A$ in $\mc{A}$ has the form 
	$$A=\left[\begin{array}{cc|cc}
	\alpha & 0 & 0 & 0\\
	& \alpha & 0 & 0\\ \hline
	& & \beta & 0\\
	& & & \beta
	\end{array}\right]+(\beta-\alpha)M+R$$
	for some $\alpha,\beta\in\cc$ and $R\in\rad.$
	
	For each $i,j\in\{1,2\}$ define $m_{ij}=\langle Me_j^{(2)},e_i^{(1)}\rangle$. Furthermore, for each $k\in\rr$ let $P_k$ denote the matrix
	$$P_k\coloneqq \begin{bmatrix}
	k^2+1 & 0 & 0 & 0\\
	0 & k^2 & 0 & -k\\
	0 & 0 & k^2+1 & 0\\
	0 & -k & 0 & 1
	\end{bmatrix},\smallskip$$
	so that $\frac{1}{k^2+1}P_k$ is a projection in $\mm_4$. By direct computation, one may verify that every element $B=(b_{ij})$ in $P_k\mc{A}P_k$ satisfies the equation 
	$$(k^2+1)b_{23}-m_{21}k^2(b_{33}-b_{11})=0.\smallskip$$ If, however, $A$ is as above with $\alpha=0$, $\beta=1$, and $R=0$, then for $(P_kAP_k)^2=(c_{ij})$, we have 
	$$(k^2+1)c_{23}-m_{21}k^2(c_{33}-c_{11})=m_{21}k^2
(k^2+1)^3(1-km_{22}).\smallskip$$
The fact that $\mc{A}$ is projection compressible implies that $(P_kAP_k)^2$ belongs to $P_k\mc{A}P_k$, and hence the right-hand side of the above expression must be $0$ for all $k$. We therefore deduce that $m_{21}=\langle Me_1^{(2)},e_2^{(1)}\rangle=0$.

It now follows that $\langle Ae_1^{(2)},e_2^{(1)}\rangle=0$ for all $A\in\mc{A}$. So with respect to the basis $\left\{e_1^{(1)},e_1^{(2)},e_2^{(1)},e_2^{(2)}\right\}$ for $\cc^4$, every $A\in\mc{A}$ may be expressed as 
$$A=\left[\begin{array}{cc|cc}
\alpha & (\beta-\alpha)m_{11}+r_{11} & 0 & (\beta-\alpha)m_{12}+r_{12}\\
& \beta & 0 & 0\\ \hline
& & \alpha & (\beta-\alpha)m_{22}+r_{22}\\
& & & \beta
\end{array}\right]\smallskip$$
for some $\alpha$, $\beta$, and $r_{ij}$ in $\cc$. Since $\alpha$ and $\beta$ may be chosen arbitrarily, this contradicts Theorem~\ref{At most one non-scalar corner theorem}. Thus, $Q_{1\Omega}$ and $Q_{2\Omega}$ must be linked.

For the final claim, observe that $BD(\mc{A})=\cc I$ as $Q_{1\Omega}$ and $Q_{2\Omega}$ are linked. By the remarks following Theorem~\ref{every algebra is similar to an unhinged algebra}, we have that $\mc{A}=\cc I\dotplus \rad$, and hence $$\langle Ae_{j_0}^{(2)},e_{i_0}^{(1)}\rangle=0\,\,\,\text{for all}\,\,\,A\in\mc{A}.$$ Suppose to the contrary that there exist indices $k_1\in\{1,2,\ldots, n_{1\Omega}\}\setminus\{i_0\}$, $k_2\in\{1,2,\ldots, n_{2\Omega}\}\setminus\{j_0\}$ and operators $A_1,A_2\in\mc{A}$ such that $\langle A_1e_{j_0}^{(2)},e_{k_1}^{(1)}\rangle\neq 0$ and $\langle A_2 e_{k_2}^{(2)},e_{i_0}^{(1)}\rangle\neq 0$. Let $P_1$ and $P_2$ denote the orthogonal projections onto $\mathrm{span}\{e_{k_1}^{(1)},e_{j_0}^{(2)}\}$ and $\mathrm{span}\{e_{i_0}^{(1)},e_{k_2}^{(2)}\}$, respectively. It is easy to see that $P_1\mc{A}P_1\neq\cc P_1$, $P_2\mc{A}P_2\neq\cc P_2$, and $P_2\mc{A}P_1=\{0\}$. Thus, Theorem~\ref{At most one non-scalar corner theorem} indicates that $\mc{A}$ is not projection compressible---a contradiction.
%
\end{proof}
\smallskip
%

\begin{lem}\label{sum of blocks is an algebra lemma}

Let $n\geq 4$ be an even integer, and let $\mc{A}$ be a projection compressible subalgebra of $\mm_{n}$. Let $Q_1$ be a projection in $\mm_{n}$ of rank $n/2$ and  define $Q_2\coloneqq I-Q_1$. If $E\in\mm_n$ is a partial isometry satisfying $E^*E=Q_1$ and $EE^*=Q_2$, then the linear space
$$\mc{A}_0\coloneqq \left\{Q_1AQ_1+E^*AQ_1+Q_1AE+E^*AE:A\in\mc{A}\right\}$$ is an algebra.

\end{lem}

\begin{proof}

The assumptions on $E$ imply that the operator $P\coloneqq \frac{1}{2}(I+E+E^*)$ is a projection in $\mm_n$, and hence $P\mc{A}P$ is an algebra. 
One may verify that with respect to the decomposition $\cc^n=\ran(Q_1)\oplus \ran(Q_2)$, we have
$$P\mc{A}P=\left\{\begin{bmatrix}
X & X\\ X & X
\end{bmatrix}:X\in\mc{A}_0\right\}.$$
It follows that for any $X$ and $Y$ in $\mc{A}_0$,
$$\begin{bmatrix}
X & X\\ X & X
\end{bmatrix}\begin{bmatrix}
Y & Y\\ Y & Y
\end{bmatrix}=2\begin{bmatrix}
XY & XY\\ XY& XY
\end{bmatrix}\in P\mc{A}P,$$ and hence $XY$ belongs to $\mc{A}_0$ as well. Thus, $\mc{A}_0$ is an algebra.
\end{proof}
\smallskip

\begin{lem}\label{particular algebra with 3 dimensional radical is not compressible lemma}

	Let $\mc{A}$ be a type I subalgebra of $\mm_4$. If $\rad$ is $3$-dimensional and $\mc{F}_I\mc{(A)}$ contains a pair $\Omega=(d,\bigoplus_{i=1}^4\mc{V}_i)$ with $d=2$, then $\mc{A}$ is not projection compressible.

\end{lem}

\begin{proof}

Suppose that $\dim\rad=3$ and $\Omega$ is a pair in $\mc{F}_I(\mc{A})$ as described above. Write $\mc{A}=\mc{S}\dotplus\rad$, where $\mc{S}$ is similar to $BD(\mc{A})$ via a block upper triangular similarity. If $Q_{1\Omega}$ and $Q_{2\Omega}$ are linked, then $\mc{A}=\left\{\alpha I: \alpha\in\cc\right\}\dotplus\rad.$ If instead $Q_{1\Omega}$ and $Q_{2\Omega}$ are unlinked, there is a matrix $M\in Q_{1\Omega}\mm_4Q_{2\Omega}$ such that
	$$\mc{A}=\left\{\alpha Q_{1\Omega}+\beta Q_{2\Omega}+(\beta-\alpha)M:\alpha,\beta\in\cc\right\}\dotplus\rad.\smallskip$$ 

Note that the only distinctions between the linked and unlinked settings are the presence of the matrix $M$ and the freedom to choose $\alpha$ and $\beta$ independently. In the arguments that follow, we treat the entries of $M$ as arbitrary constants (possibly zero), and make no attempt to choose independent values for $\alpha$ and $\beta$. Thus, these arguments are applicable to both cases. 

For each $i\in\{1,2\}$, let $\left\{e_1^{(i)},e_2^{(i)}\right\}$ be an orthonormal basis for $\ran(Q_{i\Omega})$. Since $\rad$ is a $3$-dimensional subspace of $Q_{1\Omega}\mm_4Q_{2\Omega}$, there is a non-zero matrix $\Gamma\in Q_{1\Omega}\mm_4Q_{2\Omega}$ such that $\mathrm{Tr}(\Gamma^*R)=0$ for all $R$ in $\rad.$ By reordering the bases for $\ran(Q_{1\Omega})$ and $\ran(Q_{2\Omega})$ if necessary, we may assume that $\langle \Gamma e_1^{(2)},e_1^{(1)}\rangle$ is non-zero. From this it follows that there exist $\gamma_{12},\gamma_{21},\gamma_{22}\in\cc$ such that 
	$$\rad=\left\{\left[\begin{array}{cc|cc} 
0 & 0 & \gamma_{12}r_{12}+\gamma_{21}r_{21}+\gamma_{22}r_{22} & r_{12}\\
0 & 0 & r_{21} & r_{22}\\ \hline
0 & 0 & 0 & 0\\
0 & 0 & 0 & 0
\end{array}\right]:r_{12},r_{21},r_{22}\in\cc\right\}$$
with respect to the basis $\left\{e_1^{(1)},e_2^{(1)},e_1^{(2)},e_2^{(2)}\right\}$ for $\cc^4$. 

To see that $\mc{A}$ is not projection compressible, consider the matrix 
$$P\coloneqq \begin{bmatrix}
2 & 0 & 0 & 0\\
0 & 1 & 0 & 1\\
0 & 0 & 2 & 0\\
0 & 1 & 0 & 1
\end{bmatrix},\smallskip$$ and note that $\frac{1}{2}P$ is a projection in $\mm_4$. One may verify that every operator $B=(b_{ij})$ in $P\mc{A}P$ satisfies the equation
\begin{align*}
b_{13}-4\gamma_{22}b_{24}-2\gamma_{21}b_{23}-2\gamma_{12}b_{14}-(&\gamma_{12}m_{12}+\gamma_{21}m_{21}-\gamma_{22}(1-m_{22})-m_{11})b_{11}\\
+(\gamma_{12}m_{12}&+\gamma_{21}m_{21}+\gamma_{22}(1+m_{22})-m_{11})b_{33}=0,
\end{align*}
where for each $i,j\in\{1,2\}$, we define $m_{ij}=\langle Me_j^{(2)},e_i^{(1)}\rangle.$ 
If, however, $A$ is the element of $\mc{A}$ obtained by setting $\alpha=\beta=r_{12}=r_{21}=1$ and $r_{22}=0$, then $B\coloneqq (PAP)^2$ produces a value of $8$ on the left-hand side of the above equation. Consequently, $(PAP)^2$ does not belong to $P\mc{A}P$, so $P\mc{A}P$ is not an algebra.
\end{proof}
\smallskip

The following classical theorem from linear algebra will be applied in the proof of Lemma~\ref{Case I - unlinked, same size lemma} and used extensively throughout \S5. For reference, see \cite[Theorem~2.6.3]{HornJohnson}.

\begin{thm}[Singular Value Decomposition]\label{svd lem} Let $n$ and $p$ be positive integers, and let $A$ be a complex $n\times p$ matrix.

	\begin{itemize}
	\item[(i)] If $n\leq p$, then there are unitaries $U\in\mm_n$ and $V\in\mm_p$, and a positive semi-definite diagonal matrix $D\in\mm_n$ such that
$$U^*AV=\left[\begin{array}{cc}
D & 0
\end{array}\right].$$

	\item[(ii)] If $n\geq p$, then there are unitaries $U\in\mm_n$ and $V\in\mm_p$, and a positive semi-definite diagonal matrix $D\in\mm_p$ such that
$$U^*AV=\left[\begin{array}{c}
D \\  0
\end{array}\right].$$

	\end{itemize}

\end{thm}

The principal application of Theorem~\ref{svd lem} will be in simplifying the structure of the semi-simple part of an algebra $\mc{A}$ in reduced block upper triangular form. Indeed, suppose that $\mc{A}=\mc{S}\dotplus\rad$ is a type~I subalgebra of $\mm_n$ where $\mc{S}$ is semi-simple. Let ${\Omega=(d,\bigoplus_{i=1}^n\mc{V}_i)}$ be a pair in $\mc{F}_I(\mc{A})$, and assume that the projections $Q_{1\Omega}$ and $Q_{2\Omega}$ are unlinked. For each $i\in\{1,2\}$, let $\left\{e_1^{(i)},e_2^{(i)},\ldots, e_{n_{i\Omega}}^{(i)}\right\}$ be an orthonormal basis for $\ran(Q_{i\Omega})$. As a consequence of Theorem~\ref{every algebra is similar to an unhinged algebra}, there is a matrix $M\in Q_{1\Omega}\mm_nQ_{2\Omega}$ such that $$\mc{S}=\left\{\alpha Q_{1\Omega}+\beta Q_{2\Omega}+(\beta-\alpha)M:\alpha,\beta\in\cc\right\}.\smallskip$$ 
\noindent It then follows from Theorem~\ref{svd lem} that there is a unitary $U\in\mm_n$ such that $Q_{1\Omega}UQ_{2\Omega}=0$, $Q_{2\Omega}UQ_{1\Omega}=0$, and $\langle U^*MUe_j^{(2)},e_i^{(1)}\rangle=0$ whenever $i\neq j.$

Finally, the proof of Lemma~\ref{Case I - unlinked, same size lemma} will require the following result of Azoff concerning the minimum dimension of a transitive space of linear operators.  Recall that a set $\mc{L}$ of linear transformations from $\cc^n$ to $\cc^m$ is said to be \textit{transitive} if for every non-zero $x\in\cc^n$ and arbitrary $y\in\cc^m$, there exists some $L\in\mc{L}$ such that $Lx=y$.

\begin{thm}\textup{\cite[Proposition 4.7]{Azoff}}\label{Azoff thm} If $\mc{L}$ is a transitive space of linear transformations from $\cc^n$ to $\cc^m$, then the dimension of $\mc{L}$ is at least $m+n-1$.
\end{thm}\smallskip

We are now prepared to state and prove Lemma~\ref{Case I - unlinked, same size lemma}. This result indicates that under certain restrictive assumptions, a projection compressible type I algebra with unlinked projections is unitarily equivalent to the type~I algebra from Example~\ref{exmp:families of compressible algebras:1}. Loosening these assumptions will require a refinement of Theorem~\ref{Azoff thm} to specific classes of transitive spaces of operators.

\begin{lem}\label{Case I - unlinked, same size lemma}

Let $n\geq 4$ be an even integer, and let $\mc{A}$ be a projection compressible type~I subalgebra of $\mm_n$. Suppose that $\mc{F}_I\mc{(A)}$ contains a pair $\Omega=(d,\bigoplus_{i=1}^n\mc{V}_i)$ with $d=n/2$. If the projections $Q_{1\Omega}$ and $Q_{2\Omega}$ are unlinked, then $\mc{A}$ is unitarily equivalent to $$\cc Q_{1\Omega}+\cc Q_{2\Omega}+Q_{1\Omega}\mm_n Q_{2\Omega},\smallskip$$
the type~I algebra from Example~\ref{exmp:families of compressible algebras:1}. Consequently, $\mc{A}$ is idempotent compressible.


\end{lem}

\begin{proof}
For each $i\in\{1,2\}$, let $\left\{e_1^{(i)},e_2^{(i)},\ldots, e_d^{(i)}\right\}$ be an orthonormal basis for $\ran(Q_{i\Omega})$. As a consequence of Theorem~\ref{every algebra is similar to an unhinged algebra}, there is a matrix $M$ in $Q_{1\Omega}\mm_n Q_{2\Omega}$ such that 
$$\mc{A}=\left\{\alpha Q_{1\Omega}+\beta Q_{2\Omega}+(\beta-\alpha)M:\alpha,\beta\in\cc\right\}\dotplus\rad.\smallskip$$ In fact, one may assume by Theorem~\ref{svd lem} and its subsequent remarks that there are constants $m_{ij}\geq 0$ such that $\langle Me_j^{(2)},e_i^{(1)}\rangle=\delta_{ij}m_{ij}$ for all $i$ and $j$.

 Let $E\in\mm_n$ denote the partial isometry satisfying $Ee_i^{(1)}=e_i^{(2)}$ and $Ee_i^{(2)}=0$ for all $i\in\{1,2,\ldots, d\}$. Since $\mc{A}$ is projection compressible, Lemma~\ref{sum of blocks is an algebra lemma} implies that $$\mc{A}_0\coloneqq \left\{(\alpha+\beta)Q_{1\Omega}+(\beta-\alpha)ME+RE:\alpha,\beta\in\cc, R\in\rad\right\}\smallskip$$ is a subalgebra of $Q_{1\Omega}\mm_nQ_{1\Omega}$. If this subalgebra were proper, then by Burnside's theorem, we may change the orthonormal basis for $\ran(Q_{1\Omega})$ if necessary to assume that $\langle Ae_1^{(1)},e_d^{(1)}\rangle=0$ for all $A\in\mc{A}_0$. In this case, one may change the orthonormal basis for $\ran(Q_{2\Omega})$ accordingly and assume that $\langle R e_1^{(2)},e_d^{(1)}\rangle=0$ for all $R\in\rad.$ Since $Q_{1\Omega}$ and $Q_{2\Omega}$ are unlinked, an application of Lemma~\ref{Case I - no 0 entry in radical lemma} demonstrates that $\mc{A}$ lacks the projection compression property---a contradiction. 

We may therefore assume that $\mc{A}_0$ is equal to $Q_{1\Omega}\mm_nQ_{1\Omega}$. This means that $\rad E$ can be enlarged to a $d^2$-dimensional space by adding $$
\{\alpha(Q_{1\Omega}-ME)+\beta(Q_{1\Omega}+ME):\alpha,\beta\in\cc\},\smallskip$$ the linear span of two diagonal matrices in $Q_{1\Omega}\mm_nQ_{1\Omega}$. It follows that $$\dim\rad E=\dim\rad\geq d^2-2,$$ and any entries in $\rad E$ that depend linearly on other entries must be located on the diagonal. Our goal is to show that $\dim\rad=d^2$, and hence $\rad=Q_{1\Omega}\mm_nQ_{2\Omega}$.

Let us begin by addressing the case in which $n=4$, and hence $d=2$. If $\dim\rad$ is strictly less than $d^2=4$, then $\rad$ is $2$- or $3$-dimensional by the analysis above. If $\dim\rad=2,$ then by Theorem~\ref{Azoff thm}, $\rad$ is not transitive as a space of linear maps from $\ran(Q_{2\Omega})$ to $\ran(Q_{1\Omega})$. In this case there exist unit vectors $v\in\ran(Q_{1\Omega})$ and $w\in\ran(Q_{2\Omega})$ such that $Rw\in\cc v$ for every $R\in\rad.$ 
Choose unit vectors ${v^\prime\in\ran(Q_{1\Omega})\cap(\cc v)^\perp}$ and $w^\prime\in\ran(Q_{2\Omega})\cap(\cc w)^\perp$, and replace the orthonormal bases for $\ran(Q_{1\Omega})$ and $\ran(Q_{2\Omega})$ with $\{v, v^\prime\}$ and $\{w,w^\prime\}$, respectively. Since $$\begin{array}{rl}\langle Rw,v^\prime\rangle=\langle \lambda v,v^\prime\rangle=0& \text{for all}\,\, R\in\rad,\end{array}\smallskip$$ 
$\mc{A}$ lacks the projection compression property by Lemma~\ref{Case I - no 0 entry in radical lemma}---a contradiction. Using Lemma~\ref{particular algebra with 3 dimensional radical is not compressible lemma}, one may also obtain a contradiction in the case that $\dim\rad=3$.

Assume now that $n>4$. By the above analysis, there are at most two entries from $\rad E$ which cannot be chosen arbitrarily, and these entries necessarily occur on the diagonal. By reordering the bases for $\ran(Q_{1\Omega})$ and $\ran(Q_{2\Omega})$, we may relocate the linearly dependent entries to the $(1,n-1)$ and $(2,n)$ positions of $\rad$, respectively. That is, we may assume that with respect to the decomposition 
$$\cc^n=\vee\left\{e_1^{(1)},e_2^{(1)},\ldots, e_{d-1}^{(1)}\right\}\oplus \vee\left\{e_d^{(1)},e_1^{(2)}\right\}\oplus \vee\left\{e_2^{(2)},e_3^{(2)},\ldots, e_{d}^{(2)}\right\},\smallskip$$ each $A\in\mc{A}$ can be represented by a matrix of the form 
$$A=\left[\begin{array}{ccccc|cc|ccccc}
	\alpha & & & & & 0 & t_{11} & t_{12} & \cdots & t_{1,d-2} & \gamma_1 & t_{1d}\\
	& \alpha & & & & 0 & t_{21} & t_{22} & \cdots & t_{2,d-2} & t_{2,d-1} & \gamma_2\\ 
	&  & \alpha & & & 0 & t_{31} & t_{32} & \cdots & t_{3,d-2} & t_{3,d-1} & t_{3d}\\
	& & & \ddots & & \vdots & \vdots & \vdots & \ddots & \vdots & \vdots & \vdots\\
	 &&&& \alpha & 0 & t_{d-1,1} & t_{d-1,2} & \cdots & t_{d-1,d-2} & t_{d-1,d-1} & t_{d-1,d}\\ \hline
	 &&&&&\alpha & t_{d1} & t_{d2} & \cdots & t_{d,d-2} & t_{d,d-1} & t_{dd}\\
	 &&&&&&\beta & 0 & \cdots & 0 & 0 & 0 \\ \hline
	 &&&&&&&\beta\\
	 &&&&&&&&\ddots\\
	 &&&&&&&&&\beta\\
	 &&&&&&&&&&\beta\\
	 &&&&&&&&&&&\beta	   
	\end{array}\right],$$
	where $\alpha$, $\beta$, and $t_{ij}$ can be chosen arbitrarily, and $\gamma_1$ and $\gamma_2$ may depend linearly on these entries. We will demonstrate that, in fact, $\gamma_1$ and $\gamma_2$ can be chosen arbitrarily and independently of the remaining terms.
	
	Consider the matrix 
	$$P=\left[\begin{array}{ccccc|cc|ccccc}
	2 & & & & & &  &  & &  & &\\
	& 2 & & & & &  & &  &  &  & \\ 
	&  & 2 & & & &  &  &  &  & & \\
	& & & \ddots& && & &  &  &  & \\
	 &&&& 2 & &  &  & & &  & \\ \hline
	 &&&&& 1 & 1 &&&&&\\
	 &&&&& 1 & 1 & & & & & \\ \hline
	 &&&&&&&2\\
	 &&&&&&&&\ddots\\
	 &&&&&&&&&2\\
	 &&&&&&&&&&2\\
	 &&&&&&&&&&&2	   
	\end{array}\right]$$
	written with respect to the decomposition above. Observe that $\frac{1}{2}P$ is a projection in $\mm_n$. Direct computations show that with $A$ as above, $PAP$ is given by	
\begin{equation*}
\resizebox{\textwidth}{!} 
{
$\left[\begin{array}{ccccc|cc|ccccc}
	4\alpha & & & & & 2t_{11} & 2t_{11} & 4t_{12} & \cdots & 4t_{1,d-2} & 4\gamma_1 & 4t_{1d}\\
	& 4\alpha & & & & 2t_{21} & 2t_{21} & 4t_{22} & \cdots & 4t_{2,d-2} & 4t_{2,d-1} & 4\gamma_2\\ 
	&  & 4\alpha & & & 2t_{31} & 2t_{31} & 4t_{32} & \cdots & 4t_{3,d-2} & 4t_{3,d-1} & 4t_{3d}\\
	& & & \ddots & & \vdots & \vdots & \vdots & \ddots & \vdots & \vdots & \vdots\\
	 &&&& 4\alpha & 2t_{d-1,1} & 2t_{d-1,1} & 4t_{d-1,2} & \cdots & 4t_{d-1,d-2} & 4t_{d-1,d-1} & 4t_{d-1,d}\\ \hline
	 &&&&&\alpha+\beta+t_{d1} & \alpha+\beta+t_{d1} & 2t_{d2} & \cdots & 2t_{d,d-2} & 2t_{d,d-1} & 2t_{dd}\\
	 &&&&&\alpha+\beta+t_{d1} & \alpha+\beta+t_{d1} & 2t_{d2} & \cdots & 2t_{d,d-2} & 2t_{d,d-1} & 2t_{dd}\\ \hline
	 &&&&&&&4\beta\\
	 &&&&&&&&\ddots\\
	 &&&&&&&&&4\beta\\
	 &&&&&&&&&&4\beta\\
	 &&&&&&&&&&&4\beta	   
	\end{array}\right].
	$
	}
	\end{equation*}
Hence, it suffices to prove that $e_1^{(1)}\otimes {e_{d-1}^{(2)*}}$ and $e_2^{(1)}\otimes {e_d^{(2)*}}$ belong to $P\mc{A}P$. 

	To see that this is the case, let $A$ be as above with $t_{11}=t_{d,d-1}=1$ and $\alpha=\beta=t_{ij}=0$ for all other indices $i$ and $j$. It is straightforward to verify that $$(PAP)^2=8e_1^{(1)}\otimes {e_{d-1}^{(2)^*}}.\smallskip$$ Consequently, $e_1^{(1)}\otimes {e_{d-1}^{(2)*}}$ belongs to $P\mc{A}P$, so $\gamma_1$ can indeed be chosen arbitrarily. By reordering the basis to interchange the positions of $\gamma_1$ and $\gamma_2$, one may repeat this process to show that $\gamma_2$ may be chosen arbitrarily as well.
\end{proof}
\smallskip

Observe that the success of Lemma~\ref{Case I - unlinked, same size lemma} relied heavily on the existence of the pair $\Omega=(d,\bigoplus_{i=1}^n\mc{V}_i)$ with $d=n/2$. Indeed, without such a pair, one would be unable to directly apply Lemma~\ref{sum of blocks is an algebra lemma} or Burnside's Theorem to infer that $\dim\rad\geq d^2-2$.

Our final goal of this section is to generalize Lemma~\ref{Case I - unlinked, same size lemma} to type I algebras $\mc{A}$ that may not admit a pair $\Omega$ as describe above. We will accomplish this goal by applying Lemma~\ref{Case I - unlinked, same size lemma} to study the structure of the radical of certain principal compressions of $\mc{A}$. It will then follow from \cite[Theorem 1.2]{DMRTransitive}---an extension of Theorem~\ref{Azoff thm}---that $\mc{A}$ is unitarily equivalent to the type~I algebra from Example~\ref{exmp:families of compressible algebras:1}. In order to introduce this extension, we first present the following definition.

%
%

\begin{defn}

Let $\mc{L}$ be a vector space of linear transformations from $\cc^n$ to $\cc^m$, and let $k$ be a positive integer. We say that $\mc{L}$ is $k$\textit{-transitive} if for every choice of $k$ linearly independent vectors $x_1,x_2,\ldots, x_k$ in $\cc^n$, and every choice of $k$ arbitrary vectors $y_1,y_2,\ldots, y_k$ in $\cc^m$, there is an element $A\in\mc{L}$ such that $Ax_i=y_i$ for all $i\in\{1,2,\ldots, k\}.$\smallskip

\end{defn}

\begin{thm}\textup{\cite[Theorem 1.2]{DMRTransitive}}\label{DMR Theorem} If $\mc{L}$ is a $k$-transitive space of linear transformations from $\cc^n$ to $\cc^m$, then the dimension of $\mc{L}$ is at least $k(m+n-k)$.\smallskip

\end{thm}
%
%

We are now prepared to prove the classification in the general case of type I algebras with unlinked projections.

\begin{thm}\label{unlinked type I general case theorem}

	Let $\mc{A}$ be a projection compressible type I subalgebra of $\mm_n$, and let $\Omega=(d,\bigoplus_{i=1}^n\mc{V}_i)$ be a pair in $\mc{F}_I(\mc{A})$ with $1<d<n-1$. If $Q_{1\Omega}$ and $Q_{2\Omega}$ are unlinked, then $\mc{A}$ is unitarily equivalent to
	$$\cc Q_{1\Omega}+\cc Q_{2\Omega}+Q_{1\Omega}\mm_nQ_{2\Omega},\smallskip$$
	the type I algebra from Example~\ref{exmp:families of compressible algebras:1}. Consequently, $\mc{A}$ is idempotent compressible.

\end{thm}

\begin{proof}

	By replacing $\mc{A}$ with $\mc{A}^{aT}$ if necessary, we may assume that $d\leq n-d$. That is, $n_{1\Omega}\leq n_{2\Omega}$. We will demonstrate that $\rad$ has dimension $d(n-d)$, and hence must be equal to $Q_{1\Omega}\mm_nQ_{2\Omega}$. Of course, it is clear that $\dim\rad\leq d(n-d)$.
	
	Note that $\rad$ is $d$-transitive as a space of linear maps from $\ran(Q_{2\Omega})$ to $\ran(Q_{1\Omega})$. Indeed, let $S$ be a linearly independent $d$-element subset of $\ran(Q_{2\Omega})$, and let $Q_S$ denote the orthogonal projection onto the span of $S$. Since $Q_{1\Omega}$ and $Q_S$ are both of rank $d$, Lemma~\ref{Case I - unlinked, same size lemma} implies that the radical of  $$\mc{A}_0\coloneqq (Q_{1\Omega}+Q_S)\mc{A}(Q_{1\Omega}+Q_S)$$ is equal to $Q_{1\Omega}\mm_nQ_S$. As a result, the vectors in $S$ can be mapped anywhere in $\ran(Q_{1\Omega})$ by elements of $\rad$. We conclude that $\rad$ is $d$-transitive.
	
	The proof ends with an application of Theorem~\ref{DMR Theorem}. Since $\rad$ is a $d$-transitive subspace of $Q_{1\Omega}\mm_nQ_{2\Omega}$, we have that $\dim\rad\geq d(d+(n-d)-d)=d(n-d).$
\end{proof}
\smallskip

\subsection{Type I Algebras with Linked Projections}\label{Subsection: Type I algebras linked}
We now wish to describe the projection compressible type I algebras $\mc{A}$ for which the pairs $\Omega=(d,\bigoplus_{i=1}^n\mc{V}_i)$ in $\mc{F}_I(\mc{A})$ are such that $Q_{1\Omega}$ is linked to $Q_{2\Omega}$. An inductive argument in Theorem~\ref{big thm for linked type I algebras} will demonstrate that every such algebra is the unitization of an $\mc{LR}$-algebra. The base case of this argument will require the following lemma.

\begin{lem}\label{Case I - n_1=2 case for linked algebras lemma}

	Let $\mc{A}$ be a projection compressible type I subalgebra of $\mm_4$, and suppose that $\mc{F}_I\mc{(A)}$ contains a pair $\Omega=(d,\bigoplus_{i=1}^4\mc{V}_i)$ with $d=2$. If $Q_{1\Omega}$ and $Q_{2\Omega}$ are linked, then there are projections $Q_1^\prime\leq Q_{1\Omega}$ and $Q_2^\prime\leq Q_{2\Omega}$ such that $\rad=Q_1^\prime\mm_4Q_2^\prime$. In this case $\mc{A}$ is the unitization of an $\mc{LR}$-algebra, so $\mc{A}$ is idempotent compressible.

\end{lem}

\begin{proof}

Let $\Omega$ be a pair in $\mc{F}_I(\mc{A})$ as above, and assume that $Q_{1\Omega}$ and $Q_{2\Omega}$ are linked. By the observations following Theorem~\ref{every algebra is similar to an unhinged algebra}, $\mc{A}=\cc I\dotplus \rad$.
	
	For each $i\in\{1,2\}$, let $\left\{e_1^{(i)},e_2^{(i)}\right\}$ be a fixed orthonormal basis for $\ran(Q_{i\Omega})$. Furthermore, let $E\in\mm_n$ denote the partial isometry satisfying $Ee_i^{(1)}=e_i^{(2)}$ and $Ee_i^{(2)}=0$ for each $i\in\{1,2\}$. By Lemma~\ref{sum of blocks is an algebra lemma}, $$\mc{A}_0\coloneqq \cc Q_{1\Omega}+\rad E$$ is a subalgebra of $Q_{1\Omega}\mm_4Q_{1\Omega}$. If this subalgebra $\mc{A}_0$ is proper, then by Burnside's Theorem, we may change the orthonormal basis for $\ran(Q_{1\Omega})$ if required and assume that $\langle Ae_1^{(1)},e_2^{(1)}\rangle=0$ for all $A\in\mc{A}_0$. In this case we may adjust the orthonormal basis for $\ran(Q_{2\Omega})$ accordingly and assume that $\langle Re_1^{(2)},e_2^{(1)}\rangle=0$ for all $R\in\rad.$ Thus, by Lemma~\ref{Case I - no 0 entry in radical lemma}, either $\langle Re_1^{(2)},e_1^{(1)}\rangle=0$ for all $R\in\rad$, or $\langle Re_2^{(2)},e_2^{(1)}\rangle=0$ for all ${R\in\rad}$. The fact that $\rad$ has the required form now follows from Theorem~\ref{structure of modules over Mn}.
	
	Suppose instead that $\cc Q_{1\Omega}+\rad E$ is equal to $Q_{1\Omega}\mm_4Q_{1\Omega}$. It follows that $\rad$ is at least $3$-dimensional.  If $\dim\rad=3$, then $\mc{A}$ is of the form described in Lemma~\ref{particular algebra with 3 dimensional radical is not compressible lemma}, and hence $\mc{A}$ is not projection compressible. We therefore have that $\dim\rad=4$, so $\rad=Q_{1\Omega}\mm_4Q_{2\Omega}$.
\end{proof}
\smallskip

	\begin{thm}\label{big thm for linked type I algebras}
	Let $\mc{A}$ be a projection compressible type I subalgebra of $\mm_n$, and let $\Omega=(d,\bigoplus_{i=1}^n\mc{V}_i)$ be a pair in $\mc{F}_I\mc{(A)}$. If $Q_{1\Omega}$ and $Q_{2\Omega}$ are linked, then there are projections $Q_1^\prime\leq Q_{1\Omega}$ and $Q_2^\prime\leq Q_{2\Omega}$ such that $\rad=Q_1^\prime\mm_nQ_2^\prime$. Thus, $\mc{A}$ is the unitization of an $\mc{LR}$-algebra, so $\mc{A}$ is idempotent compressible.
	
	\end{thm}
	
	\begin{proof}
	
	We will proceed by induction on $n$. By definition of a type I algebra, our base case occurs when $n=4$. That said, let $\mc{A}$ be a projection compressible type I subalgebra of $\mm_4$, and suppose that $\Omega=(d,\bigoplus_{i=1}^n\mc{V}_i)$ is a pair in $\mc{F}_I(\mc{A})$ with $Q_{1\Omega}$ linked to $Q_{2\Omega}$. If $d=1$ or $d=3$, then Proposition~\ref{prop k0=2 or k0=n} guarantees that $\rad$ admits the required form. If instead $d=2$, then $\mc{A}$ and $\Omega$ are as in Lemma~\ref{Case I - n_1=2 case for linked algebras lemma}. Once again $\rad$ is of the correct form.

	Now fix an integer $N\geq 5$. Assume that for every positive integer $n<N$, if $\mc{A}$ is a projection compressible type~I subalgebra of $\mm_n$ and $\Omega$ is a pair in $\mc{F}_I\mc{(A)}$ with $Q_{1\Omega}$ linked to $Q_{2\Omega}$, then $\rad=Q_1^\prime\mm_n Q_2^\prime$ for some subprojections $Q_1^\prime\leq Q_{1\Omega}$ and $Q_2^\prime\leq Q_{2\Omega}$. We claim that this is also the case for every such subalgebra $\mc{A}$ of $\mm_N$ and pair $\Omega\in\mc{F}_I(\mc{A})$. 
Indeed, fix a subalgebra $\mc{A}$ of $\mm_N$ and pair $\Omega=(d,\bigoplus_{i=1}^N\mc{V}_i)$ in $\mc{F}_I(\mc{A})$ as in the statement of the theorem. If $d=1$ or $d=N-1$, then Proposition~\ref{prop k0=2 or k0=n} ensures that $\rad$ is of the desired form. Thus, we will assume that $1<d<N-1$. By replacing $\mc{A}$ with $\mc{A}^{aT}$ if necessary, we will also assume that $d\leq N-d$. 

First consider the case that $N$ is even and $d=N-d=N/2$. Fix orthonormal bases $\left\{e_1^{(1)},e_2^{(1)},\ldots, e_{d}^{(1)}\right\}$ and $\left\{e_1^{(2)},e_2^{(2)},\ldots, e_{d}^{(2)}\right\}$ for $\ran(Q_{1\Omega})$ and $\ran(Q_{2\Omega})$, respectively. Let $E\in\mm_n$ denote the partial isometry satisfying $Ee_i^{(1)}=e_i^{(2)}$ and $Ee_i^{(2)}=0$ for each $i\in\{1,2,\ldots, d\}$. Arguing as in the proof of Lemma~\ref{Case I - unlinked, same size lemma}, either $\cc Q_{1\Omega}+\rad E$ is equal to $Q_{1\Omega}\mm_NQ_{1\Omega}$, or Burnside's Theorem may be used to assume that  
$$\langle Re_1^{(2)},e_d^{(1)}\rangle=0\,\,\text{for all}\,\, R\in\rad.\smallskip$$
\indent \,If the latter holds, then by Lemma~\ref{Case I - no 0 entry in radical lemma}, $\rad$ contains a permanent row or column of zeros. In the case of a permanent row of zeros, consider the algebra $\mc{A}_0$ obtained be deleting this row and its corresponding column from $\mc{A}$. We have that $\mc{A}_0$ is a projection compressible type I subalgebra of $\mm_{N-1}$, so $Rad(\mc{A}_0)$ admits the the required form by the inductive hypothesis. Upon reintroducing the removed row and column, one can see that $\rad$ is also of the required form. An analogous argument can be made in the case of a permanent column of zeros. We may therefore assume that $\cc Q_{1\Omega}+\rad E=Q_{1\Omega}\mm_NQ_{1\Omega}$. 

Since $\rad E$ can be enlarged to a $d^2$-dimensional space by adding $\cc Q_{1\Omega}$, $\dim\rad\geq d^2-1$. We claim that in fact, $\dim\rad=d^2$, and hence $\rad=Q_{1\Omega}\mm_n Q_{2\Omega}$. To see this is the case, reorder the bases for $\ran(Q_{1\Omega})$ and $\ran(Q_{2\Omega})$ if necessary to assume that with respect to the decomposition
	$$\cc^N=\vee\left\{e_1^{(1)},e_2^{(1)},\ldots, e_{d-1}^{(1)}\right\}\oplus \vee\left\{e_{d}^{(1)},e_1^{(2)}\right\}\oplus \vee\left\{e_2^{(2)},e_3^{(2)},\ldots, e_{d}^{(2)}\right\},$$
	each $A\in\mc{A}$ can be expressed as a matrix of the form 
$$A=\left[\begin{array}{ccccc|cc|ccccc}
	\alpha & & & & & 0 & t_{11} & t_{12} & \cdots & t_{1,d-2} & \gamma & t_{1d}\\
	& \alpha & & & & 0 & t_{21} & t_{22} & \cdots & t_{2,d-2} & t_{2,d-1} & t_{2d}\\ 
	&  & \alpha & & & 0 & t_{31} & t_{32} & \cdots & t_{3,d-2} & t_{3,d-1} & t_{3d}\\
	& & & \ddots & & \vdots & \vdots & \vdots & \ddots & \vdots & \vdots & \vdots\\
	 &&&& \alpha & 0 & t_{d-1,1} & t_{d-1,2} & \cdots & t_{d-1,d-2} & t_{d-1,d-1} & t_{d-1,d}\\ \hline
	 &&&&&\alpha & t_{d1} & t_{d2} & \cdots & t_{d,d-2} & t_{d,d-1} & t_{dd}\\
	 &&&&&&\alpha & 0 & \cdots & 0 & 0 & 0 \\ \hline
	 &&&&&&&\alpha\\
	 &&&&&&&&\ddots\\
	 &&&&&&&&&\alpha\\
	 &&&&&&&&&&\alpha\\
	 &&&&&&&&&&&\alpha	   
	\end{array}\right].$$
	Here, $\alpha$ and $t_{ij}$ are arbitrary values in $\cc$, and $\gamma$ may depend linearly on these entries. 
	
	It will be shown that $\gamma$ is in fact, independent of the other terms. Indeed, let $P$ denote the matrix from the proof of Lemma~\ref{Case I - unlinked, same size lemma}, so that $\frac{1}{2}P$ is a projection in $\mm_N$. Proceed now as in the proof of that lemma by noting that with $A$ as above, $PAP$ is given by $$\left[\begin{array}{ccccc|cc|ccccc}
	4\alpha & & & & & 2t_{11} & 2t_{11} & 4t_{12} & \cdots & 4t_{1,d-2} & 4\gamma & 4t_{1d}\\
	& 4\alpha & & & & 2t_{21} & 2t_{21} & 4t_{22} & \cdots & 4t_{2,d-2} & 4t_{2,d-1} & 4t_{2d}\\ 
	&  & 4\alpha & & & 2t_{31} & 2t_{31} & 4t_{32} & \cdots & 4t_{3,d-2} & 4t_{3,d-1} & 4t_{3d}\\
	& & & \ddots & & \vdots & \vdots & \vdots & \ddots & \vdots & \vdots & \vdots\\
	 &&&& 4\alpha & 2t_{d-1,1} & 2t_{d-1,1} & 4t_{d-1,2} &\cdots & 4t_{d-1,d-2} & 4t_{d-1,d-1} & 4t_{d-1,d}\\ \hline
	 &&&&&2\alpha+t_{d1} & 2\alpha+t_{d1} & 2t_{d2} & \cdots & 2t_{d,d-2} & 2t_{d,d-1} & 2t_{dd}\\
	 &&&&&2\alpha+t_{d1} & 2\alpha+t_{d1} & 2t_{d2} & \cdots & 2t_{d,d-2} & 2t_{d,d-1} & 2t_{dd}\\ \hline
	 &&&&&&&4\alpha\\
	 &&&&&&&&\ddots\\
	 &&&&&&&&&4\alpha\\
	 &&&&&&&&&&4\alpha\\
	 &&&&&&&&&&&4\alpha	   
	\end{array}\right].$$ It therefore suffices to prove that  $e_1^{(1)}\otimes e_{d-1}^{(2)*}$ belongs to $P\mc{A}P$. But if $A$ denotes the particular element of $\mc{A}$ obtained by taking $t_{11}=t_{d,d-1}=1$ and $\alpha=t_{ij}=0$ for all other indices $i$ and $j$, then $$(PAP)^2=8e_1^{(1)}\otimes e_{d-1}^{(2)*}.\smallskip$$ Since $\mc{A}$ is projection compressible, this element belongs to $P\mc{A}P$. We conclude that $\rad=Q_{1\Omega}\mm_NQ_{2\Omega}$, and hence the proof of the $d=N-d$ case is complete.

Let us now turn to the case in which $d< N-{d}$. As above, let $\left\{e_1^{(1)},e_2^{(1)},\ldots, e_{n_{1\Omega}}^{(1)}\right\}$ and $\left\{e_1^{(2)},e_2^{(2)},\ldots, e_{n_{2\Omega}}^{(2)}\right\}$ be fixed orthonormal bases for $\ran(Q_{1\Omega})$ and $\ran(Q_{2\Omega})$, respectively. For each linearly independent $d$-element subset $S$ of $\ran(Q_{2\Omega})$, let $Q_S$ denote the orthogonal projection onto the span of $S$, and define $P_S\coloneqq Q_{1\Omega}+Q_S$. Let $\mc{A}_S$ denote the compression $P_S\mc{A}P_S$, which we regard as a subalgebra of $\cc I\dotplus Q_{1\Omega}\mm_{2d}Q_S$. 

If each compression $\mc{A}_S$ is equal to $\cc I\dotplus Q_{1\Omega}\mm_{2d}Q_S$, then $\rad$ is a $d$-transitive space of linear maps from $\ran(Q_{2\Omega})$ into $\ran(Q_{1\Omega})$. In this case we may apply Theorem~\ref{DMR Theorem} to conclude that $\rad=Q_{1\Omega}\mm_NQ_{2\Omega}$, as desired. Instead, suppose that one of the sets $S$ is such that the radical of $\mc{A}_S$ is properly contained in $Q_{1\Omega}\mm_{2d} Q_S$. For such an $S$, the inductive hypothesis gives rise to subprojections $Q_1^\prime\leq Q_{1\Omega}$ and $Q_S^\prime\leq Q_S$ such that $$Rad(\mc{A}_S)=Q_1^\prime\mm_{2d}Q_S^\prime.\smallskip$$ At least one of these subprojections must be proper.

	If $Q_S^\prime\neq Q_S$ or $Q_1^\prime=0$, then there is an orthonormal basis for $\cc^{2d}$ with respect to which $Rad(\mc{A}_S)$ has a permanent column of zeros. One may then extend this basis to an orthonormal basis for $\cc^N$ with respect to which $\rad$ also admits a permanent column of zeros. By deleting this column and its corresponding row from $\mc{A}$, we obtain a projection compressible type I subalgebra of $\mm_{N-1}$. The inductive hypothesis then implies that the radical of this compression is of the desired form. Upon reintroducing the column and row deleted from $\mc{A}$, it is easy to see that $\rad$ is of the desired form as well.
	
	On the other hand, if $Q_S=Q_S^\prime$ and $Q_1^\prime$ is a proper non-zero subprojection of $Q_{1\Omega}$, then it must be the case that $Rad(\mc{A}_S)$ has a permanent row of zeros, but not a permanent column of zeros. Thus, $\rad$ has a permanent  row of zeros by Lemma~\ref{Case I - no 0 entry in radical lemma}. By removing this row and its corresponding column from $\mc{A}$, we obtain a projection compressible type~I subalgebra of $\mm_{N-1}$. The radical of this algebra is of the correct form by the inductive hypothesis, and hence so too is $\rad$.	
	\end{proof}
\smallskip

\section[5]{Algebras of Type II}\label{case2}

	The term \textit{type II} will be used to describe a unital subalgebra $\mc{A}$ of $\mm_n$, $n\geq 4$, that has a reduced block upper triangular form with respect to an orthogonal decomposition $\bigoplus_{i=1}^m\mc{V}_i$ of $\cc^n$, such that $\dim\mc{V}_k\geq 2$ for some $k$. For example, the algebra from Example~\ref{exmp:families of compressible algebras:1} is of type II if and only if $\rank(Q_2)\geq 2$. It follows from this definition that every type II algebra satisfies the assumptions of Corollary~\ref{special case of uniqueness corollary}. 
	
	The purpose of this section is to classify the type II algebras that afford the projection compression property. It will be shown that every projection compressible algebra of type~II is either the unitization of an $\mc{LR}$-algebra, or is unitarily equivalent to the type II algebra from Example~\ref{exmp:families of compressible algebras:1}. 

As in the case of type I algebras, it will be helpful to keep a record of all orthogonal decompositions of $\cc^n$ that satisfy the conditions of Corollary~\ref{special case of uniqueness corollary} for a given type II algebra $\mc{A}$. Thus, we make the following definition.\smallskip

\begin{defn}\label{definition of F2(A)}
If $\mc{A}$ is an algebra of type II, let $\mc{F}_{II}=\mc{F}_{II}\mc{(A)}$ denote the set of triples $\Omega=(d,k,\bigoplus_{i=1}^m\mc{V}_i)$ that satisfy the following conditions:

\begin{itemize}

		\item[(i)]$\bigoplus_{i=1}^m\mc{V}_i$ is an orthogonal decomposition of $\cc^n$ with respect to which $\mc{A}$ is reduced block upper triangular;\smallskip
		
		\item[(ii)]$d$ and $k$ are integers such that $d\geq 2$, $k\in\{1,2,\ldots, m\}$, and $\dim\mc{V}_k=d$.\smallskip \vspace{0.1cm}

	\end{itemize}
\end{defn}

\begin{notation}
\upshape{If $\mc{A}$ is an algebra of type II and $\Omega$ is a triple in $\mc{F}_{II}(\mc{A})$, let $Q_{1\Omega}$, $Q_{2\Omega}$, and $Q_{3\Omega}$ denote the orthogonal projections onto $\bigoplus_{i<k}\mc{V}_i$, $\mc{V}_k$, and $\bigoplus_{i>k}\mc{V}_i$, respectively. Furthermore, for each $i\in\{1,2,3\}$, let $n_{i\Omega}$ denote the rank of $Q_{i\Omega}$. \smallskip
}
\end{notation}

	Observe that if $\mc{A}$ is a projection compressible type II subalgebra of $\mm_n$ and $\mc{F}_{II}(\mc{A})$ contains a triple $\Omega=(d,\bigoplus_{i=1}^m\mc{V}_i)$, then Corollary~\ref{special case of uniqueness corollary} implies that $Q_{2\Omega}\mc{A}Q_{2\Omega}=Q_{2\Omega}\mm_n Q_{2\Omega}$ and $Q_{i\Omega}\mc{A}Q_{i\Omega}=\cc Q_{i\Omega}$ for each $i\in\{1,3\}$. In this case, $n_{1\Omega}=k-1$, $n_{2\Omega}=d$, and $n_{3\Omega}=n-d-k+1$.
	
	We will begin by considering the extreme case of a type~II algebra $\mc{A}$ such that $\mc{F}_{II}(\mc{A})$ contains a triple $\Omega=(d,k,\bigoplus_{i=1}^m\mc{V}_i)$ with $k=1$ or $k=m$. The projection compressible algebras of this form can be easily identified using Theorem~\ref{structure of modules over Mn}.

	\begin{prop}\label{case2 - k=1 or k=m prop}

		Let $\mc{A}$ be a projection compressible type II subalgebra of $\mm_n$. If there exists a triple \linebreak $\Omega=(d,k,\bigoplus_{i=1}^m\mc{V}_i)$ in $\mc{F}_{II}(\mc{A})$ with $k=1$ or $k=m$, then $\mc{A}$ is the unitization of an $\mc{LR}$-algebra. Consequently, $\mc{A}$ is idempotent compressible.
	
	\end{prop}
	
	\begin{proof}
	
		Let $\Omega\in\mc{F}_{II}(\mc{A})$ be as in the statement above. By replacing $\mc{A}$ with $\mc{A}^{aT}$ if necessary, we may assume that $k=m$. Furthermore, since any algebra similar to an $\mc{LR}$-algebra is again an $\mc{LR}$-algebra, we may assume that $\mc{A}$ is unhinged with respect to $\bigoplus_{i=1}^m\mc{V}_i$. 
		
		Since $\rad$ is a right $\mm_d$-module, Theorem~\ref{structure of modules over Mn} indicates that $\rad=Q_1^\prime\mm_{n}Q_2$ for some projection $Q_1^\prime\leq Q_{1\Omega}$. It follows that, $$\mc{A}=BD(\mc{A})\dotplus\rad=(Q_1^\prime+Q_{2\Omega})\mm_nQ_{2\Omega}+\cc I,\smallskip$$
		and hence $\mc{A}$ is the unitization of an $\mc{LR}$-algebra.
	\end{proof}	
	\smallskip
	
	By Proposition~\ref{case2 - k=1 or k=m prop}, it suffices to consider the type II algebras for which the triples $\Omega=(d,k,\bigoplus_{i=1}^m\mc{V}_i)$ in $\mc{F}_{II}$ are such that $1<k<m$. For such an algebra $\mc{A}$ and triple $\Omega$, the projections $Q_{1\Omega}$, $Q_{2\Omega}$, and $Q_{3\Omega}$ are all non-zero. In the language of Theorem~\ref{Cor 14 from LMMR} and the remarks that follow, the corners $Q_{1\Omega}\mc{A}Q_{1\Omega}$ and $Q_{3\Omega}\mc{A}Q_{3\Omega}$ are diagonal algebras, each comprised of mutually linked $1\times 1$ blocks. Note that the blocks in $Q_{1\Omega}\mc{A}Q_{1\Omega}$ may be linked to those in $Q_{3\Omega}\mc{A}Q_{3\Omega}$. If this is the case, we will say that $Q_{1\Omega}$ and $Q_{3\Omega}$ are \textit{linked}. Otherwise, we will say that $Q_{1\Omega}$ and $Q_{3\Omega}$ are \textit{unlinked}. In either case, dimension considerations imply that neither $Q_{1\Omega}$ nor $Q_{3\Omega}$ is linked to $Q_{2\Omega}$. As in our analysis of type~I algebras, it will be important to distinguish between these settings. 

The following lemma concerns the independence of the blocks in the radical of an algebra in reduced block upper triangular form, and will play a key role in our study of type~II algebras.

\begin{lem}\label{middle block unlinked implies good radical lem}

	Let $n$ be a positive integer, and let $\mc{A}$ be a unital subalgebra of $\mm_n$ in reduced block upper triangular form with respect to a decomposition $\bigoplus_{i=1}^m \mc{V}_i$ of $\cc^n$. Suppose that there is an index $k$, $1<k<m$, that is unlinked from all indices $i\neq k$. Let $Q_1, Q_2,$ and $Q_3$ denote the orthogonal projections onto $\bigoplus_{i<k}\mc{V}_i$, $\mc{V}_k$, and $\bigoplus_{i>k}\mc{V}_i$, respectively, and assume that $Q_1\rad Q_1=Q_3\rad Q_3=\{0\}.$
	
	\begin{itemize}
	
		\item[(i)]For every $R\in\rad$, there are elements $R^\prime=Q_1R^\prime$ and $R^{\prime\prime}=R^{\prime\prime}Q_3$ in $\rad$ such that $$R^\prime Q_2=Q_1R Q_2\,\,\,\,\,\text{and}\,\,\,\,\,Q_2R^{\prime\prime}=Q_2RQ_3.$$
		\item[(ii)]If there exist projections $Q_1^\prime\leq Q_1$ and $Q_3^\prime\leq Q_3$ such that $$Q_1\rad Q_2=Q_1^\prime\mm_nQ_2,\,\,\,\,\,Q_2Rad(\mc{A})Q_3=Q_2\mm_nQ_3^\prime,$$ and $$Q_1\rad Q_3=Q_1^\prime\rad Q_3^\prime$$
then $$Rad(\mc{A})=Q_1^\prime\mm_n Q_2\dotplus Q_1^\prime\mm_n Q_3^\prime\dotplus Q_2\mm_nQ_3^\prime.$$

	\end{itemize}

\end{lem}

\begin{proof}
For (i), let $R$ belong to $\rad$. Since $\mc{V}_k$ is unlinked from all other spaces $\mc{V}_i$, there is an element $A\in\mc{A}$ such that $Q_1 AQ_1=Q_3AQ_3=0$ and $Q_2AQ_2=Q_2$. Thus, with respect to the decomposition $\cc^n=\ran(Q_1)\oplus\ran(Q_2)\oplus\ran(Q_3)$, $A$ and $R$ may be expressed as 
	$$A=\begin{bmatrix}
	0 & A_{12} & A_{13}\\
	0 & I & A_{23}\\
	0 & 0 & 0
	\end{bmatrix}\,\,\,\,\text{and}\,\,\,\,R=\begin{bmatrix}
	0 & R_{12} & R_{13}\\
	0 & 0 & R_{23}\\
	0 & 0 & 0
	\end{bmatrix}\smallskip$$
	for some $A_{ij}$ and $R_{ij}$. It is then easy to see that $R^\prime\coloneqq RA$ and $R^{\prime\prime}\coloneqq AR$ define elements of $\rad$ that satisfy the requirements of (i).
	
	For (ii), let $M_1$ and $M_2$ denote arbitrary elements of $Q_1^\prime\mm_nQ_2$ and $Q_2\mm_nQ_3^\prime,$ respectively. By (i), there are elements $S_1$ and $S_2$ in $Q_1\mm_n Q_3$ such that $M_1+S_1$ and $M_2+S_2$ belong to $\rad$. Moreover, since $Q_1\rad Q_3=Q_1^\prime\rad Q_3^\prime$, we have that $S_1$ and $S_2$ are contained in $Q_1^\prime\mm_n Q_3^\prime$. 
	
	Observe that $R\coloneqq (M_1+S_1)(M_2+S_2)$ belongs to $\rad$. With respect to the decomposition of $\cc^n$ described above, this element can be expressed as 
	$$R=\begin{bmatrix}
	0 & M_1 & S_1\\
	0 & 0 & 0\\
	0 & 0 & 0
	\end{bmatrix}\begin{bmatrix}
	0 & 0 & S_2\\
	0 & 0 & M_2\\
	0 & 0 & 0
	\end{bmatrix}=\begin{bmatrix}
	0 & 0 & M_1M_2\\
	0 & 0 & 0\\
	0 & 0 & 0
	\end{bmatrix}.\smallskip$$
	But since $M_1$ and $M_2$ were arbitrary, this implies that $Q_1^\prime\mm_nQ_3^\prime\subseteq\rad.$ In particular, $\rad$ contains $S_1$ and $S_2$. It then follows that $M_1$ and $M_2$ belong to $\rad$ as well. We conclude that $\rad$ contains $Q_1^\prime\mm_n Q_2$ and $Q_2\mm_n Q_3^\prime$, as $M_1$ and $M_2$ were arbitrary.\end{proof}
\smallskip

Notably, if $\mc{A}$ is a type~II algebra for which the triples $\Omega=(d,k,\bigoplus_{i=1}^m\mc{V}_i)$ in $\mc{F}_{II}$ are such that $1<k<m$, then $Q_{2\Omega}$ is necessarily unlinked from $Q_{1\Omega}$ and $Q_{3\Omega}$, and hence $\mc{A}$ satisfies the assumptions of Lemma~\ref{middle block unlinked implies good radical lem}.

\subsection{Type II Algebras with Unlinked Projections}\label{Subsection: Type II algebras unlinked}
	
	Let us first consider the type II algebras $\mc{A}$ for which the triples $\Omega=(d,k,\bigoplus_{i=1}^m\mc{V}_i)$ in $\mc{F}_{II}(\mc{A})$ are such that $Q_{1\Omega}$ and $Q_{3\Omega}$ are unlinked. We aim to show that the only such algebras with the projection compression property are those that are unitarily equivalent to the type II algebra in Example~\ref{exmp:families of compressible algebras:1}. To accomplish this goal, we will first show in Lemma~\ref{case II unlinked algebras base case} that the result holds in the $\mm_4$ setting. An extension to larger type II algebras will be made in Theorem~\ref{big thm for unlinked type II algebras} by applying Lemma~\ref{case II unlinked algebras base case} to their $4\times 4$ compressions.


\begin{lem}\label{case II unlinked algebras base case}
Let $\mc{A}$ be a projection compressible type II subalgebra of $\mm_4$. Assume that $\mc{F}_{II}(\mc{A})$ contains a pair $\Omega=(d,k,\bigoplus_{i=1}^3\mc{V}_i)$ such that $d=k=2$. If $Q_{1\Omega}$ and $Q_{3\Omega}$ are unlinked, then $\mc{A}$ is unitarily equivalent to 
$$\cc Q_{1\Omega}+\cc Q_{3\Omega}+(Q_{1\Omega}+Q_{2\Omega})\mm_4(Q_{2\Omega}+Q_{3\Omega}),\smallskip$$ the type II algebra from Example~\ref{exmp:families of compressible algebras:1}. Consequently, $\mc{A}$ is idempotent compressible.

\end{lem}
	
	\begin{proof}
	
		Suppose to the contrary that $\mc{A}$ is not unitarily equivalent to the algebra described above. Lemma~\ref{middle block unlinked implies good radical lem} (ii) then implies that $$\begin{array}{rl}
		Q_{1\Omega}\rad Q_{2\Omega}\neq Q_{1\Omega}\mm_4 Q_{2\Omega}\phantom{.} & \text{or}\vspace{0.2cm}\\
		Q_{2\Omega}\rad Q_{3\Omega}\neq Q_{2\Omega}\mm_4 Q_{3\Omega}.\end{array}\smallskip$$
		\noindent By replacing $\mc{A}$ with $\mc{A}^{aT}$ if necessary, we may assume that $Q_{1\Omega}\rad Q_{2\Omega}\neq Q_{1\Omega}\mm_4 Q_{2\Omega}$. Consequently, $Q_{1\Omega}\rad Q_{2\Omega}=\{0\}$ by Theorem~\ref{structure of modules over Mn}.
		
		An application of Theorem~\ref{every algebra is similar to an unhinged algebra} provides a precise description of $Q_{1\Omega}\mc{A}Q_{2\Omega}$. Since $\mc{A}$ is similar to $BD(\mc{A})\dotplus\rad$ via a block upper triangular similarity, there is a fixed element $T\in Q_{1\Omega}\mm_4 Q_{2\Omega}$ such that 
		$$Q_{1\Omega}AQ_{2\Omega}=(Q_{1\Omega}AQ_{1\Omega})T-T(Q_{2\Omega}AQ_{2\Omega})\,\,\,\text{for every}\,\,A\in\mc{A}.\smallskip$$
		For each $i\in\{1,2,3\}$, fix an orthonormal basis $\left\{e_1^{(i)},e_2^{(i)},\ldots, e_{n_{i\Omega}}^{(i)}\right\}$ for $\ran(Q_{i\Omega})$. To simplify matters, we may use Theorem~\ref{svd lem} and the remarks that follow to assume that $\langle Te_2^{(2)},e_1^{(1)}\rangle=0.$ That is, with respect to the basis $\left\{e_1^{(1)},e_1^{(2)},e_2^{(2)},e_1^{(3)}\right\}$ for $\cc^4$, each $A\in\mc{A}$ may be expressed as
		$$A=\left[\begin{array}{c|cc|c}
		a_{11} & a_{11}t-ta_{22} & -ta_{23} & a_{14}\\ \hline
		& a_{22} & a_{23} & a_{24}\\
		& a_{32} & a_{33} & a_{34}\\ \hline
		& & & a_{44}
		\end{array}\right],\smallskip$$
where $a_{ij}\in\cc$ and $t\coloneqq \langle Te_1^{(2)},e_1^{(1)}\rangle$. Here, the entries on the block-diagonal may be selected arbitrarily.

To reach a contradiction, consider the matrices 
$$P_0\coloneqq \begin{bmatrix}
2 & 0 & 0 & 0\\
0 & 1 & 0 & 1\\
0& 0 & 2 & 0\\
0& 1 & 0 & 1
\end{bmatrix}\,,\,\,\,P_1\coloneqq \begin{bmatrix}
\phantom{-}1 & 0 & 0 & -1\\
\phantom{-}0 & 2 & 0 & \phantom{-}0\\
\phantom{-}0 & 0 & 2 & \phantom{-}0\\
-1 & 0 & 0 & \phantom{-}1
\end{bmatrix}\,,\,\,\,\text{and}\,\,\,P_2\coloneqq \begin{bmatrix}
1 & 0 & 0 & 1\\
0 & 2 & 0 & 0\\
0 & 0 & 2 & 0\\
1 & 0 & 0 & 1
\end{bmatrix}.\smallskip$$
Observe that for each $i$, $\frac{1}{2}P_i$ is a projection in $\mm_4$. Through direct computation, one may verify that $$\langle Be_2^{(2)},e_1^{(1)}\rangle+2t\langle Be_2^{(2)},e_1^{(2)}\rangle=0\,\,\text{for all}\,\,B\in P_0\mc{A}P_0.\smallskip$$ Yet with $A$ as above and $B_0\coloneqq (P_0AP_0)^2$, we have
$$\langle B_0 e_2^{(2)},e_1^{(1)}\rangle +2t\langle B_0 e_2^{(2)},e_1^{(2)}\rangle=8a_{23}\left(a_{14}-t(a_{11}-a_{44}-a_{24})\right).\smallskip$$ It follows that $a_{23}=0$ for all $A\in\mc{A}$ or $a_{14}=t(a_{11}-a_{44}-a_{24})$ for all $A\in\mc{A}$. Indeed, it is clear that every element of $A$ must satisfying at least one of these equations. But if $\mc{A}$ contained an operator $A_1$ satisfying only the first equation and an operator $A_2$ satisfying only the second, then neither equation would hold for $A_1+A_2$. Since $a_{23}$ may be selected arbitrarily, it must be that $a_{14}=t(a_{11}-a_{44}-a_{24})$ for every $A\in\mc{A}$. 

One may now derive similar relations using $P_1$ and $P_2$. Indeed, it is straightforward to check that for $j\in\{1,2\}$, the equation
$$t\langle P_jAP_je_2^{(2)},e_1^{(2)}\rangle+2\langle P_jAP_j e_2^{(2)},e_1^{(1)}\rangle=0\smallskip$$ holds for every $A\in\mc{A}$. Yet if $A_0$ denotes any element of $\mc{A}$ of the above form satisfying $a_{11}=a_{23}=1$ and $a_{44}=0$, then for $B_j\coloneqq (P_jA_0P_j)^2$,
$$
\left(t\langle B_1e_2^{(2)},e_1^{(2)}\rangle+\right.\left.2\langle B_1 e_2^{(2)},e_1^{(1)}\rangle\right)-\left(t\langle B_2e_2^{(2)},e_1^{(2)}\rangle +2\langle B_2 e_2^{(2)},e_1^{(1)}\rangle\right)=16t^2.\smallskip
$$
Since $B_1$ and $B_2$ belong to $P_1\mc{A}P_1$ and $P_2\mc{A}P_2$, respectively, we conclude  that $t=0$. That is, $Q_{1\Omega}\mc{A}Q_{2\Omega}=\{0\}$. It follows that with respect to the basis $\left\{e_1^{(2)},e_2^{(2)},e_1^{(1)},e_1^{(3)}\right\}$ for $\cc^4$, each $A\in\mc{A}$ may be written as 
 $$A=\left[\begin{array}{cc|c|c}
 a_{22} & a_{23} & 0 & a_{24}\\
 a_{32} & a_{33} & 0 & a_{34}\\ \hline
 & & a_{11} & 0\\ \hline
 & & & a_{44}
 \end{array}\right]\smallskip$$ for some $a_{ij}\in\cc$. Theorem~\ref{At most one non-scalar corner theorem} now demonstrates that $\mc{A}$ is not projection compressible, as the entries in $BD(\mc{A})$ may be chosen arbitrarily. This is a contradiction.
	\end{proof}
\smallskip		
	
\begin{thm}\label{big thm for unlinked type II algebras}

	Let $\mc{A}$ be a projection compressible type II subalgebra of $\mm_n$, and assume that there is a triple $\Omega=(d,k,\bigoplus_{i=1}^m\mc{V}_i)$ in $\mc{F}_{II}(\mc{A})$ with $1<k<m$. If $Q_{1\Omega}$ and $Q_{3\Omega}$ are unlinked, then $\mc{A}$ is unitarily equivalent to
	$$\cc Q_{1\Omega}+\cc Q_{3\Omega}+(Q_{1\Omega}+Q_{2\Omega})\mm_n(Q_{2\Omega}+Q_{3\Omega}),\smallskip$$
	the type II algebra from Example~\ref{exmp:families of compressible algebras:1}. Consequently, $\mc{A}$ is idempotent compressible.

\end{thm}

\begin{proof}
	Suppose to the contrary that $\mc{A}$ is not unitarily equivalent to the algebra described above. As in the proof of the previous result, we may appeal to Lemma~\ref{middle block unlinked implies good radical lem} (ii) and assume without loss of generality that $Q_{1\Omega}\rad Q_{2\Omega}\neq Q_{1\Omega}\mm_n Q_{2\Omega}.$ Thus, Theorem~\ref{structure of modules over Mn} gives rise to a proper subprojection $Q_1^\prime$ of $Q_{1\Omega}$ satisfying $$Q_{1\Omega}\rad Q_{2\Omega}=Q_1^\prime \mm_nQ_{2\Omega}.$$ 
	
	Define $Q_1^{\prime\prime}\coloneqq Q_{1\Omega}-Q_1^\prime$, and let $\left\{e_1^{(1)},e_2^{(1)},\ldots,e_{n_{1\Omega}}^{(1)}\right\}$ be an orthonormal basis for $\ran(Q_{1\Omega})$ such that 
	$$\ran(Q_1^{\prime\prime})=\vee\left\{e_1^{(1)},e_2^{(1)},\ldots, e_{\ell}^{(1)}\right\}\smallskip$$ 
	for some integer $1\leq\ell\leq n_{1\Omega}$. Since $\mc{A}$ is similar to $BD(\mc{A})\dotplus\rad$ via a matrix that is block upper triangular with respect to $\cc^n=\ran(Q_{1\Omega})\oplus\ran(Q_{2\Omega})\oplus\ran(Q_{3\Omega})$, there is an operator $T\in Q_1^{\prime\prime}\mm_nQ_{2\Omega}$ such that $$\begin{array}{cc}Q_1^{\prime\prime}AQ_{2\Omega}=(Q_1^{\prime\prime}AQ_1^{\prime\prime})T-T(Q_{2\Omega}AQ_{2\Omega}) & \text{for all}\,\,A\in\mc{A}.\end{array}
	\smallskip$$
	By Theorem~\ref{svd lem}, one may choose a suitable orthonormal basis $\left\{e_1^{(2)}, e_2^{(2)}, \ldots, e_{n_{2\Omega}}^{(2)}\right\}$ for $\ran(Q_{2\Omega})$ and adjust the basis for $\ran(Q_1^{\prime\prime})$ if necessary to impose additional structure on $T$. Specifically, one may assume that $\langle Te_j^{(2)},e_i^{(1)}\rangle=0$ whenever $i\neq j$.

	 Let $e_1^{(3)}$ be any non-zero vector in $\ran(Q_{3\Omega})$, and define $\mc{B}=\left\{e_1^{(1)},e_1^{(2)},e_2^{(2)},e_1^{(3)}\right\}.$ Let $P$ denote the orthogonal projection onto the span of $\mc{B}$, and consider the compression $\mc{A}_0\coloneqq P\mc{A}P$. It is easy to see that $\mc{A}_0$ is a projection compressible type II subalgebra of $\mm_4$. Moreover, if $$\begin{array}{cccc}\mc{W}_1\coloneqq \cc e_1^{(1)}, & \mc{W}_2\coloneqq \vee\left\{e_1^{(2)},e_2^{(2)}\right\},& \text{and} & \mc{W}_3\coloneqq \cc e_1^{(3)},\end{array}\smallskip$$ then the triple $\Omega^\prime=(2,2,\bigoplus_{i=1}^3\mc{W}_i)$ belongs to $\mc{F}_{II}(\mc{A}_0)$. Since $Q_{1\Omega^\prime}$ and $Q_{3\Omega^\prime}$ are unlinked, $\mc{A}_0$ is among the class of algebras addressed in Lemma~\ref{case II unlinked algebras base case}. With respect to the basis $\mc{B}$ for $\ran(P)$, however, every element of $\mc{A}_0$ may be expressed as a matrix of the form 
	 $$A=\left[\begin{array}{c|cc|c}
	 a_{11} & a_{11}t-ta_{22} & -ta_{23} & a_{14}\\ \hline
	 & a_{22} & a_{23} & a_{24}\\
	 & a_{32} & a_{33} & a_{34}\\ \hline
	 & & & a_{44}
	 \end{array}\right],$$
	 where $t\coloneqq \langle Te_1^{(2)},e_1^{(1)}\rangle$. Since $\mc{A}_0$ is not of the form prescribed by Lemma~\ref{case II unlinked algebras base case}, it follows that $\mc{A}_0$ is not projection compressible---a contradiction.
\end{proof}
\smallskip

\subsection{Type II Algebras with Linked Projections}\label{Subsection: Type II algebras linked}
Consider now the type II algebras $\mc{A}$ for which the triples $\Omega=(d,k,\bigoplus_{i=1}^m\mc{V}_i)$ in $\mc{F}_{II}(\mc{A})$ are such that $Q_{1\Omega}$ and $Q_{3\Omega}$ are linked. It will be shown in Theorem~\ref{case II - linked case implies LR theorem} that all projection compressible algebras of this form are unitizations of $\mc{LR}$-algebras. The proof of this result requires a careful analysis of the upper triangular blocks in the semi-simple part of the algebra. The following lemma is the crux of this analysis.

\begin{lem}\label{extending incomplete radical lemma}

Let $\mc{A}$ be a projection compressible type II subalgebra of $\mm_4$. Assume that $\mc{F}_{II}(\mc{A})$ contains a triple $\Omega=(d,k,\bigoplus_{i=1}^3\mc{V}_i)$ with $d=k=2$, and such that $Q_{1\Omega}$ and $Q_{3\Omega}$ are linked. 
\begin{itemize}
\item[(i)] If there are a constant $t\in\cc$ and for each $i\in\{1,2,3\}$, an orthonormal basis $\left\{e_1^{(i)},e_2^{(i)},\ldots, e_{n_{i\Omega}}^{(i)}\right\}$ for $\ran(Q_{i\Omega})$ such that
$$\begin{array}{ll}\langle Ae_1^{(2)},e_1^{(1)}\rangle=\phantom{-}t\left(\langle Ae_1^{(1)},e_1^{(1)}\rangle-\langle Ae_1^{(2)},e_1^{(2)}\rangle\right) & and \vspace{0.1cm} \\
\langle Ae_2^{(2)},e_1^{(1)}\rangle= -t\langle Ae_2^{(2)},e_1^{(2)}\rangle\end{array}\smallskip$$
 for all $A\in\mc{A}$, then $\langle Ae_1^{(3)},e_1^{(1)}\rangle=-t\langle Ae_1^{(3)},e_1^{(2)}\rangle$ for all $A\in\mc{A}.$ \smallskip

\item[(ii)] If there are a constant $t\in\cc$ and for each $i\in\{1,2,3\}$, an orthonormal basis $\left\{e_1^{(i)},e_2^{(i)},\ldots, e_{n_{i\Omega}}^{(i)}\right\}$ for $\ran(Q_{i\Omega})$ such that
$$\begin{array}{ll}\langle Ae_1^{(3)},e_1^{(2)}\rangle=t\left(\langle Ae_1^{(2)},e_1^{(2)}\rangle-\langle Ae_1^{(3)},e_1^{(3)}\rangle\right) & and\phantom{-} \vspace{0.1cm} \\
\langle Ae_1^{(3)},e_2^{(2)}\rangle= t\langle Ae_1^{(2)},e_2^{(2)}\rangle\end{array}\smallskip$$
 for all $A\in\mc{A}$, then $\langle Ae_1^{(3)},e_1^{(1)}\rangle=t\langle Ae_1^{(2)},e_1^{(1)}\rangle$ for all $A\in\mc{A}.$
 
 \end{itemize}

\end{lem}

\begin{proof}

	We will begin with the proof of (i). Suppose that there are a constant $t\in\cc$ and for each $i\in\{1,2,3\}$, an orthonormal basis $\left\{e_1^{(i)},e_2^{(i)},\ldots, e_{n_{i\Omega}}^{(i)}\right\}$ for $\ran(Q_{i\Omega})$ as described above. Then with respect to the basis $\left\{e_1^{(1)},e_1^{(2)},e_2^{(2)},e_1^{(3)}\right\}$ for $\cc^4$, each $A\in\mc{A}$ can be written as 
	$$A=\left[\begin{array}{c|cc|c}
	a_{11} & t(a_{11}-a_{22}) & -ta_{23} & a_{14}\\ \hline
	& a_{22} & a_{23} & a_{24}\\
	& a_{32} & a_{33} & a_{34}\\ \hline
	& & & a_{11}
	\end{array}\right]$$
	for some $a_{ij}\in\cc$. Since $\mc{A}$ is in reduced block upper triangular form, the entries on the block-diagonal may be chosen arbitrarily.
	
	Consider the matrix 
	$$P\coloneqq \begin{bmatrix}
	1 & 0 & 0 & 1\\
	0 & 2 & 0 & 0\\
	0 & 0 & 2 & 0\\
	1 & 0 & 0 & 1
	\end{bmatrix},\smallskip$$
	and note that $\frac{1}{2}P$ is a projection in $\mm_4$. One may verify that
	$$t\langle Be_2^{(2)},e_2^{(1)}\rangle+2\langle Be_2^{(2)},e_1^{(1)}\rangle=0\,\,\,\text{for all}\,\,\,B\in P\mc{A}P.\smallskip$$ But with $A$ as above and $B\coloneqq (PAP)^2$, we see that $$t\langle Be_2^{(2)},e_1^{(2)}\rangle+2\langle Be_2^{(2)},e_1^{(1)}\rangle=-8ta_{23}(a_{14}+ta_{24}).\smallskip$$
	The projection compressibility of $\mc{A}$ implies that $B$ belongs to $P\mc{A}P$. Consequently, $ta_{23}(a_{14}+ta_{24})=0$ for all $A\in\mc{A}.$
	
	If $t\neq 0$, then either $a_{23}=0$ for all $A\in\mc{A}$ or $a_{14}=-ta_{24}$ for all $A\in\mc{A}$. Indeed, it is clear that every operator in $\mc{A}$ must satisfy at least one of these equation. If, however, $\mc{A}$ contained an operator $A_1$ satisfying the first equation but not the second, as well as an operator $A_2$ satisfying the second but not the first, then neither equation would hold for $A_1+A_2$. Finally, since $a_{23}$ can be selected arbitrarily, we conclude that either $t=0$ or $a_{14}=-ta_{24}$ for all $A$.

If the latter holds, then every $A\in\mc{A}$ satisfies the equation $\langle Ae_1^{(3)},e_1^{(1)}\rangle=-t\langle Ae_1^{(3)},e_1^{(2)}\rangle,$ as required. If instead $t=0$, then with respect to the basis $\left\{e_1^{(2)},e_2^{(2)},e_1^{(1)},e_1^{(3)}\right\}$ for $\cc^4$, each $A\in\mc{A}$ may be expressed as a matrix of the form 
	$$A=\left[\begin{array}{cc|c|c}
	a_{22} & a_{23} & 0 & a_{24}\\
	a_{32} & a_{33} & 0 & a_{34}\\ \hline
	& & a_{11} & a_{14}\\ \hline
	& & & a_{11}
	\end{array}\right]$$
	for some $a_{ij}\in\cc$.
	It follows from Theorem~\ref{At most one non-scalar corner theorem} that $a_{14}=\langle Ae_1^{(3)},e_1^{(1)}\rangle=0$ for all $A$, and hence the equation $\langle Ae_1^{(3)},e_1^{(1)}\rangle=-t\langle Ae_1^{(3)},e_1^{(2)}\rangle$ holds in this case as well.
	
	In the context of (ii), note that every $A\in\mc{A}$ may be expressed as a matrix of the form 
	$$\left[\begin{array}{c|cc|c}
	a_{11} & a_{12} & a_{13} & a_{14}\\ \hline
	& a_{22} & a_{23} & t(a_{22}-a_{11})\\
	& a_{32} & a_{33} & ta_{32}\\ \hline
	& & & a_{11}
	\end{array}\right]\smallskip$$ with respect to the basis $\left\{e_1^{(1)},e_1^{(2)},e_2^{(2)},e_1^{(3)}\right\}$ for $\cc^4$. Since this matrix is transpose equivalent to 
	$$\left[\begin{array}{c|cc|c}
	a_{11} & t(a_{22}-a_{11}) & ta_{32} & a_{14}\\ \hline
	& a_{22} & a_{32} & a_{12}\\
	& a_{23} & a_{33} & a_{13}\\ \hline
	& & & a_{11}
	\end{array}\right],\smallskip$$
	we conclude from (i) that $a_{14}=ta_{12}$. That is, $\langle A e_1^{(3)},e_1^{(1)}\rangle=t\langle Ae_1^{(2)},e_1^{(1)}\rangle$ for all $A\in\mc{A}$.
\end{proof}
\smallskip

\begin{thm}\label{case II - linked case implies LR theorem}

	Let $\mc{A}$ be a projection compressible type II subalgebra of $\mm_n$, and let $\Omega=(d,k,\bigoplus_{i=1}^m\mc{V}_i)$ be a triple in $\mc{F}_{II}(\mc{A})$. If $Q_{1\Omega}$ and $Q_{3\Omega}$ are linked, then $\mc{A}$ is the unitization of an $\mc{LR}$-algebra. Consequently, $\mc{A}$ is idempotent compressible.

\end{thm}

\begin{proof}
Let $\Omega$ be as above, and assume that $Q_{1\Omega}$ and $Q_{3\Omega}$ are linked. Note that if $k=1$ or $k=m$, then $\mc{A}$ is the unitization of an $\mc{LR}$-algebra by Proposition~\ref{case2 - k=1 or k=m prop}. Thus, we will assume that $1<k<m$. In this case, Theorem~\ref{structure of modules over Mn} gives rise to subprojections $Q_1^\prime\leq Q_{1\Omega}$ and $Q_3^\prime\leq Q_{3\Omega}$ such that $$\begin{array}{rl}Q_{1\Omega}\rad Q_{2\Omega}=Q_1^\prime \mm_n Q_{2\Omega}\phantom{.} & \text{and}\vspace{0.2cm} \\
Q_{2\Omega}\rad Q_{3\Omega}=Q_{2\Omega}\mm_n Q_3^\prime.\end{array}\smallskip$$
Our goal is to show that $\mc{A}$ is similar to $$\mc{A}_0\coloneqq (Q_1^\prime+Q_{2\Omega})\mm_n(Q_{2\Omega}+Q_3^\prime)+\cc I.\smallskip$$ Since $\mc{A}_0$ is the unitization of an $\mc{LR}$-algebra, this will demonstrate that so too is $\mc{A}$. We will accomplish this task by first determining the structure of $Q_{1\Omega}\mc{A}Q_{3\Omega}$.

Define $Q_1^{\prime\prime}\coloneqq Q_{1\Omega}-Q_{1}^\prime$ and $Q_3^{\prime\prime}\coloneqq Q_{3\Omega}-Q_3^\prime$. For each $i\in\{1,2,3\}$, let $\left\{e_1^{(i)},e_2^{(i)},\ldots, e_{n_{i\Omega}}^{(i)}\right\}$ be an orthonormal basis for $\ran(Q_{i\Omega})$ such that if $Q_i^{\prime\prime}\neq 0$, then
$$\ran(Q_i^{\prime\prime})=\vee\left\{e_{1}^{(i)},e_2^{(i)},\ldots, e_{\ell_i}^{(i)}\right\}\smallskip$$ for some $\ell_i\in\{1,2,\ldots, n_{i\Omega}\}$. 
Since $\mc{A}$ is similar to $BD(\mc{A})\dotplus\rad$ via a matrix that is block upper triangular with respect to $\cc^n=\ran(Q_{1\Omega})\oplus\ran(Q_{2\Omega})\oplus\ran(Q_{3\Omega})$, there are operators $T_1\in Q_1^{\prime\prime}\mm_nQ_{2\Omega}$ and $T_2\in Q_{2\Omega}\mm_nQ_3^{\prime\prime}$ such that every $A\in\mc{A}$ satisfies 
$$\begin{array}{rl}Q_1^{\prime\prime}AQ_{2\Omega}=(Q_1^{\prime\prime}AQ_1^{\prime\prime})T_1-T_1(Q_{2\Omega}AQ_{2\Omega})\phantom{.} & \text{and}\vspace{0.2cm} \\
Q_{2\Omega}AQ_3^{\prime\prime}=(Q_{2\Omega}AQ_{2\Omega})T_2-T_2(Q_3^{\prime\prime}AQ_3^{\prime\prime}).\end{array}\smallskip$$ 

We will begin by using Lemma~\ref{extending incomplete radical lemma} to identify the structure of $Q_{1}^{\prime\prime}\mc{A}Q_{3\Omega}$. Of course, there is little to be said when $Q_1^{\prime\prime}=0$, so assume for now that $Q_1^{\prime\prime}\neq 0$. By Theorem~\ref{svd lem} and its subsequent remarks, one may change the orthonormal bases for $\ran(Q_1^{\prime\prime})$ and $\ran(Q_{2\Omega})$ if required and assume that $$t_{ij}^{(1)}\coloneqq \langle T_1e_j^{(2)},e_i^{(1)}\rangle=0\,\,\,\text{for all}\,\,i\neq j.$$

Let $i$ and $i^\prime$ be arbitrary indices from $\{1,2,\ldots,\ell_1\}$ and $\{1,2,\ldots, n_{3\Omega}\}$, respectively. Define $$j=\left\{\begin{array}{ll}
i & \text{if}\,\,i\leq n_{2\Omega},\\
1 & \text{otherwise},\end{array}\right.\smallskip$$ and fix an index $j^\prime\in\{1,2,\ldots,n_{2\Omega}\}\setminus\{j\}$. Let $P$ denote the orthogonal projection onto the span of $\mc{B}\coloneqq \left\{e_{i}^{(1)},e_{j}^{(2)},e_{j^\prime}^{(2)},e_{i^\prime}^{(3)}\right\}$, and consider the algebra $P\mc{A}P$. If $i>n_{2\Omega}$, then for each $A\in\mc{A}$, $PAP$ may be expressed as a matrix of the form 
$$PAP=\left[\begin{array}{c|cc|c}
a_{11} & 0 & 0 & a_{14}\\  \hline
& a_{22} & a_{23} & a_{24}\\
& a_{32} & a_{33} & a_{34}\\ \hline
& & & a_{11}
\end{array}\right]\smallskip$$ with respect to $\mc{B}$. In this case, $P\mc{A}P$ is an algebra of the form described in Lemma~\ref{extending incomplete radical lemma}~(i) with $t=0$. Thus, this result implies that $$a_{14}=\langle Ae_{i^\prime}^{(3)},e_i^{(1)}\rangle=0\,\,\,\text{for all}\,\,A\in\mc{A}.\smallskip$$ Suppose instead that $i\leq n_{2\Omega}$. We then have that for each $A\in\mc{A}$, $PAP$ can be written as a matrix of the form
$$PAP=\left[\begin{array}{c|cc|c}
a_{11} & a_{11}t_{ii}^{(1)}-t_{ii}^{(1)}a_{22} & -t_{ii}^{(1)}a_{23} & a_{14}\\  \hline
& a_{22} & a_{23} & a_{24}\\
& a_{32} & a_{33} & a_{34}\\ \hline
& & & a_{11}
\end{array}\right]$$ with respect to $\mc{B}$.
It follows that $P\mc{A}P$ is of the form described in Lemma~\ref{extending incomplete radical lemma}~(i) with $t=t_{ii}^{(1)}$, and hence $$a_{14}=\langle Ae_{i^\prime}^{(3)},e_i^{(1)}\rangle=-t_{ii}^{(1)}\langle Ae_{i^\prime}^{(3)},e_i^{(2)}\rangle\,\,\,\text{for all}\,\,A\in\mc{A}.\smallskip$$
Since our choice of indices was arbitrary, these conclusions hold for all $i\in\{1,2,\ldots, \ell_1\}$ and all ${i^\prime}\in\{1,2,\ldots, n_{3\Omega}\}$. Consequently,
$$Q_1^{\prime\prime}AQ_{3\Omega}=-T_1Q_{2\Omega}AQ_{3\Omega}\,\,\,\text{for all}\,\,A\in\mc{A}.\smallskip$$

	We now wish to obtain information on the structure of $Q_{1\Omega}\mc{A}Q_3^{\prime\prime}$. As in the analysis above, it will be convenient to simplify the description of $T_2$ by choosing suitable bases for $\ran(Q_{2\Omega})$ and $\ran(Q_{3\Omega})$. Specifically, Theorem~\ref{svd lem} gives rise to operators $V\in Q_{2\Omega}\mm_nQ_{2\Omega}$, $W\in Q_3^{\prime\prime}\mm_nQ_3^{\prime\prime}$, and a unitary $U\in\mm_n$ such that 
	 $$\begin{array}{l}
	 (Q_{1\Omega}+Q_3^\prime)U(Q_{1\Omega}+Q_3^\prime)=Q_{1\Omega}+Q_3^\prime,\vspace{0.3cm}\\ (Q_{2\Omega}+Q_3^{\prime\prime})U(Q_{2\Omega}+Q_3^{\prime\prime})=V+W,
	 \end{array}\smallskip$$ 
	 and $$\langle U^*T_2Ue_j^{(3)},e_i^{(2)}\rangle=\langle V^*T_2W e_j^{(3)},e_i^{(2)}\rangle=0\,\,\,\text{for all}\,\,i\neq j.\smallskip$$
By considering the algebra $U^*\mc{A}U$ and arguing as above, one may deduce that 
	$$(Q_{1\Omega}AQ_3^{\prime\prime})=(Q_{1\Omega}AQ_{2\Omega})T_2\,\,\,\text{for all}\,\,A\in\mc{A}.\smallskip$$
	
Our findings thus far indicate that with respect to the decomposition $$\cc^n=\ran(Q_1^{\prime\prime})\oplus\ran(Q_1^\prime)\oplus\ran(Q_{2\Omega})\oplus\ran(Q_3^{\prime\prime})\oplus\ran(Q_3^\prime),\smallskip$$ each $A\in\mc{A}$ can be expressed as a matrix of the form $$A=\left[\begin{array}{c|c|c|c|c}
	 a_{11}I & 0 & a_{11}T_1-T_1M & -T_1(MT_2-a_{11}T_2) & -T_1J_2\\ \hline
	 & a_{11} I & J_1 & J_1T_2 & A_{25}\\ \hline
	 & & M & MT_2-a_{11}T_2 & J_2\\ \hline
	 & & & a_{11}I & 0\\ \hline
	 & & & & a_{11}I\end{array}\right]\smallskip$$
	 for some $a_{11}\in\cc$ and operators $M\in Q_{2\Omega}\mm_n Q_{2\Omega}$, $J_1\in Q_1^\prime\rad Q_{2\Omega}$, $J_2\in Q_{2\Omega}\rad Q_3^\prime$, and $A_{25}\in Q_1^{\prime}\mm_n Q_3^\prime$. With this description in hand we are prepared to show that $\mc{A}$ is similar to $\mc{A}_0$, and hence is the unitization of an $\mc{LR}$-algebra. 
	 
	 Consider the operator $S\coloneqq I-T_1-T_2$. This $S$ is invertible with $S^{-1}=I+T_1+T_2+T_1T_2.$ Moreover, for each $A\in\mc{A}$ as above, we have that  
	$$
	S^{-1}AS=\left[\begin{array}{c|c|c|c|c}
	 a_{11}I & 0 & 0 & 0 & 0\\ \hline
	 & a_{11} I & J_1 & 0 & A_{25}\\ \hline
	 & & M & 0 & J_2\\ \hline
	 & & & a_{11}I & 0\\ \hline	 & & & & a_{11}I\end{array}\right].\smallskip$$
From here it is easy to see that $S^{-1}\mc{A}S$ is a type II algebra that has a reduced block upper triangular form with respect to the above decomposition. Moreover, 
	 $$\begin{array}{rl}Q_{1\Omega}Rad(S^{-1}\mc{A}S)Q_{2\Omega}=Q_1^\prime\rad Q_{2\Omega}=Q_1^\prime\mm_nQ_{2\Omega}\phantom{.}&\text{and}\vspace{0.2cm} \\
	 Q_{2\Omega}Rad(S^{-1}\mc{A}S)Q_{3\Omega}=Q_{2\Omega}\rad Q_3^\prime=Q_{2\Omega}\mm_nQ_3^\prime.\end{array}\smallskip$$ 

\noindent Thus, Lemma~\ref{middle block unlinked implies good radical lem} (ii) implies that 
$$S^{-1}\mc{A}S=(Q_1^\prime+Q_{2\Omega})\mm_n(Q_{2\Omega}+Q_3^\prime)+\cc I=\mc{A}_0,\smallskip$$
as claimed.
\end{proof}

\vspace{0.2cm}

\section[6]{Algebras of Type III}\label{case3}
	
	We now begin the final stage of our classification of unital projection compressible subalgebras of $\mm_n$ when $n\geq 4$. The term \textit{type III} will be used to describe a unital subalgebra $\mc{A}$ of $\mm_n$, $n\geq 4$, such that for every orthogonal decomposition $\bigoplus_{i=1}^m\mc{V}_i$ of $\cc^n$ with respect to which $\mc{A}$ is reduced block upper triangular, $\dim\mc{V}_i=1$ for all $i$ (i.e., $m=n$), and there is an integer $k$ as in Corollary~\ref{unique integer k corollary}. It is obvious that such a $k$ must lie strictly between $1$ and $n$. 
	
	As in the preceding sections, it will be important to maintain a record of the integers $k$ and decompositions of $\cc^n$ that satisfy the assumptions of Corollary~\ref{unique integer k corollary} for a given type~III algebra $\mc{A}$.
	
\begin{defn}\label{definition of F3(A)}
If $\mc{A}$ is an algebra of type III, let $\mc{F}_{III}=\mc{F}_{III}\mc{(A)}$ denote the set of pairs $\Omega=(k,\bigoplus_{i=1}^n\mc{V}_i)$ that satisfy the following conditions:

\begin{itemize}

		\item[(i)]$\bigoplus_{i=1}^n\mc{V}_i$ is an orthogonal decomposition of $\cc^n$ with respect to which $\mc{A}$ is reduced block upper triangular;\smallskip
		
		\item[(ii)]$k$ is an integer in $\{2,\ldots, n-1\}$ such that if $Q_{1\Omega}$, $Q_{2\Omega}$, and $Q_{3\Omega}$ denote the orthogonal projections onto $\bigoplus_{i<k}\mc{V}_i$, $\mc{V}_k$, and $\bigoplus_{i>k}\mc{V}_i$, respectively, then for each $i\in\{1,3\}$,
		$$(Q_{i\Omega}+Q_{2\Omega})\mc{A}(Q_{i\Omega}+Q_{2\Omega})\neq\cc(Q_{i\Omega}+Q_{2\Omega}).\smallskip$$

\end{itemize}
		
\end{defn}

\begin{notation}\upshape{
 If $\mc{A}$ is an algebra of type III and $\Omega=(k,\bigoplus_{i=1}^n\mc{V}_i)$ is a pair in $\mc{F}_{III}(\mc{A})$, let $n_{1\Omega}=k-1$, $n_{2\Omega}=1$, and $n_{3\Omega}=n-k$ denote the ranks of $Q_{1\Omega}$, $Q_{2\Omega}$, and $Q_{3\Omega}$, respectively. Note that since $n_{2\Omega}=1$ and $n\geq 4$, we necessarily have $\max\{n_{1\Omega},n_{3\Omega}\}\geq 2$.}

\end{notation}

If $\mc{A}$ is a projection compressible algebra of type III with pair $\Omega\in\mc{F}_{III}(\mc{A})$, then $Q_{i\Omega}\mc{A}Q_{i\Omega}=\cc Q_{i\Omega}$ for each $i\in\{1,2,3\}$. Thus, each corner $Q_{i\Omega}\mc{A}Q_{i\Omega}$ is a diagonal algebra comprised of mutually linked $1\times 1$ blocks. Of course, the blocks in $Q_{i\Omega}\mc{A}Q_{i\Omega}$ may or may not be linked to those in $Q_{j\Omega}\mc{A}Q_{j\Omega}$. If there is linkage between these blocks, we will say that the projections $Q_{i\Omega}$ and $Q_{j\Omega}$ are \textit{linked}; otherwise, we will say that they are \textit{unlinked}. 

Unlike in \S5, it is now entirely possible that $Q_{2\Omega}$ is linked to $Q_{1\Omega}$ or $Q_{3\Omega}$. As the following result demonstrates, however, there do not exist projection compressible algebras of type III for which all projections $Q_{i\Omega}$ are mutually linked.

\begin{prop}\label{case III - Q2 linked prop}

	Let $\mc{A}$ be a projection compressible algebra of type III, and let $\Omega$ be a pair in $\mc{F}_{III}(\mc{A})$.
	\begin{itemize}
		\item[(i)] If $Q_{2\Omega}$ is linked to $Q_{1\Omega}$, then $n_{1\Omega}=1$ and $Q_{1\Omega}\rad Q_{2\Omega}=Q_{1\Omega}\mm_n Q_{2\Omega}$
		
		\item[(ii)] If $Q_{2\Omega}$ is linked to $Q_{3\Omega}$, then $n_{3\Omega}=1$ and $Q_{2\Omega}\rad Q_{3\Omega}=Q_{2\Omega}\mm_nQ_{3\Omega}$.
		
	\end{itemize}
	
	\noindent Consequently, $Q_{2\Omega}$ cannot be linked to both $Q_{1\Omega}$ and $Q_{3\Omega}$.
\end{prop}

\begin{proof}
Clearly (ii) follows from (i) by replacing $\mc{A}$ with $\mc{A}^{aT}$. Thus, it suffices to prove (i).

Suppose to the contrary that $n_{1\Omega}\geq 2$. For each $i\in\{1,2,3\}$, let $\left\{e_1^{(i)},e_2^{(i)},\ldots, e_{n_{i\Omega}}^{(i)}\right\}$ be an orthonormal basis for $\ran(Q_{i\Omega})$. For each index $j$ in $\{1,2,\ldots, n_{3\Omega}\}$, let $P_j$ denote the orthogonal projection onto the span of $\mc{B}_j\coloneqq\left\{e_1^{(1)},e_2^{(1)},e_1^{(2)},e_j^{(3)}\right\}$. Furthermore, define $P_{j}^\prime$ to be the operator
$$P_j^\prime=\begin{bmatrix}
1 & 0 & 0 & 1\\
0 & 2 & 0 & 0\\
0 & 0 & 2 & 0\\
1 & 0 & 0 & 1
\end{bmatrix},\smallskip$$ acting on $\ran(P_j)$ and written with respect to the basis $\mc{B}_j$. It is clear that $\frac{1}{2}P_j^\prime$ is a subprojection of $P_j$. 

One may verify that every $B\in P_j^\prime\mc{A}P_j^\prime$ satisfies the equation $\langle Be_1^{(2)},e_1^{(2)}\rangle=
\langle Be_{2}^{(1)},e_{2}^{(1)}\rangle.$ 
But if $A$ belongs to $\mc{A}$ and $C\coloneqq (P_j^\prime AP_j^\prime)^2$, then  
$$\langle Ce_1^{(2)},e_1^{(2)}\rangle-\langle Ce_{2}^{(1)},e_{2}^{(1)}\rangle=8\langle Ae_1^{(2)},e_1^{(1)}\rangle\langle Ae_j^{(3)},e_1^{(2)}\rangle.\smallskip$$
Since $C$ is an element of $P_j^\prime\mc{A}P_j^\prime$, the right-hand side of this equation must be zero. To obtain a contradiction, it therefore suffices to exhibit an element $A$ in $\mc{A}$ such that for some $j\in\{1,2,\ldots, n_{3\Omega}\}$, both $\langle Ae_1^{(2)},e_1^{(1)}\rangle$ and $\langle Ae_j^{(3)},e_1^{(2)}\rangle$ are non-zero.

First suppose that the projections $Q_{1\Omega}$, $Q_{2\Omega}$, and $Q_{3\Omega}$ are mutually linked. By definition of $\Omega$ as a pair in $\mc{F}_{III}(\mc{A})$, there exist $i\in\{1,2,\ldots, n_{1\Omega}\}$ and $j\in\{1,2,\ldots, n_{3\Omega}\}$, as well as $A_1,A_2\in\mc{A}$, such that $\langle A_1e_1^{(2)},e_i^{(1)}\rangle\neq 0$ and $\langle A_2e_j^{(3)},e_1^{(2)}\rangle\neq 0$. By reordering the basis for $\ran(Q_{1\Omega})$ if necessary, we may assume that $i=1$. If $\langle A_2e_1^{(2)},e_1^{(1)}\rangle\neq 0$ or $\langle A_1e_j^{(3)},e_1^{(2)}\rangle\neq 0$, then we obtain the required contradiction. Otherwise, $A\coloneqq A_1+A_2$ is such that $\langle Ae_1^{(2)},e_1^{(1)}\rangle\neq 0$ and $\langle Ae_j^{(3)},e_1^{(2)}\rangle\neq 0,$ as desired.

Now suppose that $Q_{3\Omega}$ is unlinked from $Q_{1\Omega}$ and $Q_{2\Omega}$. By reordering the basis for $\ran(Q_{1\Omega})$ if necessary, we may obtain an element $A_1\in\mc{A}$ such that $\langle A_1e_1^{(2)},e_1^{(1)}\rangle\neq 0$. If there is an element $A_2\in\mc{A}$ such that $\langle A_2e_j^{(3)},e_1^{(2)}\rangle\neq 0$ for some $j\in\{1,2,\ldots, n_{3\Omega}\}$, then arguments similar to those in the linked case above provide the required contradiction. Of course, it is now entirely possible that no such $A_2$ exists, as $Q_{2\Omega}$ and $Q_{3\Omega}$ are unlinked. That is, it may be that $Q_{2\Omega}\mc{A}Q_{3\Omega}=\{0\}$. Assume that this is the case.

Let $\mc{B}=\left\{e_1^{(1)},e_2^{(1)},e_1^{(3)},e_1^{(2)}\right\}$, and define $P$ to be the orthogonal projection onto the span of $\mc{B}$. Note that with respect to the basis $\mc{B}$ for $\ran(P)$, each $A\in P\mc{A}P$ may be written as 
$$A=\left[\begin{array}{cc|c|c}
\alpha & 0 & a_{13} & a_{14}\\
 & \alpha & a_{23} & a_{24}\\ \hline
 & & \beta & 0\\ \hline
 & & & \alpha
\end{array}\right]$$ for some $\alpha$, $\beta$, and $a_{ij}\in\cc$. Consider the operator
$$P^\prime=\begin{bmatrix}
\phantom{-}2 & 0 & -1 & -1\\
\phantom{-}0 & 3 & \phantom{-}0 & \phantom{-}0\\
-1 & 0 & \phantom{-}2 & -1\\
-1 & 0 & -1 & \phantom{-}2
\end{bmatrix},\smallskip$$ acting on $\ran(P)$ and written with respect to $\mc{B}$. It is easy to see that $\frac{1}{3}P^\prime$ is a subprojection of $P$. Moreover, one may verify that every element $B=(b_{ij})$ in $P^\prime\mc{A}P^\prime$ satisfies the equation 
$b_{33}+2b_{31}-b_{43}-2b_{41}-b_{22}=0.$
But if $A$ is as above and we define $(P^\prime AP^\prime)^2=(c_{ij})$, then 
$$c_{33}+2c_{31}-c_{43}-2c_{41}-c_{22}=27a_{14}(\beta-\alpha).\smallskip$$
Since $\alpha$ and $\beta$ may be chosen arbitrarily, it must be that $a_{14}=\langle Ae_1^{(2)},e_1^{(1)}\rangle=0$ for all $A$. This is a contradiction, as $\langle A_1e_1^{(2)},e_1^{(1)}\rangle\neq 0$. We therefore conclude that $n_{1\Omega}=1$.

Since $Q_{1\Omega}$ and $Q_{2\Omega}$ are linked, yet $(Q_{1\Omega}+Q_{2\Omega})\mc{A}(Q_{1\Omega}+Q_{2\Omega})\neq\cc(Q_{1\Omega}+Q_{2\Omega})$ by definition of $\Omega$ as a pair in $\mc{F}_{III}(\mc{A})$, it follows that $Q_{1\Omega}\rad Q_{2\Omega}\neq \{0\}$. Consequently, $Q_{1\Omega}\rad Q_{2\Omega}=Q_{1\Omega}\mm_nQ_{2\Omega}$ as $n_{1\Omega}=n_{2\Omega}=1$.

The final claim now follows from the fact that $\max\left\{n_{1\Omega},n_{3\Omega}\right\}\geq 2$.
\end{proof}
\smallskip

The above result indicates that if $\mc{A}$ is a projection compressible algebra of type III and $\Omega$ is a pair in $\mc{F}_{III}(\mc{A})$, then there is a projection $Q_{i\Omega}$ that is unlinked from $Q_{2\Omega}$. In the case that this $Q_{i\Omega}$ is also unlinked from the remaining projection $Q_{j\Omega}$, one can say more about the structure of $\mc{A}$.

	\begin{prop}\label{case III - all unlinked, n3=1 prop}
		
	Let $\mc{A}$ be a projection compressible type III  subalgebra of $\mm_n$, and let $\Omega$ be a pair in $\mc{F}_{III}(\mc{A})$. 	
	\begin{itemize}	
	\item[(i)] If $Q_{3\Omega}$ is unlinked from $Q_{1\Omega}$ and $Q_{2\Omega}$, then either  $Q_{2\Omega}\rad Q_{3\Omega}=Q_{2\Omega}\mm_nQ_{3\Omega}$; or $n_{3\Omega}=1$ and $Q_{2\Omega}\rad Q_{3\Omega}=\{0\}$.	\smallskip
	
	\item[(ii)] If $Q_{1\Omega}$ is unlinked from $Q_{2\Omega}$ and $Q_{3\Omega}$, then either $Q_{1\Omega}\rad Q_{2\Omega}=Q_{1\Omega}\mm_nQ_{2\Omega}$; or $n_{1\Omega}=1$ and $Q_{1\Omega}\rad Q_{2\Omega}=\{0\}$.

	\end{itemize}
	
	\end{prop}
	
	\begin{proof}
	As in the previous proof it is easy that (ii) follows from (i) by replacing $\mc{A}$ with $\mc{A}^{aT}$. Thus, it suffices to prove (i). 
	
	Assume that $Q_{3\Omega}$ is unlinked from both $Q_{1\Omega}$ and $Q_{2\Omega}$. Suppose for the sake of contradiction that $n_{3\Omega}\geq 2$ and $Q_{2\Omega}\rad Q_{3\Omega}\neq Q_{2\Omega}\mm_nQ_{3\Omega}$. For each $i\in\{1,2,3\}$, let $\left\{e_1^{(i)},e_2^{(i)},\ldots, e_{n_{i\Omega}}^{(i)}\right\}$ be an orthonormal basis for $\ran(Q_{i\Omega})$, and assume that the basis for $\ran(Q_{3\Omega})$ is chosen so that $\langle Re_1^{(3)},e_1^{(2)}\rangle=0$ for all $R\in\rad$.
	
	Define $\mc{B}=\left\{e_1^{(1)},e_1^{(2)},e_1^{(3)},e_2^{(3)}\right\}$, let $P$ denote the orthogonal projection onto the span of $\mc{B}$, and consider the compression $\mc{A}_0\coloneqq P\mc{A}P$. As a consequence of Theorem~\ref{every algebra is similar to an unhinged algebra}, there is a constant $t\in\cc$ such that with respect to the basis $\mc{B}$ for $\ran(P)$, each $A$ in $\mc{A}_0$ admits a matrix of the form 
	$$A=\left[\begin{array}{c|c|cc}
	\alpha & a_{12} &  a_{13} & a_{14}\\ \hline
	& \beta & t(\beta-\gamma) & a_{24}\\ \hline
	& & \gamma & 0\\
	& & & \gamma
	\end{array}\right]$$
	for some $\alpha$, $\beta$, $\gamma$, and $a_{ij}$ in $\cc.$ Note that in the case that $Q_{1\Omega}$ and $Q_{2\Omega}$ are linked, $\alpha$ and $\beta$ must coincide for each $A\in\mc{A}_0$. In the case that they are unlinked, these values may be chosen independently. With this in mind, the following arguments are applicable to either setting.
	
	Consider the matrices  
	$$\begin{array}{ccc}P_1\coloneqq \begin{bmatrix}
	1 & 0 & 0 & 1\\
	0 & 2 & 0 & 0\\
	0 & 0 & 2 & 0\\
	1 & 0 & 0 & 1
	\end{bmatrix} & \text{and} & P_2\coloneqq \begin{bmatrix}
	1 & 0 & 1 & 0\\
	0 & 2 & 0 & 0\\
	1 & 0 & 1 & 0\\
	0 & 0 & 0 & 2
	\end{bmatrix}\end{array},\smallskip$$
	acting on $\ran(P)$ and written with respect to the basis $\mc{B}$. It is easy to see that $\frac{1}{2}P_1$ and $\frac{1}{2}P_2$ are subprojections of $P$. In addition, one may verify that every $B\in P_1\mc{A}_0 P_1$ satisfies the equation $$\langle B e_1^{(3)},e_1^{(2)}\rangle-t\langle Be_1^{(2)},e_1^{(2)}\rangle+t\langle Be_1^{(3)},e_1^{(3)}\rangle =0.\smallskip$$ Thus, if $A$ belongs to $\mc{A}_0$ and $C\coloneqq (P_1AP_1)^2$, then 
	$$\langle C e_1^{(3)},e_1^{(2)}\rangle-t\langle Ce_1^{(2)},e_1^{(2)}\rangle+t\langle Ce_1^{(3)},e_1^{(3)}\rangle=8\langle Ae_2^{(3)},e_1^{(2)}\rangle\left(\langle Ae_1^{(3)},e_1^{(1)}\rangle-t\langle Ae_1^{(2)},e_1^{(1)}\rangle\right)\smallskip$$ must be zero. It follows that $\langle Ae_2^{(3)},e_1^{(2)}\rangle=0$ for all $A\in\mc{A}_0$, or $\langle Ae_1^{(3)},e_1^{(1)}\rangle=t\langle Ae_1^{(2)},e_1^{(1)}\rangle$ for all $A\in\mc{A}_0.$ Indeed, it is clear that every member of $\mc{A}_0$ must satisfy at least one of these equations. If, however, there were elements $A_1$ and $A_2$ in $\mc{A}_0$ such that $\langle A_1e_2^{(3)},e_1^{(2)}\rangle\neq 0$ and $\langle A_2e_1^{(3)},e_1^{(1)}\rangle\neq t\langle A_2e_1^{(2)},e_1^{(1)}\rangle$, then neither equation would be satisfied by their sum.
	
	If it were the case that $\langle Ae_2^{(3)},e_1^{(2)}\rangle=0$ for every $A\in\mc{A}_0$, then by viewing $\mc{A}_0$ as an algebra of matrices with respect to the reordered basis $\left\{e_1^{(1)},e_2^{(3)}, e_1^{(2)},e_1^{(3)}\right\}$ for $\ran(P)$, $\mc{A}_0$ would be seen to lack the projection compression property by Theorem~\ref{At most one non-scalar corner theorem}. This is clearly a contradiction, so it must be that $$\langle Ae_1^{(3)},e_1^{(1)}\rangle=t\langle Ae_1^{(2)},e_1^{(1)}\rangle\,\,\text{for all}\,\,A.\smallskip$$
From here one may verify that every $B\in P_2\mc{A}_0P_2$ satisfies the equation $$2\langle Be_1^{(3)},e_1^{(2)}\rangle-t\langle Be_1^{(2)},e_1^{(2)}\rangle+t\langle Be_2^{(3)},e_2^{(3)}\rangle=0.\smallskip$$ In particular, if $A\in\mc{A}_0$ is as above, then this equation must also hold for $D\coloneqq (P_2AP_2)^2$. Since
	$$2\langle De_1^{(3)},e_1^{(2)}\rangle-t\langle De_1^{(2)},e_1^{(2)}\rangle+t\langle De_2^{(3)},e_2^{(3)}\rangle=8t(\beta-\gamma)(\alpha-\gamma)\smallskip$$ and $\gamma$ may be selected independently from $\alpha$ and $\beta$, we deduce that $t=0$. It is now evident that every $A\in\mc{A}_0$ can be expressed as a matrix of the form 
	$$A=\left[\begin{array}{cc|cc}
	\alpha & 0 & a_{12} & a_{14}\\
	& \gamma & 0 & 0\\ \hline
	& & \beta & a_{24}\\
	& & & \gamma
	\end{array}\right]$$
	with respect to the basis $\left\{e_1^{(1)},e_1^{(3)},e_1^{(2)},e_2^{(3)}\right\}$ for $\ran(P)$. Thus, Theorem~\ref{At most one non-scalar corner theorem} provides the required contradiction. 
	
	It must therefore be the case that $Q_{2\Omega}\rad Q_{3\Omega}=Q_{2\Omega}\mm_nQ_{3\Omega}$ or $n_{3\Omega}=1$. Of course, in the event that $Q_{2\Omega}\rad Q_{3\Omega}\neq Q_{2\Omega}\mm_n Q_{3\Omega}$ and hence $n_{3\Omega}=1$, it follows immediately that $Q_{2\Omega}\rad Q_{3\Omega}=\{0\}$.
	\end{proof}	
\smallskip

The preceding propositions will be key ingredients in our treatment of projection compressible algebras of type III. Our analysis will proceed in the same spirit as those for algebras of types I or II. We will begin in \S\ref{Subsection: Type III algebras unlinked} by classifying the projection compressible type III algebras for which the projections $Q_{i\Omega}$ are mutually unlinked. In $\S\ref{Subsection: Type III algebras linked}$, we will classify the projection compressible type III algebras for which exactly two distinct projections $Q_{i\Omega}$ and $Q_{j\Omega}$ are linked.

\subsection{Type III Algebras with  Unlinked Projections}\label{Subsection: Type III algebras unlinked}

	In this section we present a classification of the projection compressible type III algebras for which the pairs $\Omega$ in $\mc{F}_{III}$ are such that no two distinct projections $Q_{i\Omega}$ and $Q_{j\Omega}$ are linked. Such algebras include the algebra from Example~\ref{exmp:families of compressible algebras:1} when $Q_1\neq 0, Q_3\neq 0$ and $\dim Q_2=1$; and the algebra from Example~\ref{exmp:families of compressible algebras:2}. As the following theorem demonstrates, every projection compressible type III algebra with mutually unlinked projections is either transpose equivalent to the former, or transpose similar to the latter.

	\begin{thm}\label{big theorem for type III mutually unlinked}
	
	Let $\mc{A}$ be a projection compressible type III subalgebra of $\mm_n$. If there is a pair $\Omega$ in $\mc{F}_{III}(\mc{A})$ such that no two distinct projections $Q_{i\Omega}$ and $Q_{j\Omega}$ are linked, then $\mc{A}$ is transpose equivalent to the type III algebra from Example~\ref{exmp:families of compressible algebras:1}, or transpose similar to the algebra from Example~\ref{exmp:families of compressible algebras:2}. Consequently, $\mc{A}$ is idempotent compressible.

	\end{thm}	
	
	\begin{proof}
	
		Let $\Omega=(k,\bigoplus_{i=1}^n\mc{V}_i)$ be a pair in $\mc{F}_{III}(\mc{A})$ as in the statement of the theorem. For each $i$ in $\{1,2,3\}$, fix an orthonormal basis $\left\{e_1^{(i)},e_2^{(i)},\ldots, e_{n_{i\Omega}}^{(i)}\right\}$ for $\ran(Q_{i\Omega})$. 
		
		Note that if $Q_{1\Omega}\rad Q_{2\Omega}=Q_{1\Omega}\mm_n Q_{2\Omega}$ and $Q_{2\Omega}\rad Q_{3\Omega}=Q_{2\Omega}\mm_n Q_{3\Omega}$, then by Lemma~\ref{middle block unlinked implies good radical lem}~(ii), 
	$$\rad=Q_{1\Omega}\mm_n Q_{2\Omega}\dotplus Q_{1\Omega}\mm_n Q_{3\Omega}\dotplus Q_{2\Omega}\mm_n Q_{3\Omega}.\smallskip$$
	In this case, $\mc{A}$ is the type III algebra from Example~\ref{exmp:families of compressible algebras:1}, so $\mc{A}$ is idempotent compressible. It therefore suffices to consider the case in which $Q_{1\Omega}\rad Q_{2\Omega}\neq Q_{1\Omega}\mm_n Q_{2\Omega}$ or $Q_{2\Omega}\rad Q_{3\Omega}\neq  Q_{2\Omega}\mm_n Q_{3\Omega}.$

By replacing $\mc{A}$ with $\mc{A}^{aT}$ if necessary, we may assume without loss of generality that $$Q_{2\Omega}\rad Q_{3\Omega}\neq Q_{2\Omega}\mm_n Q_{3\Omega}.$$ It then follows from Proposition~\ref{case III - all unlinked, n3=1 prop}~(i) that $n_{3\Omega}=1$ and $Q_{2\Omega}\rad Q_{3\Omega}=\{0\}$. Consequently, $n_{1\Omega}\geq 2$ and hence $Q_{1\Omega}\rad Q_{2\Omega}=Q_{1\Omega}\mm_nQ_{2\Omega}$ by Proposition~\ref{case III - all unlinked, n3=1 prop}~(ii). 

The above observations imply that for every $X\in Q_{1\Omega}\mm_n Q_{2\Omega}$, there exists an element $Y_X\in Q_{1\Omega}\mm_n Q_{3\Omega}$ such that $X+Y_X\in\rad$. Additionally, as a consequence of Theorem~\ref{every algebra is similar to an unhinged algebra}, there is a constant $t\in\cc$ such that 
		$$\langle Ae_1^{(3)},e_1^{(2)}\rangle =t\left(\langle Ae_1^{(2)},e_1^{(2)}\rangle-\langle Ae_1^{(3)},e_1^{(3)}\rangle\right)\,\,\text{for all}\,\,A\in\mc{A}.\smallskip$$		
		It therefore suffices to prove that $\rad=Q_{1\Omega}\mm_n(Q_{2\Omega}+Q_{3\Omega})$. Indeed, when this is the case, consider the operator $S\coloneqq I-te_1^{(2)}\otimes {e_1^{(3)*}}\in\mm_n$. One may verify that $S$ is invertible with $S^{-1}=I+te_1^{(2)}\otimes e_1^{(3)*}$, and $S^{-1}\mc{A}S$ is the anti-transpose of the type~III algebra from Example~\ref{exmp:families of compressible algebras:2}.
		
		To this end, note that since $Q_{1\Omega},Q_{2\Omega},$ and $Q_{3\Omega}$ are mutually unlinked, there is an element $A_1\in\mc{A}$ such that $Q_{2\Omega}A_1Q_{2\Omega}=Q_{2\Omega}$ and  $Q_{1\Omega}A_1Q_{1\Omega}=Q_{3\Omega}A_1Q_{3\Omega}=0.$ With respect to the direct sum decomposition $\cc^n=\ran(Q_{1\Omega})\oplus\ran(Q_{2\Omega})\oplus\ran(Q_{3\Omega})$, we may write 
		$$A_1=\left[\begin{array}{c|c|c}
		0 & A_{12} & A_{13}\\ \hline
		& 1 & t\\ \hline
		& & 0
		\end{array}\right]\smallskip$$
		for some $A_{12}\in Q_{1\Omega}\mm_n Q_{2\Omega}$ and $A_{13}\in Q_{1\Omega}\mm_n Q_{3\Omega}$. Thus, for any $X\in Q_{1\Omega}\mc{A}Q_{2\Omega}$, there exists $Y_X\in Q_{1\Omega}\mc{A}Q_{3\Omega}$ such that $\rad$ contains 
		$$(X+Y_X)A_1=\left[\begin{array}{c|c|c}
		0 & X & Y_X\\ \hline
		& 0 & 0\\ \hline
		& & 0
		\end{array}\right]\left[\begin{array}{c|c|c}
		0 & A_{12} & A_{13}\\ \hline
		& 1 & t\\ \hline
		& & 0
		\end{array}\right]=\left[\begin{array}{c|c|c}
		0 & X & tX\\ \hline
		& 0 & 0\\ \hline
		& & 0
		\end{array}\right].\smallskip$$
		We conclude that $\rad=\mc{R}^{(1)}\dotplus\mc{R}^{(2)}$ where 
		$$\mc{R}^{(1)}\coloneqq \left\{\left[\begin{array}{c|c|c}
		0 & X & tX\\ \hline
		& 0 & 0\\ \hline
		& & 0
		\end{array}\right]:X\in \mm_{(k-1)\times 1}\right\}.\smallskip$$
		and $\mc{R}^{(2)}\coloneqq \rad\cap Q_{1\Omega}\mm_n Q_{3\Omega}$.
		
		We claim that $\mc{R}^{(2)}$ must be equal to $Q_{1\Omega}\mm_n Q_{3\Omega}$. Suppose to the contrary that this is not the case. By changing the orthonormal basis for $\ran(Q_{1\Omega})$ if necessary, we may assume that $$\langle Ye_1^{(3)},e_1^{(1)}\rangle=0\,\,\,\text{for all}\,\,Y\in\mc{R}^{(2)}.\smallskip$$ 
		Consider the set $\mc{B}=\left\{e_1^{(1)},e_2^{(1)},e_1^{(2)},e_1^{(3)}\right\}$ and let $P$ denote the orthogonal projection onto the span of $\mc{B}$. Define $\mc{A}_0$ to be the compression $P\mc{A}P$, and accordingly, define $$\begin{array}{ccc}\mc{R}_0^{(1)}\coloneqq P\mc{R}^{(1)}P & \text{and} & \mc{R}_0^{(2)}\coloneqq P\mc{R}^{(2)}P.\end{array}\smallskip$$ Since $\mc{A}_0=\mc{S}\dotplus Rad(\mc{A}_0)$ where $\mc{S}$ is similar to $BD(\mc{A}_0)$ via a block upper triangular similarity, there are constants $u_1,u_2,v_1,v_2\in\cc$ such that each $A\in\mc{A}_0$ can be written as \vspace{0.1cm}
		\begin{equation*}
\resizebox{\textwidth}{!} 
{
$A=\left[\begin{array}{cc|c|c}
		\alpha & 0 & v_1(\alpha-\beta) & u_1(\alpha-\gamma)-tv_1(\beta-\gamma)\\
		& \alpha & v_2(\alpha-\beta) & u_2(\alpha-\gamma)-tv_2(\beta-\gamma)\\ \hline
		& & \beta & t(\beta-\gamma)\\ \hline
		& & & \gamma		 
		\end{array}\right]+\left[\begin{array}{cc|c|c}
		0 & 0 & x_1 & tx_1\\
		 & 0 & x_2 & tx_2\\ \hline
		 & &  0 & 0\\ \hline
		 & & & 0
		\end{array}\right]+\left[\begin{array}{cc|c|c}
		0 & 0 & 0 & 0\\
		& 0 & 0 & y\\ \hline
		& & 0 & 0\\ \hline
		& & & 0
		\end{array}\right].\vspace{0.3cm}$
		}
		\end{equation*}
		where the above summands are expressed with respect to the basis $\mc{B}$ for $\ran(P)$, and belong to $\mc{S}$, $\mc{R}_0^{(1)}$, and $\mc{R}_0^{(2)}$, respectively. We will obtain a contradiction by showing that a certain compression of $\mc{A}_0$ violates Theorem~\ref{At most one non-scalar corner theorem}. To accomplish this goal, it will first be necessary to prove that $t=u_1=0$.
		 
		  With this in mind, consider the matrices 
		$$\begin{array}{cccc}
		P_1\coloneqq \begin{bmatrix}
		1 & 0 & 0 & 1\\
		0 & 2 & 0 & 0\\
		0 & 0 & 2 & 0\\
		1 & 0 & 0 & 1
		\end{bmatrix}, & P_2\coloneqq \begin{bmatrix}
		\phantom{-}1 & 0 & 0 & -1\\
		\phantom{-}0 & 2 & 0 & \phantom{-}0\\
		\phantom{-}0 & 0 & 2 & \phantom{-}0\\
		-1 & 0 & 0 & \phantom{-}1
		\end{bmatrix}, & \text{and} &
		P_3\coloneqq \begin{bmatrix}
		1 & 0 & 1 & 0\\
		0 & 2 & 0 & 0\\
		1 & 0 & 1 & 0\\
		0 & 0 & 0 & 2
		\end{bmatrix},
		\end{array}\smallskip$$
		acting on $\ran(P)$ and written with respect to the basis $\mc{B}$. It is clear that for each $i$, $\frac{1}{2}P_i$ is a subprojection of $P$. One may verify that if $B_1=(b_{ij}^{(1)})$ and $B_2=(b_{ij}^{(2)})$ belong to $P_1\mc{A}_0 P_1$ and $P_2\mc{A}_0P_2$, respectively, then their entries satisfy the equations 
		$$\begin{array}{rl}
		4tb_{14}^{(1)}+2(tv_1-u_1+1)b_{34}^{(1)}-2t^2b_{13}^{(1)}+t(tv_1-u_1-1)b_{22}^{(1)}-t(tv_1-u_1+1)b_{33}^{(1)}=0, & \text{and}\vspace{0.5cm}\\
		4tb_{14}^{(2)}+2(tv_1-u_1-1)b_{34}^{(2)}-2t^2b_{13}^{(2)}+t(tv_1-u_1+1)b_{22}^{(2)}-t(tv_1-u_1-1)b_{33}^{(2)}=0.
		\end{array}\smallskip$$
Let $A_0$ denote the element of $\mc{A}_0$ obtained by setting $\alpha=\beta=x_2=y=0$ and $\gamma=x_1=1$. That is, 
		$$A_0=\left[\begin{array}{cc|c|c}
		0 & 0 & 1 & tv_1-u_1+t\\
		& 0 & 0 & tv_2-u_2\\ \hline
		& & 0 & -t\\ \hline
		& & & 1
		\end{array}\right].\smallskip$$
		Since $\mc{A}$ is projection compressible, $C_1\coloneqq (P_1A_0P_1)^2$ must satisfy the first equation above, while $C_2\coloneqq (P_2A_0P_2)^2$ must satisfy the second. But with ${C_1=(c_{ij}^{(1)})}$ and $C_2=(c_{ij}^{(2)})$, we have
$$\begin{array}{lcl}
4tc_{14}^{(1)}+2(tv_1-u_1+1)c_{34}^{(1)}-2t^2c_{13}^{(1)}\vspace{0.2cm}\\
\,\hspace{1.2cm}+t(tv_1-u_1-1)c_{22}^{(1)}-t(tv_1-u_1+1)c_{33}^{(1)}&=&8t^2(tv_1-u_1-1),\,\,\,\,\,\text{and}\vspace{0.5cm}\\
 4tc_{14}^{(2)}+2(tv_1-u_1-1)c_{34}^{(2)}-2t^2c_{13}^{(2)}\vspace{0.2cm}\\
 \,\hspace{1.2cm}+t(tv_1-u_1+1)c_{22}^{(2)}-t(tv_1-u_1-1)c_{33}^{(2)}
 &=&-8t^2(tv_1-u_1+1).	
\end{array}\smallskip$$
		Adding these equations, it becomes evident that $t=0$. Consequently, $Q_{1\Omega}\mc{R}_0^{(1)}Q_{3\Omega}=\{0\}$.
		
		We now prove that $u_1=0$. Let $A_{0}^\prime$ denote the element of $\mc{A}_0$ obtained by setting $\alpha=\beta=x_1=1$ and $\gamma=x_2=y=0$. That is, 
		$$A_{0}^\prime=\left[\begin{array}{cc|c|c}
		1 & 0 & 1 & u_1\\
		& 1 & 0 & u_2\\ \hline
		& & 1 & 0\\ \hline
		& & & 0
		\end{array}\right].\smallskip$$  Since any element $B_3=(b_{ij}^{(3)})$ in $P_3\mc{A}_0 P_3$ satisfies the equation $2b_{14}^{(3)}-u_1(b_{22}^{(3)}-b_{44}^{(3)})=0,$ it must be the case that the element $C_3\coloneqq (P_3A_{0}^\prime P_3)^2$ satisfies this equation as well. But if $C_3=(c_{ij}^{(3)})$, then $2c_{14}^{(3)}-u_1(c_{22}^{(3)}-c_{44}^{(3)})=8u_1.$
		Therefore, $u_1=0$.

		We deduce that every element in $\mc{A}_0$ admits a matrix representation of the form
		$$\left[\begin{array}{cc|cc}
		\alpha & u_2(\alpha-\gamma)+y & 0 & v_2(\alpha-\beta)+x_2\\
		& \gamma & 0 & 0\\ \hline
		& & \alpha & v_1(\alpha-\beta)+x_1\\
		& & & \beta
		\end{array}\right]\smallskip$$	
		with respect to the reordered basis $\left\{e_2^{(1)},e_1^{(3)},e_1^{(1)},e_1^{(2)}\right\}$ for $\ran(P)$. Since the values of $\alpha$, $\beta$, and $\gamma$ can be selected arbitrarily, an application of Theorem~\ref{At most one non-scalar corner theorem} shows that $\mc{A}_0$ is not projection compressible---a contradiction.
		
		The arguments above demonstrate that $\mc{R}^{(2)}=Q_{1\Omega}\mm_nQ_{3\Omega}$. Thus, $\rad=Q_{1\Omega}\mm_n(Q_{2\Omega}+Q_{3\Omega})$, as required.
		\end{proof}
		\smallskip

\subsection{Type III Algebras with Linked Projections}\label{Subsection: Type III algebras linked}
Let us now consider the projection compressible type III algebras that admit pairs $\Omega\in\mc{F}_{III}$ with distinct mutually linked projections. By Proposition~\ref{case III - Q2 linked prop}, it cannot be the case that all three projections $Q_{1\Omega}$, $Q_{2\Omega}$, and $Q_{3\Omega}$ are mutually linked.

We begin with the case in which there is a pair $\Omega\in\mc{F}_{III}$ with $Q_{2\Omega}$ linked to $Q_{1\Omega}$ or $Q_{3\Omega}$. One example of such an algebra is given by the type~III algebra from Example~\ref{exmp:families of compressible algebras:3}. The following theorem demonstrates that this algebra is in fact, the only example up to transpose equivalence.

\begin{thm}\label{case III - full radical thm}

Let $\mc{A}$ be a projection compressible type III subalgebra of $\mm_n$. If there is a pair $\Omega$ in $\mc{F}_{III}(\mc{A})$ such that $Q_{2\Omega}$ is linked to $Q_{1\Omega}$ or $Q_{3\Omega}$, then $\mc{A}$ is transpose equivalent to the algebra from Example~\ref{exmp:families of compressible algebras:3}. Consequently, $\mc{A}$ is idempotent compressible.

\end{thm}

\begin{proof}

	Let $\Omega$ be as in the statement of the theorem. By replacing $\mc{A}$ with $\mc{A}^{aT}$ if necessary, we may assume without loss of generality that $Q_{1\Omega}$ is the projection that is linked to $Q_{2\Omega}$. In this case, Proposition~\ref{case III - Q2 linked prop}~(i) implies that $n_{1\Omega}=1$ and $Q_{1\Omega}\rad Q_{2\Omega}=Q_{1\Omega}\mm_nQ_{2\Omega}$. It follows that $n_{3\Omega}\geq 2$, and hence  $Q_{3\Omega}$ is unlinked from $Q_{1\Omega}$ and $Q_{2\Omega}$ by Proposition~\ref{case III - Q2 linked prop}~(ii). Thus,  $Q_{2\Omega}\rad Q_{3\Omega}=Q_{2\Omega}\mm_nQ_{3\Omega}$ by Proposition~\ref{case III - all unlinked, n3=1 prop}.
	
	Fix operators $T_1\in Q_{1\Omega}\mm_nQ_{2\Omega}$ and $T_2\in Q_{2\Omega}\mm_nQ_{3\Omega}$. By the observations above, there exist $R_1,R_2$ in $\rad$ such that $Q_{1\Omega}R_1Q_{2\Omega}=T_1$ and $Q_{2\Omega}R_2Q_{3\Omega}=T_2$. With respect to the direct sum decomposition $\cc^n=\ran(Q_{1\Omega})\oplus\ran(Q_{2\Omega})\oplus\ran(Q_{3\Omega})$, we may write 
	$$\begin{array}{ccc}
	R_1=\begin{bmatrix}
	0 & T_1 & R_{13}^{(1)}\\
	0 & 0 & R_{23}^{(1)}\\
	0 & 0 & 0
	\end{bmatrix} & \text{and} & R_2=\begin{bmatrix}
	0 & R_{12}^{(2)} & R_{13}^{(2)}\\
	0 & 0 & T_2\\
	0 & 0 & 0
	\end{bmatrix}
	\end{array}\smallskip$$ for some operators $R_{ij}^{(1)}$ and $R_{ij}^{(2)}$.
	From here it is easy to see that $R_1R_2=T_1T_2\in\rad.$
Since $T_1$ and $T_2$ were arbitrary, we conclude that $\rad$ contains $Q_{1\Omega}\mm_nQ_{3\Omega}$.

It will now be shown that each block $Q_{i\Omega}\rad Q_{j\Omega}$ exists independently in $\rad$. First, write $\mc{A}=\mc{S}\dotplus\rad$ where $\mc{S}$ is semi-simple. Since $Q_{1\Omega}$ and $Q_{2\Omega}$ are linked, $\mc{S}$ is similar to $\cc(Q_{1\Omega}+Q_{2\Omega})+\cc Q_{3\Omega}$ via an upper triangular similarity. From this it follows that $Q_{1\Omega}\mc{S}Q_{2\Omega}=\{0\}$, and hence $\mc{S}$ contains an element $A$ of the form 
$$A=\begin{bmatrix}
0 & 0 & A_{13}\\
0 & 0 & A_{23}\\
0 & 0 & I
\end{bmatrix}.\smallskip$$ 

Using the fact that $Q_{1\Omega}\mm_nQ_{3\Omega}\subseteq \rad$, we deduce that $T_2=R_2A-Q_{1\Omega}R_2AQ_{3\Omega}$ belongs to $\rad$. Since $T_2$ was arbitrary, $\rad$ contains $Q_{2\Omega}\mm_nQ_{3\Omega}$. Consequently, $T_1=R_1-Q_{1\Omega}R_1Q_{3\Omega}-Q_{2\Omega}R_1Q_{3\Omega}$ belongs to $\rad$. This proves that $\rad$ contains $Q_{1\Omega}\mm_nQ_{2\Omega}$, and therefore $$\rad=Q_{1\Omega}\mm_nQ_{2\Omega}\dotplus Q_{1\Omega}\mm_nQ_{3\Omega}\dotplus Q_{2\Omega}\mm_n Q_{3\Omega}.$$ We conclude that $\mc{A}=\cc(Q_{1\Omega}+Q_{2\Omega})+\cc Q_{3\Omega}\dotplus\rad$. Thus, $\mc{A}$ is the algebra from Example~\ref{exmp:families of compressible algebras:3}, as claimed.
\end{proof}
\smallskip

With the proof of Theorem~\ref{case III - full radical thm} complete, we are left only to classify the projection compressible type~III algebras such that $\mc{F}_{III}$ contains a pair $\Omega$ in which $Q_{1\Omega}$ and $Q_{3\Omega}$ linked, yet neither of these projections is linked to $Q_{2\Omega}$. It will be shown in Theorem~\ref{case III - Q1 and Q3 linked, Q2 unlinked theorem} that such an algebra is necessarily the unitization of an $\mc{LR}$-algebra. Unsurprisingly, the proof of this result shares many similarities with that of Theorem~\ref{case II - linked case implies LR theorem}, the analogous result for algebras of type II. One must modify the arguments in the type III case, however, to reflect the absence of a block in $BD(\mc{A})$ of size $2$ or greater. 

The first step in this direction is the following adaptation of Lemma~\ref{extending incomplete radical lemma} to the type~III setting.

\begin{lem}\label{extending incomplete radical lemma - case III}

Let $\mc{A}$ be a projection compressible type III subalgebra of $\mm_4$, and suppose that  $\mc{F}_{III}(\mc{A})$ contains a pair $\Omega=(k,\bigoplus_{i=1}^4\mc{V}_i)$ with $k=3$. Assume that $Q_{1\Omega}$ and $Q_{3\Omega}$ are linked.
	\begin{itemize}
		\item[(i)] If there exist a constant $t\in\cc$ and for each $i\in\{1,2,3\}$, an orthonormal basis $\left\{e_1^{(i)},e_2^{(i)},\ldots, e_{n_{i\Omega}}^{(i)}\right\}$ for $\ran(Q_{i\Omega})$ such that $$\langle Ae_1^{(2)},e_1^{(1)}\rangle=t\left(\langle Ae_1^{(1)},e_1^{(1)}\rangle-\langle Ae_1^{(2)},e_1^{(2)}\rangle\right)\,\,\,\text{for all}\,\,A\in\mc{A},\smallskip$$ then $\langle Ae_1^{(3)},e_1^{(1)}\rangle=-t\langle Ae_1^{(3)},e_1^{(2)}\rangle$ for every $A\in\mc{A}$.\\
		
		\item[(ii)] If there exist a constant $t\in\cc$ and for each $i\in\{1,2,3\}$, an orthonormal basis $\left\{e_1^{(i)},e_2^{(i)},\ldots, e_{n_{i\Omega}}^{(i)}\right\}$ for $\ran(Q_{i\Omega})$ such that $$\langle Ae_1^{(3)},e_1^{(2)}\rangle=t\left(\langle Ae_1^{(2)},e_1^{(2)}\rangle-\langle Ae_1^{(3)},e_1^{(3)}\rangle\right)\,\,\,\text{for all}\,\,A\in\mc{A},\smallskip$$
		then $\langle Ae_1^{(3)},e_i^{(1)}\rangle=t\langle Ae_1^{(2)},e_i^{(1)}\rangle$ for every $A\in\mc{A}$ and each $i\in\{1,2\}$.
	\end{itemize}

\end{lem}

\begin{proof}
	First note that since $Q_{1\Omega}$ and $Q_{3\Omega}$ are linked, Proposition~\ref{case III - Q2 linked prop} implies that neither of these projections is linked to $Q_{2\Omega}$.
	
	We begin by considering the situation of (i). With respect to the basis $\mc{B}=\left\{e_1^{(1)},e_2^{(1)},e_1^{(2)},e_1^{(3)}\right\}$ for $\cc^4$, each $A$ in $\mc{A}$ can be expressed as a matrix of the form
	$$A=\left[\begin{array}{cc|c|c}
	\alpha & 0 & t(\alpha-\beta) & a_{14}\\
	& \alpha & a_{23} & a_{24}\\ \hline
	& & \beta & a_{34}\\ \hline
	& & & \alpha
	\end{array}\right]\smallskip$$
	for some $\alpha$, $\beta,$ and $a_{ij}$ in $\cc$.
	Consider the matrix
	$$P=\begin{bmatrix}
	1 & 0 & 0 & 1\\
	0& 2 & 0 & 0\\
	0 & 0 & 2 & 0\\
	1 & 0 & 0 & 1
	\end{bmatrix}.\smallskip$$
	It is straightforward to check that $\frac{1}{2}P$ is a projection in $\mm_4$ and every element $B=(b_{ij})$ in $P\mc{A}P$ satisfies the equation $2b_{13}-t(b_{22}-b_{33})=0.$ But if $A\in\mc{A}$ is as above, and $C=(c_{ij})$ denotes the operator $(PAP)^2$, then
	$$2c_{13}-t(c_{22}-c_{33})=8t(ta_{34}+a_{14})(\alpha-\beta).\smallskip$$
	Since $\mc{A}$ is projection compressible, $C$ belongs to $P\mc{A}P$, and hence the right-hand side of this equation must be $0$ for all $\mc{A}$. Since $\alpha$ and $\beta$ may be chosen arbitrarily, it follows that either $t=0$ or $a_{14}=-ta_{34}$ for all $A$ in $\mc{A}$. 
	
	If $t=0$, then each $A\in\mc{A}$ can be expressed as a matrix of the form
	$$A=\left[\begin{array}{cc|cc}
	\alpha & a_{23} & 0 & a_{24}\\
	& \beta & 0 & a_{34}\\ \hline
	& & \alpha & a_{14}\\
	& & & \alpha
	\end{array}\right]\smallskip$$
	with respect to the reordered basis $\left\{e_2^{(1)},e_1^{(2)},e_1^{(1)},e_1^{(3)}\right\}$ for $\cc^4.$
	In this case, Theorem~\ref{At most one non-scalar corner theorem} demonstrates that $a_{14}=\langle Ae_1^{(3)},e_1^{(1)}\rangle=0$ for all $A$. Thus, the equation $a_{14}=-ta_{34}$ holds in either case. That is,
	$\langle Ae_1^{(3)},e_1^{(1)}\rangle=-t\langle Ae_1^{(3)},e_1^{(2)}\rangle$ for all $A\in\mc{A}.$ 
	
	We now turn our attention to the proof of (ii). In this setting, every $A$ in $\mc{A}$ admits a matrix representation of the form 
	$$A=\left[\begin{array}{cc|c|c}
	\alpha & 0 & a_{13} & a_{14}\\
	& \alpha & a_{23} & a_{24}\\ \hline
	& & \beta & t(\beta-\alpha)\\ \hline
	& & & \alpha
	\end{array}\right]\smallskip$$ with respect to the basis $\mc{B}=\left\{e_1^{(1)},e_2^{(1)},e_1^{(2)},e_1^{(3)}\right\}$. With $P$ as in (i), every element  $B=(b_{ij})$ in $P\mc{A}P$ satisfies the equation $2b_{34}-t(b_{33}-b_{22})=0.$ It can be verified, however, that if $A$ is as above and $C\coloneqq (PAP)^2=(c_{ij})$, then 
	$$2c_{34}-t(c_{33}-c_{22})=8t(ta_{13}-a_{14})(\alpha-\beta).\smallskip$$
	Once again, it follows that either $t=0$ or $a_{14}=ta_{13}$ for all $A\in\mc{A}$. 
	
	Suppose first that $t=0$. Let $P^\prime$ denote the matrix 
	$$P^\prime=\begin{bmatrix}
	\phantom{-}2 & \phantom{-}0 & -1 & -1\\
	\phantom{-}0 & \phantom{-}3 & \phantom{-}0 & \phantom{-}0\\
	-1 & \phantom{-}0 & \phantom{-}2 & -1\\
	-1 & \phantom{-}0 & -1 & \phantom{-}2
	\end{bmatrix},\smallskip$$
	written with respect to the basis $\mc{B}$, so $\frac{1}{3}P^\prime$ is a projection in $\mm_4$. Direct computations show that if $B=(b_{ij})$ belongs to $P^\prime\mc{A}P^\prime$,  then $b_{33}+2b_{31}-b_{43}-2b_{41}-b_{22}=0.$
	But with $A$ as above and $C\coloneqq (P^\prime AP^\prime)^2=(c_{ij})$, we have 
	$$c_{33}+2c_{31}-c_{43}-2c_{41}-c_{22}=27a_{14}(\beta-\alpha).\smallskip$$
	Since $\alpha$ and $\beta$ may be selected arbitrarily, it follows that $a_{14}=\langle Ae_1^{(3)},e_1^{(1)}\rangle=0$ for all $A$ in $\mc{A}$. Thus, the equation $a_{14}=ta_{13}$ holds in either case. That is,
	$$\langle Ae_1^{(3)},e_1^{(1)}\rangle=t\langle Ae_1^{(2)},e_1^{(1)}\rangle\,\,\text{for all}\,\,A\in\mc{A}.\smallskip$$
	
	Finally, by switching the order of the first two vectors in $\mc{B}$ and repeating the above analysis with respect to this reordered basis, one may deduce that
		$$\langle Ae_1^{(3)},e_2^{(1)}\rangle=t\langle Ae_1^{(2)},e_2^{(1)}\rangle\,\,\text{for all}\,\,A\in\mc{A}.\smallskip$$
		Thus, the proof is complete.
\end{proof}
\smallskip

\begin{thm}\label{case III - Q1 and Q3 linked, Q2 unlinked theorem}

Let $\mc{A}$ be a projection compressible type III subalgebra of $\mm_n$. If there is a pair $\Omega$ in $\mc{F}_{III}(\mc{A})$ such that $Q_{1\Omega}$ and $Q_{3\Omega}$ are linked, then $\mc{A}$ is the unitization of an $\mc{LR}$-algebra. Consequently, $\mc{A}$ is idempotent compressible.

\end{thm}

\begin{proof}

	Let $\Omega$ be a pair in $\mc{F}_{III}(\mc{A})$ such that $Q_{1\Omega}$ and $Q_{3\Omega}$ are linked. By replacing $\mc{A}$ with $\mc{A}^{aT}$ if necessary, we will assume that $n_{1\Omega}=\max\{n_{1\Omega},n_{3\Omega}\}\geq 2$. Note that by Proposition~\ref{case III - Q2 linked prop}, neither of these projections is linked to $Q_{2\Omega}$.
	
	By Theorem~\ref{structure of modules over Mn}, there are subprojections $Q_1^\prime\leq Q_{1\Omega}$ and $Q_3^\prime\leq Q_{3\Omega}$ such that $$\begin{array}{lr}Q_{1\Omega}\rad Q_{2\Omega}=Q_1^\prime \mm_n Q_{2\Omega} & \text{and}\vspace{0.3cm}\\
	 Q_{2\Omega}\rad Q_{3\Omega}=Q_{2\Omega}\mm_n Q_3^\prime.\end{array}\smallskip$$
	As in the proof of Theorem~\ref{case II - linked case implies LR theorem}, we will show that $\mc{A}$ is similar to $$\mc{A}_0\coloneqq (Q_1^\prime+Q_{2\Omega})\mm_n (Q_{2\Omega}+Q_3^\prime)+\cc I,\smallskip$$ and hence that $\mc{A}$ is the unitization of an $\mc{LR}$-algebra. To show that this is the case, we must first determine the structure of $Q_{1\Omega}\mc{A}Q_{3\Omega}$.
	
	 Define projections $Q_1^{\prime\prime}\coloneqq Q_{1\Omega}-Q_1^\prime$ and $Q_3^{\prime\prime}\coloneqq Q_{3\Omega}-Q_3^\prime$. For each $i\in\{1,3\}$, let $\left\{e_1^{(i)},e_2^{(i)},\ldots, e_{n_{i\Omega}}^{(i)}\right\}$ be an orthonormal basis for $\ran(Q_{i\Omega})$ such that if $Q_i^{\prime\prime}\neq 0$, then $$\ran(Q_i^{\prime\prime})=\vee\left\{e_{1}^{(i)},e_{2}^{(i)}, \ldots, e_{\ell_i}^{(i)}\right\}\smallskip$$ for some index $\ell_i\in\{1,2,\ldots, n_{i\Omega}\}$. Furthermore, let $e_1^{(2)}$ be a unit vector in $\ran(Q_{2\Omega})$.  Since $\mc{A}$ is similar to $BD(\mc{A})\dotplus\rad$ via an upper triangular similarity, there are matrices $T_1\in Q_1^{\prime\prime}\mm_nQ_{2\Omega}$ and $T_2\in Q_{2\Omega}\mm_nQ_3^{\prime\prime}$ such that for each $A\in\mc{A}$,
		$$\begin{array}{rl}Q_1^{\prime\prime}AQ_{2\Omega}=(Q_1^{\prime\prime}AQ_1^{\prime\prime})T_1-T_1(Q_{2\Omega}AQ_{2\Omega})\phantom{.} & \text{and}\vspace{0.33cm}\\
		Q_{2\Omega}AQ_3^{\prime\prime}=(Q_{2\Omega}AQ_{2\Omega})T_2-T_2(Q_3^{\prime\prime}AQ_3^{\prime\prime}).\end{array}\smallskip$$

	We may obtain information on the structure of $Q_1^{\prime\prime}\mc{A}Q_{3\Omega}$ by appealing to Lemma~\ref{extending incomplete radical lemma - case III}. Of course, there is little to be said when $Q_1^{\prime\prime}=0$. If instead $Q_1^{\prime\prime}\neq 0$, fix arbitrary indices $i\in\{1,2,\ldots, \ell_1\}$, \linebreak $i^\prime\in\{1,2,\ldots, n_{1\Omega}\}\setminus\{i\}$, and $j\in\{1,2,\ldots, n_{3\Omega}\}$. Define $\mc{B}=\left\{e_i^{(1)},e_{i^\prime}^{(1)},e_1^{(2)},e_j^{(3)}\right\}$ and let $P$ denote the orthogonal projection onto the span of $\mc{B}$. With respect to the basis $\mc{B}$ for $\ran(P)$, every member of $P\mc{A}P$ can be written as a matrix of the form 
	$$\left[\begin{array}{cc|c|c}
	\alpha & 0 & t_i^{(1)}(\alpha-\beta) & a_{14}\\
	& \alpha & a_{23} & a_{24}\\ \hline
	& & \beta & a_{34}\\ \hline
	& & & \alpha
	\end{array}\right],$$
	where $t_i^{(1)}\coloneqq \langle T_1e_1^{(2)},e_i^{(1)}\rangle$.  Thus, an application Lemma~\ref{extending incomplete radical lemma - case III} (i) demonstrates that $$\langle Ae_j^{(3)},e_i^{(1)}\rangle=-t_i^{(1)}\langle Ae_j^{(3)},e_1^{(2)}\rangle\,\,\text{for all}\,\,A\in\mc{A}.\smallskip$$ Since the indices $i$, $i^\prime$, and $j$ were selected arbitrarily, it follows that $$Q_1^{\prime\prime}AQ_{3\Omega}=-T_1Q_{2\Omega}AQ_{3\Omega}\,\,\,\text{for all}\,\,A\in\mc{A}.\smallskip$$
	
	A similar argument can be used to determine the structure of $Q_{1\Omega}\mc{A}Q_3^{\prime\prime}$. Indeed, there is nothing to be said when $Q_3^{\prime\prime}=0$. If instead $Q_3^{\prime\prime}\neq 0$, choose distinct indices $i$ and $i^\prime$ in $\{1,2,\ldots, n_{1\Omega}\}$, and let $j\in\{1,2,\ldots, \ell_3\}$ be arbitrary. Define $\mc{C}=\left\{e_i^{(1)},e_{i^\prime}^{(1)},e_1^{(2)},e_j^{(3)}\right\}$, and let $P^\prime$ denote the orthogonal projection onto the span of~$\mc{C}$. The compression $P^\prime \mc{A}P^\prime$ is an algebra of the form described in Lemma~\ref{extending incomplete radical lemma - case III} (ii), and hence this result indicates that each $A\in\mc{A}$ satisfies the equation $$\langle Ae_j^{(3)},e_i^{(1)}\rangle=t_j^{(2)}\langle Ae_1^{(2)},e_i^{(1)}\rangle,\smallskip$$ where $t_j^{(2)}\coloneqq \langle T_2e_j^{(3)},e_1^{(2)}\rangle.$ Again, the fact that $i$, $i^\prime$, and $j$ were chosen arbitrarily implies that $Q_{1\Omega}AQ_3^{\prime\prime}=Q_{1\Omega}AQ_{2\Omega}T_2$ for all $A\in\mc{A}.$

	Our findings thus far indicate that with respect to the decomposition 
	$$\cc^n=\ran(Q_1^{\prime\prime})\oplus\ran(Q_1^\prime)\oplus\ran(Q_{2\Omega})\oplus\ran(Q_3^{\prime\prime})\oplus\ran(Q_3^\prime),\smallskip$$ each $A$ in $\mc{A}$ can be expressed as a matrix of the form 
	$$A=\left[\begin{array}{c|c|c|c|c}
	\alpha I & 0 & (\alpha-\beta)T_1 & A_{14}
	 & A_{15}\\ \hline
	& \alpha I & J_1 & A_{24} & A_{25}\\ \hline
	& & \beta & (\beta-\alpha)T_2 & J_2\\ \hline
	& & & \alpha I & 0\\ \hline
	& & & & \alpha I
	\end{array}\right],$$
	for some $\alpha,\beta\in\cc$, $J_1\in Q_1^\prime\rad Q_{2\Omega}$, $J_2\in Q_{2\Omega}\rad Q_3^\prime$, and operators $A_{ij}$ satisfying the equations 
	$$\begin{array}{ccc} \left[\begin{array}{c|c} A_{14} & A_{15}\end{array}\right]=-T_1\left[\begin{array}{c|c} (\beta-\alpha)T_2 & J_2 \end{array}\right] & \text{\,\,\,and\,\,\,} & \left[\begin{array}{c}
	A_{14}\\ \hline
	A_{24}
	\end{array}\right]=\left[\begin{array}{c}
	(\alpha-\beta)T_1\\ \hline
	J_1
	\end{array}\right]T_2.\end{array}\smallskip$$
	
	To see that $\mc{A}$ is similar to $\mc{A}_0=(Q_1^{\prime}+Q_{2\Omega})\mm_n(Q_{2\Omega}+Q_3^\prime)+\cc I$, and hence is the unitization of an $\mc{LR}$-algebra, consider the operator
	$S\coloneqq I-T_1-T_2.$  This map is invertible with $S^{-1}=I+T_1+T_2+T_1T_2$. In addition, we have that for $A$ as above, 
	$$S^{-1}AS=\left[\begin{array}{c|c|c|c|c}
	\alpha I & 0 & 0 & 0
	 & 0\\ \hline
	& \alpha I & J_1 & 0 & A_{25}\\ \hline
	& & \beta & 0 & J_2\\ \hline
	& & & \alpha I & 0\\ \hline
	& & & & \alpha I
	\end{array}\right].$$
	It is now apparent that $S^{-1}\mc{A}S$ is a type III algebra that admits a reduced block upper triangular form with respect to the above decomposition. Since 
	 $$\begin{array}{ll}Q_{1\Omega}Rad(S^{-1}\mc{A}S)Q_{2\Omega}=Q_1^\prime\rad Q_{2\Omega}=Q_1^\prime\mm_nQ_{2\Omega} & \text{and}\vspace{0.2cm}\\ Q_{2\Omega}Rad(S^{-1}\mc{A}S)Q_{3\Omega}=Q_{2\Omega}\rad Q_3^\prime=Q_{2\Omega}\mm_nQ_3^\prime,\end{array}\smallskip$$ 
	 
	\noindent it follows from Lemma~\ref{middle block unlinked implies good radical lem} (ii) that 
	$S^{-1}\mc{A}S=(Q_1^{\prime}+Q_{2\Omega})\mm_n(Q_{2\Omega}+Q_3^\prime)+\cc I=\mc{A}_0.$
\end{proof}
\smallskip

	\section[7]{Main Result and Applications}
	
	\subsection{The Main Result}
	The analysis carried out in the preceding sections provides a description of the unital projection compressible subalebras of $\mm_n$, $n\geq 4$ up to transpose similarity. Since every such algebra was also seen to admit the idempotent compression property, it follows that the two notions of compressibility coincide for unital algebras in this setting. We therefore obtain the following theorem, the main result of this paper.
	
	\begin{thm}\label{main result}
	
		Let $\mc{A}$ be a unital subalgebra of $\mm_n$ for some integer $n\geq 4$. The following are equivalent.
		\begin{itemize}
			\item[(i)]$\mc{A}$ is projection compressible;
			\item[(ii)]$\mc{A}$ is idempotent compressible;
			\item[(iii)]$\mc{A}$ is the unitization of an $\mc{LR}$-algebra, or $\mc{A}$ is transpose similar to one of the algebras from\linebreak  Example~\ref{exmp:families of compressible algebras}.	

		\end{itemize}	
	
	\end{thm}
	
	Combining Theorem~\ref{main result} and \cite[Theorem~6.0.1]{CMR1}, we conclude that the two notions of compressibility coincide for all unital algebras.
	
	\begin{thm}\label{P comp = E comp thm}
		A unital subalgebra $\mc{A}$ of $\mm_n$, $n\geq 2$, is projection compressible if and only if it is idempotent compressible.
	\end{thm}

	\noindent In light of Theorem~\ref{P comp = E comp thm}, we make the following definition.
	
		\begin{defn}
		A unital subalgebra $\mc{A}$ of $\mm_n$ is \textit{compressible} if $\mc{A}$ is projection compressible (equivalently, if $\mc{A}$ is idempotent compressible).
		\end{defn}
		
		It is worth noting that nearly all of the classification results from \S4-6 describe the various unital compressible subalgebras of $\mm_n$ up to transpose \textit{equivalence}, not just transpose similarity. Indeed, the only instance in which a description up to transpose equivalence was not achieved was in Theorem~\ref{big theorem for type III mutually unlinked}. There it was shown that a projection compressible type III algebra is either transpose equivalent to the type III algebra from Example~\ref{exmp:families of compressible algebras:1}, or transpose similar to the algebra from Example~\ref{exmp:families of compressible algebras:2}. 
		
		The following proposition describes the similarity orbit of the algebra from Example~\ref{exmp:families of compressible algebras:2} up to unitary equivalence, thereby providing a characterization of the (unital) compressible subalgebras of $\mm_n$, $n\geq 4$, up to transpose equivalence.\smallskip

		\begin{prop}\label{similar to second example implies unitarily equivalent to At}
		Let $n\geq 3$ be an integer, let $Q_1$ and $Q_2$ be mutually orthogonal rank-one projections in $\mm_n$, and define $Q_3\coloneqq I-Q_1-Q_2.$ Let $\{e_1,e_2,\ldots, e_n\}$ be an orthonormal basis for $\cc^n$ such that $e_1\in \ran(Q_1)$, $e_2\in \ran(Q_2)$, and $e_i\in \ran(Q_3)$ for all $i\geq 3$. If $$\begin{array}{l}\mc{A}\coloneqq \cc Q_1+\cc Q_2+\cc Q_3+(Q_1+Q_2)\mm_nQ_3\vspace{0.2cm}\\
		\left.\right.\hspace{0.36cm}=\left\{\begin{bmatrix}
			\alpha & 0 & M_{13}\\
			0 & \beta & M_{23}\\
			0 & 0 & \gamma I 
			\end{bmatrix}:\alpha,\beta,\gamma\in\cc,M_{ij}\in Q_i\mm_nQ_j\right\}\end{array}$$ denotes the compressible algebra from Example~\ref{exmp:families of compressible algebras:2}, and $\mc{B}$ is an algebra that is similar to $\mc{A}$, then there is some $t\in\cc$ such that $\mc{B}$ is unitarily equivalent to $$\begin{array}{l}
			\mc{A}_t\coloneqq \left\{A+t\left(\langle Ae_1,e_1\rangle-\langle Ae_2,e_2\rangle\right)e_1\otimes e_2^*:A\in\mc{A}\right\}\vspace{0.2cm}\\
			\left.\right.\hspace{0.45cm}=\left\{\begin{bmatrix}
			\alpha & t(\alpha-\beta) & M_{13}\\
			0 & \beta & M_{23}\\
			0 & 0 & \gamma I 
			\end{bmatrix}:\alpha,\beta,\gamma\in\cc,M_{ij}\in Q_i\mm_nQ_j\right\}.\end{array}$$
		
		\end{prop}
		
		\begin{proof}

			Suppose that $\mc{B}=S^{-1}\mc{A}S$ for some invertible $S\in\mm_n$. For all indices $i\in\{1,2\}$ and $j\in\{3,4,\ldots, n\}$, define $E_{ij}:=e_i\otimes e_j^*$ and $E_{ij}^\prime:=S^{-1}E_{ij}S.$ Furthermore, define $Q_i^\prime:=S^{-1}Q_iS$ for $i\in\{1,2,3\}$. Observe that $$\mc{B}=S^{-1}\mc{A}S=\mathrm{span}\left\{Q_1^\prime,Q_2^\prime,Q_3^\prime,E_{ij}^\prime:i\in\{1,2\}, j\in\{3,4,\ldots,n\}\right\}.$$
			
			Let $\{f_1,f_2,\ldots,f_n\}$ be an orthonormal basis for $\cc^n$ such that $f_1$ and $f_2$ belong to $\ker (Q_3^\prime)$. Let $P_1$, $P_2$, and $P_3$ denote the orthogonal projections onto $\cc f_1$, $\cc f_2$, and $\mathrm{span}\left\{f_i:i\geq 3\right\}=\ker(Q_3^\prime)^\perp$, respectively. Since $P_3Q_3^\prime P_3=P_3$ and $Q_3^\prime Q_1^\prime=Q_3^\prime Q_2^\prime=0$, we have that $Q_1^\prime=(P_1+P_2)Q_1^\prime$ and $Q_2^\prime=(P_1+P_2)Q_2^\prime$. 
			
			Note that since $Q_1^\prime Q_2^\prime=Q_2^\prime Q_1^\prime=0,$ we may adjust the first two basis vectors if necessary to assume that $Q_1^\prime$ and $Q_2^\prime$ are upper triangular with respect to $\{f_1,f_2,\ldots, f_n\}$, and $\langle Q_i^\prime f_j,f_j\rangle= \delta_{ij}$ for $i,j\in\{1,2\}$. Thus, there are matrices $X_{ij}$, $Y_{ij}$, and $Z_{ij}$, and a constant $t\in\cc$ such that with respect to the decomposition $\cc^n=\ran(P_1)\oplus\ran(P_2)\oplus\ran(P_3),$
			$$\begin{array}{cccc}
			Q_1^\prime=\left[\begin{array}{c|c|c}
			1 & t & X_{13}\\ \hline
			0 & 0 & 0\\ \hline
			0 & 0 & 0
			\end{array}\right], & Q_2^\prime=\left[\begin{array}{c|c|c}
			0 & -t & Y_{13}\\ \hline
			0 & \phantom{-}1 & Y_{23}\\ \hline
			0 & \phantom{-}0 & 0	
			\end{array}\right], & \text{and} & Q_3^\prime=\left[\begin{array}{c|c|c}
			0 & 0 & Z_{13}\\ \hline
			0 & 0 & Z_{23}\\ \hline
			0 & 0 & I 
			\end{array}\right].
			\end{array}\smallskip$$
					
			Finally, since $E_{ij}^\prime=(Q_1^\prime+Q_2^\prime)E_{ij}^\prime Q_3^\prime$, we have that $E_{ij}^\prime=(P_1+P_2)E_{ij}^\prime P_3$ for all indices $i$ and $j$. Dimension considerations then imply that 
			$$\mathrm{span}\left\{E_{ij}^\prime:i\in\{1,2\},j\in\{3,4,\ldots, n\}\right\}=(P_1+P_2)\mm_nP_3,\smallskip$$
			and therefore $$\mc{B}=\left\{B+t\left(\langle Bf_1,f_1\rangle-\langle Bf_2,f_2\rangle\right)f_1\otimes f_2^*:B\in\cc P_1+\cc P_2+\cc P_3+(P_1+P_2)\mm_nP_3\right\}.\smallskip$$
			 We conclude that $\mc{A}_t=U^*\mc{B}U$ where $U\in\mm_n$ is the unitary satisfying $Ue_i=f_i$.
		\end{proof}
		\smallskip
		
		\begin{cor}\label{main result, transpose equivalence}
		Let $n\geq 4$ be an integer, and let $\mc{A}$ be a unital subalgebra of $\mm_n$.  The following are equivalent.
				\begin{itemize}
					\item[(i)]$\mc{A}$ is compressible;
					
					\item[(ii)]$\mc{A}$ is the unitization of an $\mc{LR}$-algebra, or $\mc{A}$ is transpose equivalent to the algebra from Example~\ref{exmp:families of compressible algebras:1}, the algebra from Example~\ref{exmp:families of compressible algebras:3}, or the algebra $\mc{A}_t$ from Proposition~\ref{similar to second example implies unitarily equivalent to At}. \bigskip
				
				\end{itemize}
		
		\end{cor}
		
		\begin{rmk}
		\upshape{
	The above result, together with Theorem~\ref{main result}, implies that if $\mc{A}$ is transpose similar to an algebra from Theorem~\ref{main result}~(iii), then $\mc{A}$ is transpose equivalent to an algebra from Corollary~\ref{main result, transpose equivalence}~(ii). Indeed, Proposition~\ref{similar to second example implies unitarily equivalent to At} makes this fact explicit for the algebra in Example~\ref{exmp:families of compressible algebras:2}, while in \cite{CMR1} it was shown that the class of $\mc{LR}$-algebras is invariant under transpose similarity. Arguments akin to those in the proof of Proposition~\ref{similar to second example implies unitarily equivalent to At} can be used to show that any algebra transpose similar to the algebra from Example~\ref{exmp:families of compressible algebras:1} (resp. Example~\ref{exmp:families of compressible algebras:3}) is in fact, transpose equivalent to it.\\
		}
		\end{rmk}
	
	\subsection{Applications}\label{Subsection: Applications}
	
	Here we investigate some of the applications of the classification of unital compressible algebras. It follows from Theorem~\ref{P comp = E comp thm} that the class of all such algebras is invariant under similarity and transposition. Using this fact, it is relatively straightforward to determine which unital semi-simple algebras admit the compression property.
	
		\begin{cor}
	
		Let $n\geq 2$ be an integer, and let $\mc{A}$ be a unital, semi-simple subalgebra of $\mm_n$. The following are equivalent:
		\begin{itemize}
			\item[(i)]$\mc{A}$ is compressible;
			
			\item[(ii)]$\mc{A}=\cc I$ or $\mc{A}$ is similar to $\mm_k\oplus\cc I_{n-k}$ for some positive integer $k$.
		\end{itemize}
	
	\end{cor}
	
	\begin{proof}
	
		Since $\cc I$ and $\mm_k\oplus\cc I_{n-k}$ are unitizations of $\mc{LR}$-algebras, it is obvious that (ii) implies (i). Assume now that (i) holds, so $\mc{A}$ is a unital, semi-simple subalgebra of $\mm_n$ that admits the compression property. Assume as well that $\mc{A}$ is in reduced block upper triangular form with respect to some orthogonal decomposition $\bigoplus_{i=1}^m\mc{V}_i$ of $\cc^n$. By Theorem~\ref{every algebra is similar to an unhinged algebra}, $\mc{A}$ is similar to $\mc{B}\coloneqq BD(\mc{A})$. It therefore suffices to prove that $\mc{B}$ is similar to an algebra of the form prescribed in (ii).
		
		If $n=2$, then $\mc{B}$ is equal to $\cc I$, $\cc\oplus \cc$, or $\mm_2$, and hence $\mc{B}$ is of the desired form. If instead $n=3$, then either $\mc{B}$ is equal to $\cc I$ or $\mm_3$, or $\mc{B}$ is unitarily equivalent to $\cc\oplus \cc I_2$ or $\mm_2\oplus\cc$. Indeed, the only other block diagonal subalgebra of $\mm_3$ is the algebra of all $3\times 3$ diagonal matrices. This algebra was shown to lack the compression property in \cite[Theorem~5.2.6]{CMR1}, and hence cannot be similar to $\mc{B}$. Again we see that (ii) holds.
		
		Suppose now that $n\geq 4$. By Theorem~\ref{At most one non-scalar corner theorem}, there is at most one space $\mc{V}_i$ of dimension  $2$ or greater. If such a space exists, we may reindex the sum $\bigoplus_{i=1}^m\mc{V}_i$ if necessary and assume that $\dim(\mc{V}_1)=k\geq 2$. Theorem~\ref{At most one non-scalar corner theorem} then implies that $\mc{V}_i$ is linked to $\mc{V}_j$ for all $i,j\geq 2$, so $\mc{B}=\mm_k\oplus \cc I_{n-k}$. If instead $\dim{\mc{V}}_i=1$ for all $i$, then Theorem~\ref{At most one non-scalar corner theorem} indicates that with at most one exception, all spaces $\mc{V}_i$ are mutually linked. Thus, $\mc{B}$ is equal to $\cc I$ or is unitarily equivalent to $\cc\oplus \cc I_{n-1}$.
	\end{proof}
	\smallskip
	
		Theorem~\ref{main result} can also be used to quickly identify the operators $T\in\mm_n$ such that $Alg(T,I)$---the unital algebra generated by $T$---is compressible.
	
		\begin{cor}\label{algebra generated by T and I is compressible}
	
		Let $n\geq 2$ be an integer, and let $T\in\mm_n$. The following are equivalent:	
		\begin{itemize}
			\item[(i)]$Alg(T,I)$ is compressible;
			
			\item[(ii)]$Alg(T,I)$ is the unitization of an $\mc{LR}$-algebra;
			
			\item[(iii)]$T\in span\{I,R\}$ for some $R\in\mm_n$ of rank $1$.
		\end{itemize}	
	
	\end{cor}

	\begin{proof}
	
		It is clear that (ii) implies (i). To see that (i) implies (iii), assume that $Alg(T,I)$ is compressible. It follows that $Alg(S^{-1}TS,I)=S^{-1}Alg(T,I)S$ is compressible for all invertible $S\in\mm_n$; hence we may assume that $T$ is in Jordan canonical form with respect to the standard basis $\left\{e_1,e_2,\ldots, e_n\right\}$ for $\cc^n$. 
		
		If $T$ has a Jordan block of size at least $3$, then $Alg(T,I)$ admits a principal compression of the form 
		$$\left\{\begin{bmatrix}
		x & y & z\\
		 0 & x & y\\
		 0 & 0 & x			
		\end{bmatrix}:x,y,z\in\cc\right\}.\smallskip $$ Since this algebra was shown to lack the compression property in \cite[Theorem~5.2.4]{CMR1}, it must be the case that each Jordan block of $T$ has size at most $2$. Note as well that if two or more Jordan blocks of size $2$ were present, then $Alg(T,I)$ would lack the compression property by Theorem~\ref{At most one non-scalar corner theorem}. Consequently, $T$ has at most one Jordan block of size $2$, and the remaining blocks have size $1$.
		
	If a Jordan block of size $2$ occurs, then $T$ cannot have two or more distinct eigenvalues.  Indeed, if $T$ had at least two distinct eigenvalues, then $Alg(T,I)$ would admit a principal compression that is unitarily equivalent to 
		$$\left\{\begin{bmatrix}
		x & y & 0\\
		0 & x & 0\\
		0 & 0 & z
		\end{bmatrix}: x,y,z\in\cc\right\}.$$
		By \cite[Theorem~5.2.2]{CMR1}, this algebra is not compressible---a contradiction. Thus, $T$ must be unitarily equivalent to $e_1\otimes e_2^*+\alpha I$ for some $\alpha\in\cc$. We conclude that $T=\alpha I+R$ for some $R$ in $\mm_n$ of rank $1$. 	

	Suppose now that every Jordan block of $T$ is $1\times 1$, so $T$ is diagonal. If $T$ had at least three distinct eigenvalues, then the algebra $\mc{D}$ of all $3\times 3$ diagonal matrices could be obtained as a principal compression of $Alg(T,I)$. Since  no algebra similar to $\mc{D}$ is projection compressible by \cite[Theorem~5.2.6]{CMR1}, this is not possible. Therefore, $T$ has at most two distinct eigenvalues. By Theorem~\ref{At most one non-scalar corner theorem}, one of the eigenvalues must have multiplicity $1$. We deduce that either $T$ has exactly one eigenvalue, and hence is a multiple of the identity; or $T$ has exactly two eigenvalues, and hence is a rank-one perturbation of a multiple of the identity. Thus, (iii) holds in this case as well.
	
	Finally, we will show that (iii) implies (ii). Suppose that $T\in\mathrm{span}\{I,R\}$ for some rank-one operator $R\in\mm_n$. That is, $T=\alpha I+\beta R$ for some $\alpha, \beta\in\cc$. If $\beta=0$, then $Alg(T,I)=\cc I$. Otherwise, $\beta R$ has rank $1$, and hence $Alg(T,I)=Alg(\beta R)+\cc I$ is the unitization of an $\mc{LR}$-algebra by \cite[Proposition~2.0.12]{CMR1}.
	\end{proof}
	
	It is interesting to note that in the $3$-dimensional case, the matrices of the form $\alpha I+\beta R$ for some $\alpha,\beta\in \cc$ and $R\in\mm_3$ of rank one are exactly those with two or more Jordan blocks corresponding to a common eigenvalue. Such matrices are said to be \textit{derogatory} \cite[Definition~1.4.4]{HornJohnson}. One may therefore view Corollary~\ref{algebra generated by T and I is compressible} as a higher-dimensional analogue of \cite[Corollary~5.1.3]{CMR1}.
	
	Throughout this exposition we have devoted our attention almost exclusively to unital subalgebras of $\mm_n$. Of course, it is reasonable to ask which non-unital algebras admit the projection or idempotent compression properties. In particular, it would be interesting to know whether or not the equivalence of these notions proven above in the unital case extends to non-unital algebras as well. 
	
By \cite[Proposition 2.0.6]{CMR1}, if a subalgebra $\mc{A}$ of $\mm_n$ admits the projection (resp. idempotent) compression property, then so too does its unitization. As a result, Theorem~\ref{main result} offers considerable insight into the non-unital projection (resp. idempotent) compressible algebras that exist in $\mm_n$. Specifically, this result indicates that if $\mc{A}$ is a projection compressible subalgebra of $\mm_n$, then $\widetilde{A}=\mc{A}+\cc I$ is the unitization of an $\mc{LR}$-algebra, or is transpose similar to one of the unital algebras from Example~\ref{exmp:families of compressible algebras}. Using this information, one can quickly obtain a non-unital analogue of Corollary~\ref{algebra generated by T and I is compressible}.

	\begin{cor}
	
		Let $n\geq 3$ be an integer, and let $T\in\mm_n$. The following are equivalent:
			\begin{itemize}
				\item[(i)]$Alg(T)$ is projection compressible;
				
				\item[(ii)]$Alg(T)$ is idempotent compressible;
				
				\item[(iii)]$Alg(T)$ is an $\mc{LR}$-algebra, or the unitization thereof;
				
				\item[(iv)]$T\in span\{I,R\}$ for some $R\in\mm_n$ of rank $1$, and $0$ does not occur as an eigenvalue of $T$ with algebraic multiplicity $1$.
			
			\end{itemize}
	
	\end{cor}
	
	\begin{proof}
	
		It is clear that (iii) implies (ii), and (ii) implies (i). 
		
		To see that (i) implies (iv), note that if $Alg(T)$ is projection compressible, then so too is $Alg(T,I)$. By Corollary~\ref{algebra generated by T and I is compressible}, there is a rank-one operator $R\in\mm_n$ such that $T\in\mathrm{span}\{I,R\}$. For the final claim, write $T=\alpha I+\beta R$ for some $\alpha,\beta\in\cc$, and suppose to the contrary that $\lambda=0$ is an eigenvalue of $T$ with algebraic multiplicity~$1$. Since $\rank(R)=1$, there is an orthonormal basis $\left\{e_1,e_2,\ldots, e_n\right\}$ for $\cc^n$ with respect to which $$\beta R=\gamma_{1}e_1\otimes e_1^*+\gamma_{2}e_1\otimes e_2^*\smallskip$$ for some constants $\gamma_1,\gamma_2\in\cc$.  Thus, when expressed as a matrix with respect to this basis, $T$ is upper triangular with diagonal entries $\alpha+\gamma_1$ with multiplicity~$1$, and $\alpha$ with multiplicity~$n-1$. It must therefore be the case that $\alpha+\gamma_{1}=0$ and $\alpha\neq 0$. 
		

		Let $P$ denote the orthogonal projection onto $\mathrm{span}\left\{e_1,e_2,e_3\right\}$ and define $T^\prime\coloneqq PTP$. With respect to the ordered basis $\mc{B}=\left\{e_1,e_2,e_3\right\}$ for $\ran(P)$,
		$$T^\prime=\begin{bmatrix}
		0 & \gamma_{2} & 0\\
		0 & \alpha & 0\\
		0 & 0 & \alpha
		\end{bmatrix}.$$
Thus, since $Alg(T)=\cc T$ is projection compressible, $P^\prime Alg(T)P^\prime=\cc P^\prime T^\prime P^\prime$ is an algebra for all projections $P^\prime\leq P$. Consider the matrix $$P^\prime =\begin{bmatrix}1 & 0 & 1\\
		0 & 2 & 0\\
		1 & 0 & 1\end{bmatrix},$$
		written with respect to the basis $\mc{B}$. It is easy to see that $\frac{1}{2}P^\prime $ is a subprojection of $P$. Moreover, it is straightforward to show that $\langle B e_2,e_2\rangle=4\langle B e_1,e_1\rangle$ for all $B\in P^\prime Alg(T)P^\prime$. One may verify, however, that  
		$$\langle (P^\prime TP^\prime )^2e_2,e_2\rangle-4\langle (P^\prime TP^\prime )^2e_1,e_1\rangle=8\alpha^2\neq 0,$$
		and thus $(P^\prime TP^\prime )^2\notin P^\prime Alg(T)P^\prime$. This is clearly a contradiction.
		

	It remains to show that (iv) implies (iii). To this end, let $T$ and $R$ be as in (iv), and write $T=\alpha I+\beta R$ for some $\alpha,\beta\in\cc$. Let $\left\{e_1,e_2,\ldots, e_n\right\}$ be an orthonormal basis for $\cc^n$ with respect to which $$\beta R=\gamma_{1}e_1\otimes e_1^*+\gamma_{2}e_1\otimes e_2^*\smallskip$$ for some $\gamma_1,\gamma_2\in\cc$. 
	
	First suppose that $\alpha=0$, so $Alg(T)=Alg(\beta R)$. If $\beta=0$ then this algebra is trivial. Otherwise, $Alg(T)$ is an $\mc{LR}$-algebra by \cite[Proposition 2.0.12]{CMR1}. If instead $\alpha\neq 0$, then our assumptions on $T$ imply that $\alpha+\gamma_1\neq 0$. Consequently, 
	$$I=\left(\frac{1}{\alpha}+\frac{1}{\alpha+\gamma_{1}}\right)T-\frac{1}{\alpha(\alpha+\gamma_{1})}T^2\in Alg(T).$$
	It follows that $Alg(T)=Alg(T,I)$, so $Alg(T)$ is the unitization of an $\mc{LR}$-algebra by Corollary~\ref{algebra generated by T and I is compressible}.
%
	\end{proof}		
	
	The notions of projection compressibility and idempotent compressibility can also be naturally extended to algebras of bounded linear operators acting on a Hilbert space $\mc{H}$ of arbitrary dimension. It would therefore be interesting to obtain analogues of the above results that apply in this setting. 
	
	One approach to understanding the structure of a projection (resp. idempotent) compressible operator algebra $\mc{A}$ would be to apply Theorem~\ref{main result} to the unital compressions $P\widetilde{A}P$, where $P$ is a projection (resp. idempotent) of finite rank. This technique may have its limts, however, as there could exist operator algebras $\mc{A}$ that lack the projection compression property, yet such that $P\mc{A}P$ is an algebra for all finite-rank projections $P$. With this in mind, the most viable avenue for understanding the compression properties in this setting may be to first obtain an intrinsic explanation as to why these notions coincide for unital subalgebras of $\mm_n$.\\

\section*{Acknowledgements}  The author would like to thank Laurent Marcoux and Heydar Radjavi for their insight and advice over the course of this research.\\

	\bibliography{ZCramerCompressibility2Bib}

\begin{thebibliography}{1}

\bibitem{Azoff}
E.A. Azoff.
\newblock {\em On finite rank operators and preannihilators}.
\newblock American Mathematical Soc., 1986.

\bibitem{Burnside}
W.~Burnside.
\newblock On the condition of reducibility of any group of linear substituions.
\newblock {\em Proceedings of the London Mathematical Society}, 2(1):430--434,
  1905.

\bibitem{CMR1}
Z.J. Cramer, L.W. Marcoux, and H.~Radjavi.
\newblock Matrix algebras with a certain compression property {I}.
\newblock {\em Linear Algebra and its Applications}, 621:50--85, 2021.

\bibitem{DMRTransitive}
K.R. Davidson, L.W. Marcoux, and H.~Radjavi.
\newblock Transitive spaces of operators.
\newblock {\em Integral Equations and Operator Theory}, 61(2):187--210, 2008.

\bibitem{HornJohnson}
R.A. Horn and C.R. Johnson.
\newblock {\em Matrix analysis}.
\newblock Cambridge University Press, 2013.

\bibitem{Lam}
T.Y. Lam.
\newblock {\em A first course in noncommutative rings}, volume 131.
\newblock Springer Science \& Business Media, 2013.

\bibitem{LMMRWedderburn}
L.~Livshits, G.W. MacDonald, L.W. Marcoux, and H.~Radjavi.
\newblock A spatial version of {W}edderburn's principal theorem.
\newblock {\em Linear and Multilinear Algebra}, 63(6):1216--1241, 2015.

\end{thebibliography}
	
	\bibliographystyle{plain}
	
	\medskip
	
	\Addresses
	
	\end{document}